\documentclass{article}
\usepackage{amssymb,amsfonts,amsmath,amsthm,amsopn,amstext,amscd,latexsym,xy,mathdots,stmaryrd,color}
\usepackage{hyperref}
\theoremstyle{plain}
\input xy
\xyoption{all}
\newtheorem{theorem}{Theorem}[section]
\newtheorem{lemma}[theorem]{Lemma}
\newtheorem{proposition}[theorem]{Proposition}
\newtheorem{corollary}[theorem]{Corollary}
\newtheorem{definition}[theorem]{Definition}
\newtheorem{conjecture}[theorem]{Conjecture}

\theoremstyle{remark}
\newtheorem{remark}[theorem]{Remark}
\numberwithin{equation}{section}
\numberwithin{paragraph}{section}
\newtheorem{conjecture*}[theorem]{Conjecture}

\DeclareMathOperator{\Frac}{Frac}
\DeclareMathOperator{\Hom}{Hom}

\DeclareMathOperator{\rank}{rank}
\DeclareMathOperator{\rec}{rec}

\DeclareMathOperator{\Iw}{Iw}
\DeclareMathOperator{\Gal}{Gal}
\DeclareMathOperator{\Aut}{Aut}

\DeclareMathOperator{\Supp}{Supp}
\DeclareMathOperator{\ad}{ad}

\DeclareMathOperator{\Lie}{Lie}

\DeclareMathOperator{\Ind}{Ind}

\DeclareMathOperator{\tr}{tr}

\DeclareMathOperator{\diag}{diag}

\DeclareMathOperator{\End}{End}

\DeclareMathOperator{\depth}{depth}
\DeclareMathOperator{\Frob}{Frob}
\DeclareMathOperator{\im}{im}
\DeclareMathOperator{\Spec}{Spec}

\renewcommand{\)}{\textup{)}}
\newcommand{\bl}{{\boldsymbol{\lambda}}}
\newcommand{\cA}{{\mathcal A}}
\newcommand{\cB}{{\mathcal B}}

\newcommand{\cD}{{\mathcal D}}

\newcommand{\cH}{{\mathcal H}}

\newcommand{\cJ}{{\mathcal J}}
\newcommand{\cK}{{\mathcal K}}
\newcommand{\cL}{{\mathcal L}}

\newcommand{\cO}{{\mathcal O}}

\newcommand{\cS}{{\mathcal S}}
\newcommand{\cT}{{\mathcal T}}

\newcommand{\fra}{{\mathfrak a}}

\newcommand{\frg}{{\mathfrak g}}

\newcommand{\ffrm}{{\mathfrak m}}
\newcommand{\frn}{{\mathfrak n}}

\newcommand{\bbA}{{\mathbb A}}

\newcommand{\bbC}{{\mathbb C}}

\newcommand{\bbF}{{\mathbb F}}
\newcommand{\bbG}{{\mathbb G}}

\newcommand{\bbQ}{{\mathbb Q}}
\newcommand{\bbR}{{\mathbb R}}

\newcommand{\bbT}{{\mathbb T}}

\newcommand{\bbZ}{{\mathbb Z}}

\newcommand{\GL}{\mathrm{GL}}

\newcommand{\PGL}{\mathrm{PGL}}

\newcommand{\CNL}{\mathrm{CNL}}

\newcommand{\Sets}{\mathrm{Sets}}

\newcommand{\plim}{\mathop{\varprojlim}\limits}

\newcommand{\Fil}{\mathrm{Fil}}

\newcommand{\gr}{\mathrm{gr}}
\newcommand{\Ann}{\mathrm{Ann}}
\newcommand{\Art}{\mathrm{Art}}
\newdir^{ (}{{}*!/-5pt/@^{(}}

\title{Potential automorphy and the Leopoldt conjecture}

\author{Chandrashekhar Khare and Jack A. Thorne}

\begin{document}

\maketitle
\begin{abstract}
We study in this paper Hida's  $p$-adic Hecke algebra for $\GL_n$ over a CM field $F$. Hida has made a conjecture about the dimension of  these Hecke algebras, which he calls  the non-abelian Leopoldt conjecture, and shown that his conjecture in the case $F = \bbQ$ implies the classical Leopoldt conjecture for a number field  $K$ of degree $n$ over $\bbQ$, if one assumes  further the existence of automorphic induction of characters for the extension $K/\bbQ$. 

We study Hida's conjecture using the automorphy lifting techniques adapted to the $\GL_n$ setting by Calegari--Geraghty. We prove an automorphy lifting result in this setting, conditional on existence and local-global compatibility of Galois representations arising from torsion classes in the cohomology of the corresponding symmetric manifolds. Under the same conditions  we show that one  can deduce the classical (abelian) Leopoldt conjectures  for a totally real number field $K$ and a prime $p$ using Hida's non-abelian Leopoldt conjecture for  $p$-adic Hecke algebra for $GL_n$ over CM fields without needing to assume automorphic induction of characters for the extension $K/\bbQ$. For this methods of potential automorphy results are used.  
\end{abstract}
\tableofcontents

\section{Introduction} 

Let $F$ be a number field, and let $p$ be a prime. If $U$ is a sufficiently small open compact subgroup of $\GL_n(\bbA_F^\infty)$, then the double quotient
\[ X_U = \GL_n(F) \backslash \GL_n(\bbA) / ( U \times K_\infty Z_\infty ) \]
has a natural structure of smooth manifold. (We write $G_\infty = \GL_n(F \otimes_\bbQ \bbR)$, $K_\infty \subset G_\infty$ for a choice of maximal compact subgroup, and $Z_\infty$ for the center of $G_\infty$.) This paper is an exploration of the ordinary part of the $p$-adic completed homology groups of the manifolds $X_U$. More precisely, if $U$ factors as $U = U^p U_p$, then for each integer $c \geq 1$, we define 
\[ U^p(c) = U^p \cdot \prod_{v | p} I_v(c, c), \]
where $I_v(c, c) \subset \GL_n(\cO_{F_v})$ is the subgroup of matrices which, modulo $\varpi_v^c$, are upper-triangular  with constant diagonal entry (see Equation (\ref{eqn_definition_of_iwahori})). Then the manifolds $X_{U^p(c)}$ fit into a tower
\[ \dots \to X_{U^p(c)} \to \dots \to X_{U^p(2)} \to X_{U^p(1)}, \]
and there is a corresponding tower of homology groups (which are finite $\bbZ_p$-modules):
\[ \dots \to H_\ast(X_{U^p(c)}, \bbZ_p) \to \dots \to H_\ast(X_{U^p(2)}, \bbZ_p) \to H_\ast(X_{U^p(1)}, \bbZ_p). \]
There is a Hecke operator $\mathbf{U}_p$ which acts compatibly on the whole tower of homology groups, and we define 
\[ H_\ast(X_{U^p(c)})_\text{ord} \subset H_\ast(X_{U^p(c)}) \]
to be the ordinary part of $H_\ast(X_{U^p(c)}, \bbZ_p)$ with respect to $\mathbf{U}_p$, i.e.\ the maximal direct summand $\bbZ_p$-module on which $\mathbf{U}_p$ acts invertibly. We then define
\[ H^\ast_\text{ord}(U) = \plim_c H_{d-\ast}(X_{U^p(c)}, \bbZ_p)_\text{ord}, \]
where $d = \dim X_U$. These are Hida's ordinary cohomology groups. They have a natural structure of finite $\Lambda$-module, where $\Lambda = \bbZ_p \llbracket (\cO_F \otimes \bbZ_p)^\times(p)^{n-1} \rrbracket$ is the completed group ring of the pro-$p$ part of the $p$-adic points of a maximal torus of $\PGL_n(F \otimes_\bbQ \bbQ_p)$.

A basic question is: what is the dimension of the finite $\Lambda$-module $H^\ast_\text{ord}(U)$? If $F$ is totally real and $n = 2$, then $X_U$ can be given the structure of Hermitian locally symmetric space, and one can show that $H^\ast_\text{ord}(U)$ is a faithful $\Lambda$-module. In general, however, one expects that it should have positive codimension. This question was first explored by Hida, who suggested that this codimension should be equal to the `defect'
\begin{equation}\label{eqn_intro_defect} l_0 = \rank G_\infty - \rank Z_\infty K_\infty. 
\end{equation}
Hida called this conjecture the `non-abelian Leopoldt conjecture', and gave various pieces of evidence for it \cite{Hid98}. In particular, he showed that this conjecture implies the classical Leopoldt conjecture over totally real fields, if one assumes in addition the existence of automorphic induction for Hecke characters.

In this paper, we explore Hida's conjecture using Galois deformation theory. The merit of this point of view is already discussed in \cite[Introduction]{Hid98}; conjecturally, one can identify Hecke algebras which act on the groups $H^\ast_\text{ord}(U)$ with suitable Galois deformation rings, and a calculation which goes back essentially to Mazur \cite{Maz89} gives the expected dimension of these deformation rings, which in turn implies the formula (\ref{eqn_intro_defect}) for the codimension of the cohomology as $\Lambda$-module. 

We take as our starting point the work of Calegari--Geraghty \cite{Cal13} showing how to make the Taylor--Wiles method work in the `positive defect' context, assuming the existence of Galois representations associated to torsion classes in the cohomology of the manifolds $X_U$. Using similar methods, and a similar assumption about the existence of Galois representations, we prove an $R = \bbT$ theorem for Hida's ordinary completed cohomology groups.

Assume now that $U = \prod_v U_v$ is a product. Let $S$ be a finite set of finite places of $F$, containing the places dividing $p$, and such that for each $v \not\in S$, $U_v = \GL_n(\cO_{F_v})$. Then the unramified Hecke algebra $\bbT^{S, \text{univ}} = \Lambda [ \{ T_v^i \}_{v \not\in S}]$ acts on the groups $H^\ast_\text{ord}(U)$, and we write $\bbT^S_\text{ord}(U)$ for the quotient which acts faithfully. It is a finite $\Lambda$-algebra. One expects that for each maximal ideal $\ffrm \subset \bbT^S_\text{ord}(U)$, there should exist a continuous semi-simple representation
\[ \overline{\rho}_\ffrm : G_{F, S} \to \GL_n(\bbT^S_\text{ord}(U)/\ffrm) \]
which is characterized uniquely up to isomorphism by a condition on the characteristic polynomials of Frobenius elements at finite places $v \not\in S$; this generalizes the well-known relation $\tr \rho(\Frob_p) = a_p(f)$ satisfied by the Galois representations associated to elliptic modular forms by Deligne. If $\ffrm$ is such a maximal ideal and $\overline{\rho}_\ffrm$ is absolutely irreducible, then we say that $\ffrm$ is non-Eisenstein. We in fact study the localized cohomology $H^\ast_\text{ord}(U)_\ffrm$, which is finite over $\Lambda$ and a faithful $\bbT^S_\text{ord}(U)_\ffrm$-module. (If $F$ is an imaginary CM or totally real field, then the existence of $\overline{\rho}_\ffrm$ can be deduced from recent work of Scholze \cite{Sch13}.)

We can now state a more precise version of Hida's conjecture:
\begin{conjecture}\label{conj_intro_hida_conjecture}
Let $\ffrm \subset \bbT^S_\text{ord}(U)$ be a non-Eisenstein maximal ideal. Then $\dim_\Lambda H^\ast_\text{ord}(U)_\ffrm = \dim \Lambda - l_0$.
\end{conjecture}
 We think of this conjecture as a `non-abelian analogue' of the classical Leopoldt conjecture. It is worth remarking that in the case $l_0 > 0$, one always has the trivial bound $\dim_\Lambda H^\ast_\text{ord}(U)_\ffrm < \dim \Lambda$ (Corollary \ref{cor_cohomology_is_torsion_lambda_module}). 

We can now state our main result.
\begin{theorem}\label{thm_intro_nal_implies_l}
Let $K/\bbQ$ be a Galois totally real number field with $[K : \bbQ] = n$ and $K \cap \bbQ(\zeta_p) = \bbQ$, and suppose that $p > n$. Assume Conjecture \ref{conj_intro_hida_conjecture} for all imaginary CM extensions $E/\bbQ$, together with Conjectures \ref{conj_existence_of_galois_combined} and \ref{conj_local_global_compatibility_at_taylor_wiles_primes} below. Then Leopoldt's conjecture holds for the pair $(K, p)$: in other words, the only $\bbZ_p$-extension of $K$ unramified outside $p$ is the cyclotomic one.
\end{theorem}
More informally, the non-abelian Leopoldt conjecture of Hida implies the abelian Leopoldt conjecture. We do not attempt to state precisely Conjectures \ref{conj_existence_of_galois_combined} and \ref{conj_local_global_compatibility_at_taylor_wiles_primes} here, noting only that they assert the existence of suitable Galois representations with coefficients $\bbT^S_\text{ord}(U)_\ffrm$, as well as a kind of local-global compatibility, both at places dividing $p$ (at which the Galois representations should be ordinary in the usual sense) and at Taylor--Wiles places. 

Our strategy for proving Theorem \ref{thm_intro_nal_implies_l} is to prove an $R = \bbT$ type result over general fields $F$, using the techniques of Calegari--Geraghty (Theorem \ref{thm_r_equals_t}). This theorem uses Conjectures \ref{conj_existence_of_galois_combined} and \ref{conj_local_global_compatibility_at_taylor_wiles_primes} in an essential way, but does not use Conjecture \ref{conj_intro_hida_conjecture}. To get information about the Leopoldt conjecture over $K$, we choose a character $\overline{\chi} : G_K \to \overline{\bbF}_p^\times$ and look at $\overline{\rho} = \Ind_K^\bbQ \overline{\chi}$. For well-chosen $\overline{\chi}$, the main results of \cite{Bar14} imply that $\overline{\rho}|_{G_E}$ becomes automorphic for some choice of CM extension $E/K$. (In other words, $\overline{\rho}$ is potentially automorphic.) Deformations of the character $\overline{\chi}$ give rise to deformations of $\overline{\rho}$. Our $R = \bbT$ type result then shows that Conjecture \ref{conj_intro_hida_conjecture} implies a bound on the dimension of the associated deformation ring of $\overline{\rho}$, and this in turns gives a bound on the deformation space of the character $\overline{\chi}$, implying the Leopoldt conjecture for the pair $(K, p)$.

The proof of our $R = \bbT$ type result includes two new observations that we mention here. The first is that one can construct Taylor--Wiles systems in this context using Iwahori level subgroups, if the image of the residual representation $\overline{\rho}_\ffrm$ is sufficiently large (namely, if $\overline{\rho}_\ffrm$ satisfies the conditions of Definition \ref{def_huge_image} below and therefore has what we call `enormous image'). The important point here is to show that the cohomology groups behave well during passage from `level $U$' to `level $U_0(Q)$'; see \S \ref{sec_auxiliary_primes}. The second is that one can get information about the ring $R[1/p]$ without assuming that the cohomology groups $H^\ast_\text{ord}(U)$ vanish outside of degrees in the so-called `middle range' $[q_0, q_0 + l_0]$; for our purposes, it is sufficient to know this after inverting $p$, and we can prove this unconditionally in certain cases (see Theorem \ref{thm_only_cuspidal_cohomology_survives_localization}).

We now describe in more detail the organization of this paper. Section \ref{sec_commutative_algebra} is of a preliminary nature. We find it convenient to work mostly in the derived category $\mathbf{D}(\Lambda)$ of $\Lambda$-modules, and an important point is that if $C^\bullet$ is a bounded above, perfect complex of $\Lambda$-modules and $t \in \End_{\mathbf{D}(\Lambda)}(C^\bullet)$ is an endomorphism, then it makes sense to form the ordinary part of $C^\bullet$ with respect to $t$ in $\mathbf{D}(\Lambda)$ (and not just at the level of cohomology groups). In Section \ref{sec_abstract_patching}, we prove a version of the patching criterion of Calegari--Geraghty. Since we eventually view our Hecke algebras as rings of endomorphisms in $\mathbf{D}(\Lambda)$, some modifications are required.

In Section \ref{sec_galois_deformation_theory}, we recall the basics of Galois deformation theory for representations with fixed determinant. The only new material here is the definition of a representation with `enormous image', and some remarks on how to construct such representations by inducing characters from finite extensions. In Section \ref{sec_q_adic_hecke_algebras}, we make some remarks about the Iwahori-Hecke algebra of $\GL_n(F_v)$ in characteristic $p$, when $q_v \equiv 1 \text{ mod }p$. These remarks are useful when we need to study the properties of the ordinary completed cohomology groups of Section \ref{sec_ordinary_completed_cohomology} with respect to the Taylor--Wiles systems constructed in Section \ref{sec_galois_deformation_theory}.

In Section \ref{sec_ordinary_completed_cohomology} we come to the manifolds $X_U$ and their ordinary completed cohomology groups. We first reprove a number of results of Hida in this context. Since we again wish to work in  $\mathbf{D}(\Lambda)$, we prove derived versions of Hida's control theorem and Hida's independence of weight theorem. We then prove our main $R = \bbT$ type result, conditionally on our conjectures about existence and local-global compatibility of Galois representations with coefficients in $\bbT^S_\text{ord}(U)_\ffrm$. Finally, in Section \ref{sec_potential_automorphy_and_leopoldt}, we use these results to prove our main result, Theorem \ref{thm_intro_nal_implies_l} above.

\subsection{Acknowledgments}

We thank Toby Gee, Fabian Januszewski, and James Newton for useful comments. In the period during which this research was conducted, Jack Thorne served as a Clay Research Fellow.
Chandrashekhar Khare was supported by NSF grant DMS-1161671 and by a Humboldt Research Award, and thanks the Tata Instititute of Fundamental Research, Mumbai for hospitality during the period in which  most of the work on this paper was done.

\subsection{Notation}

A base number field $F$ having been fixed, we will make the following additional choices. We fix an algebraic closure $\overline{F}$ of $F$, and for each place $v$ of $F$ algebraic closures $\overline{F}_v$ of the completion $F_v$ at $v$. We also fix embeddings $\overline{F} \hookrightarrow \overline{F}_v$ extending the natural embeddings $F \hookrightarrow F_v$. Then the absolute Galois group $G_F = \Gal(\overline{F}/F)$ of $F$ is defined, and we get embeddings of decomposition groups $G_{F_v} \hookrightarrow G_F$, by restriction.  If $v$ is a finite place of $F$, then we write $I_{F_v} \subset G_{F_v}$ for the inertia group, and $\Frob_v \in G_{F_v}/I_{F_v}$ for a \emph{geometric} Frobenius element. We write $\bbA_F$ for the adele ring of $F$. If $F$ is an imaginary CM field, then we will use the notation $F^+$ to denote the maximal totally real subfield of $F$, and $\delta_{F/F^+} : \Gal(F/F^+) \to \{ \pm 1 \}$ for the unique non-trivial character. The local Artin maps $\Art_{F_v}$ are normalized to send uniformizers to geometric Frobenius elements.

Let $p$ be a prime. We will generally fix an algebraic closure $\overline{\bbQ}_p$ of $\bbQ_p$. We refer to a finite extension $E/\bbQ_p$ inside $\overline{\bbQ}_p$ as a coefficient field, and write $\cO$ for its ring of integers, $\lambda \subset \cO$ for its maximal ideal, and $k = \cO/\lambda$ for its residue field. If $R$ is a complete Noetherian local $\cO$-algebra with residue field $k$, then we write $\mathrm{CNL}_R$ for the category of complete Noetherian local $R$-algebras with residue field $k$. If $A$ is a finite $\widehat{\bbZ}$-module (equivalently: a topologically finitely generated abelian profinite group), then we write $A(p) = A \otimes_{\widehat{\bbZ}} \bbZ_p$ for its maximal pro-$p$ quotient. If $R$ is any local ring, then we write $\ffrm_R$ for its maximal ideal.

If $F$ is a number field, then we write $\epsilon : G_F \to \bbZ_p^\times$ for the $p$-adic cyclotomic character. The prime $p$ will always be clear from the context, so we omit it from the notation. If $v$ is a place of $F$ dividing $p$, $\tau : F_v \hookrightarrow \overline{\bbQ}_p$ is a continuous embedding, and $\rho : G_{F_v} \to \GL_n(\overline{\bbQ}_p)$ is a continuous representation which is de Rham, then we write $\mathrm{HT}_\tau(\rho)$ for the multiset of $\tau$-Hodge--Tate weights of $\rho$, taken with multiplicity. We use the normalization with $\mathrm{HT}_\tau(\epsilon) = -1$. Similarly, if $\rho : G_F \to \GL_n(\overline{\bbQ}_p)$ is a continuous representation and $\rho|_{G_{F_v}}$ is de Rham for each place $v|p$ of $F$, then for each embedding $\tau : F \hookrightarrow \overline{\bbQ}_p$ we define  $\mathrm{HT}_\tau(\rho) = \mathrm{HT}_\tau(\rho|_{G_{F_v}})$, where $v$ is the $p$-adic place of $F$ induced by $\tau$.

If $F$ is a number field and $\pi$ is a cuspidal automorphic representation of $\GL_n(\bbA_F)$, we say that $\pi$ is regular algebraic if it satisfies the integrality condition on the Langlands parameter of $\pi_\infty$ given in \cite[Definition 2.1]{Clo13}. If $\iota : \overline{\bbQ}_p \cong \bbC$ is an isomorphism, we say that $\pi$ is $\iota$-ordinary if it is regular algebraic and further satisfies the condition of \cite[Definition 5.1.2]{Ger09}. If $\rho : G_F \to \GL_n(\overline{\bbQ}_p)$ is a continuous representation, unramified at all but finitely many places of $F$, we say that $\rho$ is automorphic if there exists a cuspidal, regular algebraic automorphic representation $\Pi$ of $\GL_n(\bbA_F)$ such that $\mathrm{WD}(\rho|_{G_{F_v}})^\text{F-ss} \cong \rec_{F_v}^T(\iota\Pi_v)$ at all but finitely many finite places $v$ of $F$ at which both $\rho$ and $\Pi$ are unramified. (Here we use $\mathrm{WD}$ to denote the associated Weil--Deligne representation of $\rho$, and $\rec_{F_v}^T$ to denote the arithmetically normalized local Langlands correspondence for $\GL_n(F_v)$; see \cite[\S 2.1]{Clo13} for more details.) 

If $E$ is a coefficient field with residue field $k$ and $M$ is a discrete $k[G_F]$-module, finite-dimensional as $k$-vector space, then we write $H^i(F, M) = H^i(G_F, M)$ for the continuous group cohomology groups. We write $h^i(F, M) = \dim_k H^i(F, M)$ and $\chi(F, M) = \sum_{i = 0}^\infty (-1)^i h^i(F, M)$, when these make sense. We write $F(M)$ for the extension cut out by $M$ (i.e.\ the fixed field inside $\overline{F}$ of the group $\mathrm{ker}( G_F \to \Aut_k(M))$). We use similar notation when $M$ is a discrete $k[G_{F_v}]$-module of finite dimension over $k$. 

\section{Preliminaries in commutative algebra}\label{sec_commutative_algebra}

\subsection{Change of coefficients}

Let $R$ be a ring, and let $C^\bullet$ be a complex of $R$-modules. (With some exceptions, we will usually consider cochain complexes $C^\bullet$ rather than chain complexes $C_\bullet$. All rings will be commutative.) We say that $C^\bullet$ is bounded if $C^i = 0$ for all but finitely many $i$. If $C^\bullet, D^\bullet$ are complexes of $R$-modules, then we write $\Hom_R(C^\bullet, D^\bullet)$ for the set of morphisms $f : C^\bullet \to D^\bullet$ of complexes, i.e.\ collections $f = (f^i)_{i \in \bbZ}$ of morphisms $f^i : C^i \to D^i$ such that $d f^i = f^{i+1} d$ for each $i \in \bbZ$. We write $\End_R(C^\bullet) = \Hom_R(C^\bullet, C^\bullet)$. If $n$ is an integer and $C^\bullet$ is a complex of $R$-modules, then we define its truncations
\[ (\tau_{\leq n} C^\bullet)^i = \left\{ \begin{array}{cc} C^i & i \leq n \\ \ker d_i & i = n \\ 0 & i > n, \end{array} \right. \]
and $\tau_{> n} C^\bullet = C^\bullet / \tau_{\leq n} C^\bullet$.
There is a natural morphism of complexes $\tau_{\leq n} C^\bullet \to C^\bullet$. If $i > n$ then $H^i(\tau_{\leq n} C^\bullet) = 0$, while if $i \leq n$ then this morphism induces an isomorphism $H^i(\tau_{\leq n} C^\bullet) \cong H^i(C^\bullet)$. Similar remarks apply to $\tau_{> n} C^\bullet$.

Let $M$ be an $R$-module, and let $\mathbf{x} = (x_1, \dots, x_n)$ be an ordered tuple of elements of $R$. We say that the sequence $\mathbf{x}$ is $M$-regular if it satisfies the following two conditions:
\begin{enumerate}
\item For each $i = 1, \dots, n$, the element $x_i$ is not a zero-divisor on $M/(x_1, \dots, x_{i-1})$.
\item The module $M/(x_1, \dots, x_n)$ is not 0.
\end{enumerate}
As is well-known, if $R$ is a Noetherian local ring and $\mathbf{x}$ is an $M$-regular sequence, then any permutation of $\mathbf{x}$ is also an $M$-regular sequence.
\begin{proposition}\label{prop_change_of_coefficients_spectral_sequence}
Let $R$ be a ring, $M$ an $R$-module, and let $C^\bullet$ be a bounded above complex of free $R$-modules. Then there is a spectral sequence:
\begin{gather}\label{eqn_change_of_coefficient_spectral_sequence} 
E_2^{p, q} = \mathrm{Tor}_{-p}^R( H^q(C^\bullet), M) \Rightarrow H^{p+q}(C^\bullet \otimes_R M) \\ d_2 : E_2^{p, q} \to E_2^{p+2, q-1}. 
\end{gather}
\end{proposition}
\begin{proof}
See \cite[Theorem 5.6.4]{Wei94}.
\end{proof}
\begin{corollary}\label{cor_regular_sequence_at_infinity}
Let $R$ be a Noetherian local ring, and let $C^\bullet$ be a bounded complex of finite free $R$-modules, concentrated in degrees $[a, b]$. Let $\mathbf{x} = (x_1, \dots, x_n)$ be an $R$-regular sequence. Suppose that $H^i(C^\bullet \otimes_R R/(\mathbf{x})) = 0$ if $i \neq b$. Then $H^i(C^\bullet) = 0$ if $i \neq b$, and there is a canonical isomorphism
\begin{equation}\label{eqn_isomorphism_in_degree_zero} H^b(C^\bullet) \otimes_R R/(\mathbf{x}) \cong H^b(C^\bullet \otimes_R R/(\mathbf{x})). 
\end{equation}
Moreover, if $H^b(C^\bullet \otimes_R R/(\mathbf{x})) \neq 0$, then $\mathbf{x}$ is a $H^b(C^\bullet)$-regular sequence.
\end{corollary}
\begin{proof}
After shifting, we can assume $b = 0$. By induction, it is enough to treat the case $n = 1$. The spectral sequence (\ref{eqn_change_of_coefficient_spectral_sequence}) then degenerates, and we obtain short exact sequences:
\[ 0 \to H^{q}(C^\bullet) \otimes_R R/(x_1) \to H^{q}(C^\bullet \otimes_R R/(x_1)) \to \mathrm{Tor}_1^R(H^{q+1}(C^\bullet), R/(x_1)) \to 0. \]
If $q  \neq 0$, then the middle term is 0, and hence $H^q(C^\bullet)$ is 0, by Nakayama's lemma. Looking at this sequence when $q = 0$, we obtain the isomorphism (\ref{eqn_isomorphism_in_degree_zero}). If $H^0(C^\bullet \otimes_R R / (x_1)) \neq 0$, then this isomorphism shows that $H^0(C^\bullet) \neq 0$. Looking at the short exact sequence when $q = -1$, we find that multiplication by $x_1$ is injective on $H^0(C^\bullet)$, which shows that if $H^0(C^\bullet \otimes_R R / (x_1)) \neq 0$ then the sequence $\mathbf{x}$ is $H^0(C^\bullet)$-regular, as required.
\end{proof}

\subsection{Minimal complexes and the derived category}

Let $R$ be a Noetherian local ring, and let $C^\bullet,$ $D^\bullet$ be complexes of $R$-modules. We recall that a morphism of complexes $f^\bullet : C^\bullet \to D^\bullet$ is said to be a quasi-isomorphism if the induced maps $H^i(C^\bullet) \to H^i(D^\bullet)$ are all isomorphisms. Two morphisms $f, g : C^\bullet \to D^\bullet$ are said to be homotopy equivalent if there is a graded morphism of $R$-modules $s : C^\bullet \to D^\bullet[1]$ (not necessarily a morphism of complexes) such that $f - g = sd + ds$. In this case, $f$ and $g$ induce the same maps on cohomology.

We say that a complex $C^\bullet$ is good if it is bounded, and each $C^i$ is a projective finite $R$-module. (Thus the perfect complexes are, by definition, exactly the ones which are quasi-isomorphic to a good complex.) We say that the complex $C^\bullet$ is minimal if it is good and the differentials in $C^\bullet \otimes_R R/\ffrm_R$ are all 0. 
\begin{lemma}\label{lem_perfect_and_minimal_complexes}
Let $R$ be a Noetherian local ring, and let $C^\bullet$ be a complex of $R$-modules. 
\begin{enumerate}
\item Suppose that $C^\bullet$ is good. Then there exists a minimal complex $F^\bullet$ and a quasi-isomorphism $f : F^\bullet \to C^\bullet$. 
\item Suppose that $C^\bullet$ is good (resp. minimal). Then for any proper ideal $I \subset R$, $C^\bullet \otimes_R R/I$ is a good (resp. minimal) complex of $R/I$-modules.
\end{enumerate}
\end{lemma} 
\begin{proof}
For the first part, suppose that the differential $d_i : C^i \to C^{i+1}$ satisfies $d_i \otimes_R R/\ffrm_R \neq 0$. After choosing isomorphisms $C^i \cong R^m$, $C^{i+1} \cong R^{n}$, we can assume that $d_i$ is given by a matrix
\[ d_i = \left( \begin{array}{cccc} 1 & 0 & \dots & 0 \\
0 & \ast & \dots & \ast \\
\vdots & \vdots & & \vdots \\
0 & \ast & \dots & \ast \end{array} \right). \]
Let $D^\bullet$ be the complex defined by $D^j = C^j$ if $j \not\in \{ i, i+1 \}$, $D^i = 0 \oplus R^{m-1} \subset R^m$, and $D^{i+1} = 0 \oplus R^{n-1} \subset R^n$. It is clear from the above matrix that the differentials of $C^\bullet$ leave $D^\bullet$ invariant, and the inclusion $D^\bullet \to C^\bullet$ is a quasi-isomorphism. The result now follows by induction. The second part of the lemma is immediate. This completes the proof.
\end{proof}
We also make the following definitions.
\begin{itemize}
\item $\mathbf{K}(R)$ is the category of complexes of $R$-modules, with morphisms taken modulo homotopy equivalence. $\mathbf{K}(R)^-$ is the full subcategory of $\mathbf{K}(R)$ consisting of complexes which are bounded above. 
\item $\mathbf{D}(R)$ is the derived category of complexes of $R$-modules, defined by formally inverting all quasi-isomorphisms in $\mathbf{K}(R)$,  and $\mathbf{D}(R)^-$ is its full subcategory consisting of complexes which are bounded above. (See \cite[Definition 10.3.1]{Wei94}.) Equivalently, $\mathbf{D}(R)^-$ is the localization of $\mathbf{K}(R)^-$ at the class of quasi-isomorphisms. (See \cite[Example 10.3.15]{Wei94}.)
\item $\mathbf{K}(R)^{-, \text{proj}}$ is the full subcategory of $\mathbf{K}(R)^-$ whose objects are the complexes of projective $R$-modules.
\end{itemize}
The truncation functors $\tau_{\leq n}$ and $\tau_{> n}$ preserve quasi-isomorphisms, so induce functors $\tau_{\leq n} : \mathbf{D}(R) \to \mathbf{D}(R)$ and $\tau_{> n} : \mathbf{D}(R) \to \mathbf{D}(R)$.
\begin{proposition}\label{prop_derived_category_and_projective_complexes}
Let $R$ be a Noetherian local ring.
\begin{enumerate}
\item Let $C^\bullet \in \mathbf{K}(R)^{-, \text{proj}}$ and $D^\bullet \in \mathbf{D}(R)$. Then $\Hom_{\mathbf{K}(R)}(C^\bullet, D^\bullet) = \Hom_{\mathbf{D}(R)}(C^\bullet, D^\bullet)$.
\item The functor $\mathbf{K}(R)^{-, \text{proj}} \to \mathbf{D}(R)^-$ is an equivalence of categories.
\item Let $C^\bullet, D^\bullet \in \mathbf{K}(R)^{-, \text{proj}}$, and let $f : C^\bullet \to D^\bullet$ be a quasi-isomorphism. Then there exists a quasi-isomorphism $g : D^\bullet \to C^\bullet$ such that each of $fg$, $gf$ is homotopy equivalent to the identity.
\end{enumerate}
\end{proposition}
\begin{proof}
The first part is \cite[Corollary 10.4.7]{Wei94}. The second part is \cite[Theorem 10.4.8]{Wei94}. The third part follows immediately from the second.
\end{proof}
\begin{lemma}\label{lem_nilpotents_in_derived_category}
Let $R$ be a ring.
\begin{enumerate}
\item Let $A^\bullet \to B^\bullet \to C^\bullet \to A^\bullet[1]$ be an exact triangle in $\mathbf{D}(R)$, and suppose given a morphism of exact triangles
\[ \xymatrix{ A^\bullet \ar[d]^{f_A} \ar[r] & B^\bullet \ar[d]^{f_B} \ar[r] & C^\bullet \ar[d]^{f_C} \ar[r] & A^\bullet[1] \ar[d]^{f_A[1]} \\
A^\bullet \ar[r] & B^\bullet  \ar[r] & C^\bullet \ar[r] & A^\bullet[1].} \]
Suppose that there are integers $n_A, n_C \geq 1$ such that $f_A^{n_A} = 0$ and $f_C^{n_C} = 0$. Then $f_B^{n_A + n_C} = 0$.
\item Let $C^\bullet$ be a complex of $R$-modules which is concentrated in degrees $[0, d]$. Let $f \in \End_{\mathbf{D}(R)}(C^\bullet)$, and suppose that $H^\ast(f) = 0$. Then $f^{d+1} = 0$ in $\End_{\mathbf{D}(R)}(C^\bullet)$.
\end{enumerate} 
\end{lemma}
\begin{proof}
For the first part, we observe that there is an exact sequence
\[ \xymatrix@1{ \Hom_{\mathbf{D}(R)}(B^\bullet, A^\bullet) \ar[r] & \Hom_{\mathbf{D}(R)}(B^\bullet, B^\bullet) \ar[r] & \Hom_{\mathbf{D}(R)}(B^\bullet, C^\bullet).} \]
The image of $f_B^{n_C}$ in the final term is 0, so $f_B^{n_C}$ lifts to $g \in \Hom_{\mathbf{D}(R)}(B^\bullet, A^\bullet)$. This implies that $f_B^{n_A + n_C}$, which is the image of $f_A^{n_A} g$, is also 0.

For the second part, we show by induction on $i \geq 0$ that $f^{i+1} = 0$ on the truncated complex $\tau_{\leq i}C^\bullet$. The case $i = 0$ is immediate (since $\tau_{\leq 0} C^\bullet \cong H^0(C^\bullet)$ in $\mathbf{D}(R)$). In general, we have an exact triangle in $\mathbf{D}(R)$:
\[ \xymatrix@1{ \tau_{\leq i} C^\bullet \ar[r] & \tau_{\leq i+1} C^\bullet \ar[r] & H^{i+1}(C^\bullet)[-(i+1)] \ar[r] & \tau_{\leq i} C^\bullet [1]. } \]
The result therefore follows from the first part of the lemma.
\end{proof}

\begin{lemma}\label{lem_quasi_isomorphism_of_perfect_complexes}
Let $R$ be a Noetherian local ring.
\begin{enumerate} \item Let $C^\bullet$ be a complex of projective $R$-modules concentrated in degrees $[a, b]$ such that $H^\ast(C^\bullet)$ is a finitely generated $R$-module. Then there exists a good complex $D^\bullet$ of $R$-modules concentrated in degrees $[a, b]$ and a quasi-isomorphism $f : D^\bullet \to C^\bullet$.
\item Let $f : C^\bullet \to D^\bullet$ be a quasi-isomorphism of bounded above complexes of projective $R$-modules. Then for any $R$-algebra $S$, the map $f \otimes 1 : C^\bullet \otimes_R S \to D^\bullet \otimes_R S$ is a quasi-isomorphism.
\end{enumerate}
\end{lemma}
\begin{proof}
For the first part, see \cite[Ch. II, \S 5, Lemma 1]{Mum08}. For the second part, use Proposition \ref{prop_change_of_coefficients_spectral_sequence} and \cite[Lemma 5.2.4]{Wei94}.
\end{proof}
\begin{corollary}
Let $R$ be a Noetherian local ring, and let $C^\bullet$ be a good complex of $R$-modules. Let $f : F^\bullet \to C^\bullet$ be a quasi-isomorphism, where $F^\bullet$ is a minimal complex. Then there are isomorphisms for each $i \in \bbZ$:
\[ F^i/(\ffrm_R) \cong H^i(C^\bullet \otimes_R R/\ffrm_R). \]
In particular, $F^i \neq 0$ if and only if $H^i(C^\bullet \otimes_R R/\ffrm_R) \neq 0$.
\end{corollary}
\subsection{Dimension}

Let $R$ be a Noetherian ring, and $M$ a finite $R$-module. We write $\dim_R M$ for the Krull dimension of the quotient of $R$ which acts faithfully on $M$, i.e.\ the Krull dimension of $R/\mathrm{Ann}_R(M)$. We recall that a sequence of elements $x_1, \dots, x_d$ of the maximal ideal of $R$ is said to be a system of parameters of $M$ if $d = \dim M$ and $M / (x_1, \dots, x_d) M$ has finite length as an $R$-module (see \cite[Appendix]{Bru93}).
\begin{lemma}\label{lem_dimension_theory} Let $R$ be a Noetherian local ring.
\begin{enumerate}
\item Let $\mathbf{x} = (x_1, \dots, x_q)$ be a sequence of elements in $R$. Then $\dim_R M \geq \dim_{R/(\mathbf{x})} M/(\mathbf{x}) + q$, with equality if and only if $(x_1, \dots, x_q)$ is part of a system of parameters of $M$.
\item Let $\mathbf{x} = (x_1, \dots, x_q)$ be an $M$-regular sequence in $R$. Then $\dim_R M = \dim_{R/(\mathbf{x})} M/(\mathbf{x}) + q$ and $\depth_R M = \depth_{R/(\mathbf{x})} M/(\mathbf{x}) + q$.
\item Let $R \to S$ be a local homomorphism of Noetherian local rings, and let $M$ be an $S$-module which is finite as an $R$-module. Then $\dim_R M = \dim_S M$.
\end{enumerate}
\end{lemma}
\begin{proof}
The first part is a well-known consequence of dimension theory; see \cite[Proposition A.4]{Bru93}. The second part follows from the first and the two facts that every $M$-regular sequence is part of a system of parameters (\cite[Proposition 1.2.12]{Bru93}) and that every maximal $M$-regular sequence has the same length (\cite[Theorem 1.2.5]{Bru93}).

For the third part of the lemma, we set $I = \Ann_R M$, $J = \Ann_S M$, and must show $\dim R/I = \dim S/J$. The natural map $R/I \to S/J$ is injective, as we in fact have $R/I \hookrightarrow S/J \hookrightarrow \End_R(M)$. Moreover, $\End_R(M)$ is a finite $R$-module, which shows that $S/J$ is a finite $R$-module, hence a finite $R$-algebra. It follows that $S/J$ is a finite $R/I$-algebra, so $\dim R/I = \dim S/J$, as required.
\end{proof}
\begin{lemma}\label{lem_dimension_criterion_for_exactness}
Let $l_0 \geq 0$, $q_0$ be integers, and let $S$ be a Cohen--Macaulay local ring of dimension $n \geq l_0$. Let $C^\bullet$ be a good complex of $S$-modules. Suppose that groups $H^i(C^\bullet \otimes_S S/\ffrm_S)$ are non-zero only if $i$ lies in the range $[q_0, q_0 + l_0]$. Then $\dim_S H^\ast(C^\bullet) \geq \dim S - l_0$. If equality holds, then $H^i(C^\bullet)$ is non-zero only if $i = q_0 + l_0$, and $H^{q_0+l_0}(C^\bullet)$ has projective dimension $l_0$.
\end{lemma}
\begin{proof}
We follow the proof of \cite[Lemma 3.2]{Cal13}. After replacing $C^\bullet$ be a quasi-isomorphic complex, we can assume that $C^\bullet$ is minimal, and in particular $C^i \neq 0$ only if $i \in [q_0, q_0 + l_0]$. After shifting, we can assume that $q_0 = 0$. Let $m \geq 0$ denote the smallest integer such that $H^m(C^\bullet) \neq 0$, and set $K^m = C^m / d C^{m-1}$. Then the complex
\[ C^0 \to C^1 \to \dots \to C^m \]
is a projective resolution of $K^m$, which therefore has projective dimension at most $m$. On the other hand, we have $H^m(C^\bullet) \subset K^m$, and hence
\[ \dim_S H^m(C^\bullet) \geq \depth_S K^m = \dim S - \text{ proj dim }K^m \geq \dim S - m, \]
by the Auslander-Buchsbaum formula (and since $S$ is assumed Cohen--Macaulay). Since $m \leq l_0$, this shows the first assertion. If $\dim_S H^\ast(C^\bullet) = \dim S - l_0$, then we have
\[ \dim S - l_0 = \dim_S H^\ast(C^\bullet) \geq \dim_S H^m(C^\bullet) \geq \dim S - m, \]
from which it follows that $m = l_0$ and $C^\bullet$ is a projective resolution of $H^{l_0}(C^\bullet)$ of length $l_0$. This completes the proof.
\end{proof}

\subsection{Ordinary part}\label{sec_ordinary_parts}

Let $p$ be a prime and let $E$ be a finite extension of $\bbQ_p$ with ring of integers $\cO$ and residue field $k$. Let $R$ be a complete Noetherian local $\cO$-algebra with residue field $k$. 
\begin{lemma}\label{lem_ordinary_part_for_modules} 
\begin{enumerate} \item Let $M$ be a finite $R$-module, and let $T \in \End_R(M)$.  There is a unique $T$-invariant decomposition $M = M_\text{ord} \oplus M_\text{non-ord}$ with the following property: $T$ is invertible on $M_\text{ord}$, and topologically nilpotent (for the $\ffrm_R$-adic topology) on $M_\text{non-ord}$. Moreover, the limit $e = \lim_{n \to \infty} T^{n!}$ exists in $\End_R(M)$, is an idempotent, and we have $M_\text{ord} = e M$, $M_\text{non-ord} = (1-e)M$.
\item Let $M, N$ be finite $R$-modules, let $T_1 \in \End_R(M)$, $T_2 \in \End_R(M)$, and let $f \in \Hom_R(M, N)$ intertwine $T_1$ and $T_2$. Then we have $fM_\text{ord} \subset N_\text{ord}$, $f M_\text{non-ord} \subset N_\text{non-ord}$, the decompositions being taken with respect to $T_1$ and $T_2$.
\item Let $M$ be a finite $R$-module, and let $T \in \End_R(M)$. If $I \subset R$ is a proper ideal, then we have 
\begin{gather*}
(M/(I))_\text{ord} = M_\text{ord}/(I), \\ 
(M/(I))_\text{non-ord} = M_\text{non-ord}/(I).
\end{gather*}
\end{enumerate}
\end{lemma}
\begin{proof}
The $\ffrm_R$-adic topology on $\End_R(M)$ coincides with the topology induced by kernels of the maps $\End_R(M) \to \End_R(M/(\ffrm_R^n))$, $n \geq 1$. Replacing $R$ by one of the quotients $R / \ffrm_R^n$, we can therefore assume that $R$ is Artinian. Let $\overline{P}(X) \in k[X]$ be the characteristic polynomial of $T$ on $M/(\ffrm_R)$, and let $P(X) \in R[X]$ be a monic lift of $\overline{P}$ such that $P(T) = 0$ on $M$. There is a unique factorization $\overline{P} = \overline{A}\overline{B}$ in $k[X]$, where the constant term of $\overline{A}$ is a unit, and $\overline{B} = X^m$, for some $m \geq 0$; by Hensel's lemma, this factorization lifts uniquely to a factorization $P = AB$ in $R[X]$, with $A, B$ monic.

We define $M_\text{ord} = B(T)M$, $M_\text{non-ord} = A(T)M$. Since $P(T)M = 0$, we have $A(T)M_\text{ord} = 0$, which implies that $T$ is invertible on $M_\text{ord}$ (as the constant term of $A$ is a unit). We also have $B(T) M_\text{non-ord} = 0$, which implies that $T$ is topologically nilpotent (hence nilpotent, as $R$ is Artinian) on $M_\text{non-ord}$. Since $\overline{A}, \overline{B}$ generate the unit ideal in $k[X]$, we can find $f, g \in R[X]$ such that $f A + g B = 1 + h$, where $h \in \ffrm_R[X]$. It follows that $f(T)A(T) + g(T)B(T)$ is invertible in $\End_R(M)$, implying that $M = M_\text{ord} + M_\text{non-ord}$. If $x \in M$ lies in the intersection of these two submodules, then the sequence $T^n x$ is eventually 0 (as $T$ is nilpotent), hence $x = 0$ (as $T$ is invertible on $M_\text{ord}$). We deduce that in fact $M = M_\text{ord} \oplus M_{\text{non-ord}}$. In particular, there exists an integer $n_0$ such that for all $n \geq n_0$, $T^{n!} = e$ is the identity on $M_\text{ord}$ and trivial on $M_\text{non-ord}$.

It remains to check that this decomposition is unique. Suppose that there is another decomposition $M = M_\text{ord}' \oplus M_\text{non-ord}'$ such that $T$ is invertible on $M_\text{ord}'$ and nilpotent on $M_\text{non-ord}'$. After possibly increasing $n_0$, we can assume that for all $n \geq n_0$, $T^{n!} = e$ is the identity on $M_\text{ord}'$ and trivial on $M_\text{non-ord}'$. We then obtain
\begin{gather*} M_\text{ord} = e M = M_\text{ord}', \\
M_\text{non-ord} = (1-e)M = M_\text{non-ord}'.
\end{gather*}
This completes the proof of the first part of the lemma. For the second part, it is immediate that $T_1 M_\text{non-ord} \subset N_\text{non-ord}$. Let $B(X) \in R[X]$ be a monic polynomial lifting $X^m \in k[X]$ such that $B(T_2) N_\text{non-ord} = 0$. Then $f(M_\text{ord}) = f(B(T_1) M_\text{ord}) = B(T_2) f(M_\text{ord}) \subset N_\text{ord}$. This completes the proof of the second part. The third part of the lemma is immediate from our expressions $M_\text{ord} = B(T)M$, $M_\text{non-ord} = A(T)M$.
\end{proof}
In the situation of the first part of Lemma \ref{lem_ordinary_part_for_modules}, we call $M_\text{ord}$ the ordinary part of $M$ (with respect to $T$), and $M_\text{non-ord}$ the non-ordinary part.
\begin{lemma}\label{lem_ordinary_part_for_complexes}
\begin{enumerate} \item Let $C^\bullet$ be a complex of finite $R$-modules, and let $T \in \End_R(C^\bullet)$. Then, decomposing each term of $C^\bullet$ into ordinary and non-ordinary parts, we obtain a decomposition $C^\bullet = C^\bullet_\text{ord} \oplus C^\bullet_\text{non-ord}$ of complexes of $R$-modules. Moreover, for each integer $i$ there are canonical isomorphisms
\begin{gather*}
H^i(C^\bullet_\text{ord}) \cong H^i(C^\bullet)_\text{ord},\\
H^i(C^\bullet_\text{non-ord}) \cong H^i(C^\bullet)_\text{non-ord}.
\end{gather*}
\item Let $C^\bullet,$ $D^\bullet$ be complexes of finite $R$-modules, let $T_1 \in \End_R(C^\bullet)$, $T_2 \in \End_R(C^\bullet)$ and let $f \in \Hom_R(C^\bullet, D^\bullet)$ intertwine $T_1$ and $T_2$. Then $f C^\bullet_\text{ord} \subset D^\bullet_\text{ord}$ and $f C^\bullet_\text{non-ord} \subset D^\bullet_\text{non-ord}$.
\end{enumerate}
\end{lemma}
\begin{proof}
For the first part, a standard reduction shows that it is enough to prove the following statement: take a commutative diagram with exact rows
\[ \xymatrix{ 0 \ar[r] & M'\ar[d]^{T'} \ar[r] & M \ar[r]\ar[d]^T & M'' \ar[r]\ar[d]^{T''} & 0 \\
 0 \ar[r] & M'\ar[r] & M \ar[r]& M'' \ar[r] & 0. } \]
 Then the induced sequences of $R$-modules
 \[ \xymatrix@1{ 0 \ar[r] & M'_\text{ord} \ar[r] & M_\text{ord} \ar[r] & M''_\text{ord} \ar[r] & 0 } \]
 and 
 \[ \xymatrix@1{ 0 \ar[r] & M'_\text{non-ord} \ar[r] & M_\text{non-ord} \ar[r] & M''_\text{non-ord} \ar[r] & 0 } \]
 are exact. Let $\overline{P}[X]$ be the characteristic polynomial of $T$ on $M/(\ffrm_R)$, and factor $\overline{P} = \overline{A} \overline{B}$ in $k[X]$, where the constant term of $\overline{A}$ is a unit and $\overline{B} = X^m$ for some $m \geq 0$. Let $P(X) \in R[X]$ be a monic lift of $\overline{P}$ such that $P(T) = 0$ on $M$. Then the above factorization lifts to a factorization $P = AB$ in $R[X]$ with $A, B$ monic. 
 
We now show that the sequence of ordinary parts is exact. Exactness on the left is clear. For exactness on the right, we observe that $B(T)$ is invertible on $M''_\text{ord}$ and zero on $M''_\text{non-ord}$; thus $M_\text{ord}'' = B(T)M''$ is the image in $M''$ of $B(T)M = M_\text{ord}$. For exactness in the middle, suppose that $m \in M_\text{ord}$ is killed in $M''_\text{ord}$. Since $B(T)$ is invertible on $M_\text{ord}$, we can write $m = B(T) n$ with $n \in M_\text{ord}$. The element $n$ is also killed in $M''_\text{ord}$, hence lifts to $M'$, hence $m = B(T) n$ lifts to $B(T) M' = M'_\text{ord}$. The proof of exactness of the sequence of non-ordinary parts is similar, using $A(T)$ instead of $B(T)$. This completes the proof of the first part of the lemma. The second part follows immediately from the second part of Lemma \ref{lem_ordinary_part_for_modules}.
\end{proof}
We now discuss idempotents from the point of view of the derived category. If $\cA$ is any additive category, and $X \in \cA$, $X_1, \dots, X_n \in \cA$, then the following are equivalent:
\begin{enumerate}
\item There exists an isomorphism $X \cong X_1 \oplus \dots \oplus X_n$.
\item There exist morphisms $i_j : X_j \to X$, $p_j : X \to X_j$, $j = 1, \dots, n$, such that $p_j i_j = 1$ for each $j$ and $1_X = \sum_{j=1}^n i_j p_j$.
\end{enumerate}
In the second case, each $i_j p_j \in \Hom_\cA(X, X)$ is an idempotent. We say that an additive category is idempotent complete if for any $X \in \cA$ and any idempotent $e \in \Hom_\cA(X, X)$, there exists a decomposition $X \cong X_1 \oplus X_2$ such that $e = i_1 p_1$. The derived category of an abelian category is idempotent complete.

\begin{lemma}\label{lem_ordinary_part_in_derived_category} Let $C^\bullet$ be a bounded above, perfect complex of $R$-modules. Let $t \in \End_{\mathbf{D}(R)}(C^\bullet)$. Then there is a unique idempotent $e \in \End_{\mathbf{D}(R)}(C^\bullet)$ with the following properties:
\begin{enumerate}
\item $e$ is a polynomial in $t$ (with $R$-coefficients).
\item The map $H(t)$ acts invertibly on $H(e)H^\ast(C^\bullet)$ and topologically nilpotently on $(1-H(e))H^\ast(C^\bullet)$. 
\end{enumerate}
\end{lemma}
\begin{proof}
We first show existence of $e$. Consider the $R$-subalgebra $S$ of $\End_{\mathbf{D}(R)}(C^\bullet)$ generated by $t$. Then $\End_{\mathbf{D}(R)}(C^\bullet)$ is a finite $R$-module, so $S$ is a finite $R$-algebra, hence admits a decomposition $S = \prod_\ffrm S_\ffrm$ over the finitely many maximal ideals of $S$. Let $e \in S$ denote the idempotent which is 0 on the factors such that $t \in \ffrm$, and $1$ on the factors such that $t \not\in \ffrm$. It is then clear that $e$ and $t$ commute. 

We must show that $t$ acts invertibly on $H(e)H^\ast(C^\bullet)$ and topologically nilpotently on $(1-H(e))H^\ast(C^\bullet)$. Replacing $C^\bullet$ by a quasi-isomorphic complex, we can assume that $C^\bullet$ is a good complex and $t$ is represented by an element $T \in \End_R(C^\bullet)$. We see that as $N \to \infty$ the endomorphism of the finite $R$-module $H^\ast(C^\bullet)$ induced by $t^{N!}$ approaches the endomorphism induced by $e$, so this is clear. 

We now show uniqueness. The idempotents of $S$ are in bijection with the sets of maximal ideals of $S$; we must show that there is a unique idempotent with the given action of cohomology. After writing $C^\bullet$ as a direct sum, we can assume that $S$ is local and must show that either $H(t)$ acts invertibly on $H^\ast(C^\bullet)$ or $H(t)$ acts nilpotently. Replacing $C^\bullet$ by a quasi-isomorphic good complex, it is clear that these two possibilities correspond to the cases $t \in \ffrm_S$ or $t\not\in \ffrm_S$. This completes the proof of the lemma.
\end{proof}

If $C^\bullet$ is a bounded above, perfect complex of $R$-modules, $T \subset \End_{\mathbf{D}(R)}(C^\bullet)$ is a finite (commutative) $R$-subalgebra, and $t \in \End_{\mathbf{D}(R)}(C^\bullet)$ is an endomorphism that commutes with $T$ (for example, if $t \in T$), then we define the ordinary part $C^\bullet_\text{ord}$ of $C^\bullet$ as the direct summand of $C^\bullet$ corresponding to the idempotent $e$, corresponding to the endomorphism $t$. If $C^\bullet$ is a good complex and we choose a representative $\widetilde{t} \in \End_{R}(C^\bullet)$ of $t$, then this agrees with the definition of $C^\bullet_\text{ord}$ given in Lemma \ref{lem_ordinary_part_for_complexes}. The  complex $C^\bullet_\text{ord}$ is defined uniquely up to quasi-isomorphism in $\mathbf{D}(R)$, and there is a canonical homomorphism $T \to \End_{\mathbf{D}(R)}(C^\bullet_\text{ord})$, given (essentially) by $x \mapsto exe$. 

Similarly, if $C^\bullet$ is a bounded above, perfect complex, $T \subset \End_{\mathbf{D}(R)}(C^\bullet)$ is a finite $R$-subalgebra, and $\ffrm \subset T$ is a maximal ideal, then $T_\ffrm \subset T$ is a direct factor (as $R$ is complete), and there is a corresponding idempotent $e_\ffrm \in T$. The corresponding direct factor $C^\bullet_\ffrm$ of $C^\bullet$ is defined uniquely up to quasi-isomorphism in $\mathbf{D}(R)$, and there is a canonical homomorphism $T \to \End_{\mathbf{D}(R)}(C^\bullet_\ffrm)$; its image is canonically identified with $T_\ffrm$.

We will use these constructions later to define the ordinary parts and $\ffrm$-parts of complexes which compute the cohomology of arithmetic locally symmetric spaces. In this context, the algebra $T$ will be the algebra generated by the Hecke operators at finite places, and $\ffrm$ will be a maximal ideal of this Hecke algebra corresponding to a given Hecke eigenclass. In this connection, we will also need to glue complexes. We now give a naive formulation of this kind of operation.

\begin{lemma}\label{lem_limit_of_good_complexes}
Let $I_1 \supset I_2 \supset \dots$ be a nested sequence of open ideals of $R$ such that $\cap_{i=1}^\infty I_i = 0$. Let $R_c = R/I_c$. Suppose given for each $c \geq 1$ a good complex $M^\bullet_c$ of $R_c$-modules, together with an isomorphism $f_{c+1} : M^\bullet_{c+1} \otimes_{R_{c+1}} R_{c} \to M^\bullet_{c}$. Then:
\begin{enumerate}
\item $M^\bullet_\infty = \plim_c M^\bullet_c$ is a good complex of $R$-modules, and there are canonical isomorphisms $F_c : M^\bullet_\infty \otimes_R R_c \cong M^\bullet_c$ for each $c \geq 1$.
\item The natural map $H^\ast(M^\bullet_\infty) \to \plim_c H^\ast(M^\bullet_c)$ is an isomorphism.
\item Suppose given for each $c \geq 1$ a good complex $N^\bullet_c$ of $R_c$-modules, together with an isomorphism $g_{c+1}  : N^\bullet_{c+1} \otimes_{R_{c+1}} R_c \to N^\bullet_c$ and an element $t_c \in \Hom_{\mathbf{K}(R_c)}(M^\bullet_c, N^\bullet_c)$ such that $g_{c+1} \circ t_{c+1} = t_c \circ f_{c+1}$ in $\mathbf{K}(R_c)$. Then there exists a unique element $t_\infty \in \Hom_{\mathbf{K}(R)}(M^\bullet_\infty, N^\bullet_\infty)$ such that for each $c \geq 1$, the diagram 
\[ \xymatrix{ M^\bullet_\infty \ar[r]^{t_\infty} \ar[d] & N^\bullet_\infty \ar[d] \\
M^\bullet_c \ar[r]^{t_c} & N^\bullet_c} \]
commutes in $\mathbf{K}(R)$. 
\end{enumerate}
\end{lemma}
\begin{proof}
The first part is easy. We observe that each of the rings $R_c$ is Artinian, hence any finite $R_c$-module is Artinian. In particular, any inverse system of $R$-modules $(N_c)_{c \geq 1}$ in which each $N_c$ is a finite $R_c$-module satisfies the Mittag-Leffler condition. This implies the second part of the lemma (cf. \cite[Theorem 3.5.8]{Wei94}). The third part follows in a similar way: we want to show that the natural map
\[ \Hom_{\mathbf{K}(R)}(M^\bullet_\infty, N^\bullet_\infty) \to \plim_c \Hom_{\mathbf{K}(R_c)}(M^\bullet_c, N^\bullet_c) \]
is an isomorphism. For each $1 \leq c \leq \infty$, let $H_c$ denote the $R_c$-module of graded homomorphisms $\oplus_{i \in \bbZ} M^i_c \to \oplus_{i \in \bbZ} N^{i-1}_c$ (which need not respect differentials). Then there are exact sequences
\[ \xymatrix@1{ H_c \ar[r] & \Hom_R(M^\bullet_c, N^\bullet_c) \ar[r] & \Hom_{\mathbf{K}(R)}(M^\bullet_c, N^\bullet_c) \ar[r] & 0,} \]
the first arrow being given by $s \mapsto ds + sd$. It is easy to see that the natural maps
\[ H_\infty \to \plim_c H_c, \text{ } \Hom_R(M^\bullet_\infty, N^\bullet_\infty) \to \plim_c \Hom_R(M^\bullet_c, N^\bullet_c) \]
are isomorphisms, so another application of the Mittag-Leffler property gives the result.
\end{proof}
\begin{lemma}\label{lem_limit_of_infinite_type_complexes}
Let $I_1 \supset I_2 \supset \dots$ be a nested sequence of open ideals such that $\cap_{i=1}^\infty I_i = 0$. Let $R_c = R/I_c$. Suppose given for each $c \geq 1$ a perfect, bounded above complex $M^\bullet_c$ of free $R_c$-modules, together with a quasi-isomorphism $f_{c+1} : M^\bullet_{c+1} \otimes_{R_{c+1}} R_c \cong M_c^\bullet$. Then we can find a minimal complex $F^\bullet_\infty$ of $R$-modules and a collection of quasi-isomorphisms $g_c : F^\bullet_\infty \otimes_R R_c \to M^\bullet_c$ such that for each $c \geq 1$, the diagram
\[ \xymatrix{ F^\bullet_\infty \otimes_R R_{c+1} \ar[r]^-{g_{c+1}} \ar[d] & M^\bullet_{c+1} \ar[d]^{f_{c+1}} \\
F^\bullet_\infty \otimes_R R_c \ar[r]^-{g_c} & M^\bullet_c} \]
commutes in $\mathbf{K}(R)$. Moreover, $F^\bullet_\infty$ is unique in the following sense: if $(G^\bullet_\infty, (h_c)_{c \geq 1})$ is another pair satisfying these conclusions, then there is an isomorphism $j : F^\bullet_\infty \to G^\bullet_\infty$ such that $j \otimes_R R_c = h_c^{-1} \circ g_c$ in $\mathbf{D}(R_c)$ for all $c \geq 1$, and $j$ is uniquely determined up to chain homotopy.
\end{lemma}
\begin{proof} By hypothesis, we can find for each $c \geq 1$ a minimal complex $N^\bullet_c$ of $R_c$-modules, together with morphisms $a_c : N^\bullet_c \to M_c^\bullet$ and $b_c : M^\bullet_c \to N^\bullet_c$ such that each of $a_c b_c$ and $b_c a_c$ is chain homotopic to the identity. The composite map 
\[ b_c \circ f_{c+1} \circ (a_{c+1} \otimes_{R_{c+1}} R_c) : N^\bullet_{c+1} \otimes_{R_{c+1}} R_c \to N^\bullet_c \]
is then a quasi-isomorphism, hence an isomorphism of minimal complexes. By Lemma \ref{lem_limit_of_good_complexes}, we can find a minimal complex $F^\bullet_\infty$ of $R$-modules, together with a compatible system of isomorphisms $h_c : F^\bullet_\infty \otimes_R R_c \cong N_c$. We then define $g_c = a_c \circ h_c$. It is clear that the pair $(F^\bullet_\infty, (g_c)_{c \geq 1})$ has the desired properties. The uniqueness follows from the last part of Lemma \ref{lem_limit_of_good_complexes}.
\end{proof}
\begin{proposition}\label{prop_limit_and_gluing_ordinary_parts}
Let $I_1 \supset I_2 \supset \dots$ be a nested sequence of open ideals of $R$ such that $\cap_{i=1}^\infty I_i = 0$. Let $R_c = R/I_c$. Suppose given for each $c \geq 1$ a perfect, bounded above complex $M^\bullet_c$ of free $R_c$-modules, together with an element $t_c \in \End_{\mathbf{D}(R_c)}(M^\bullet_c) = \End_{\mathbf{K}(R_c)}(M^\bullet_c)$. Suppose given as well morphisms $f_{c+1} : M^\bullet_{c+1} \otimes_{R_{c+1}} R_c \to M^\bullet_c$ making the diagrams
\[ \xymatrix{ M^\bullet_{c+1} \otimes R_c \ar[r]^-{t_{c+1} \otimes R_c} \ar[d]^{f_{c+1}} & M^\bullet_{c+1} \otimes R_c \ar[d]^{f_{c+1}} \\
M^\bullet_c \ar[r]^{t_c} & M^\bullet_c } \]
commute in $\mathbf{D}(R_c)$, and inducing isomorphisms
\[ H^\ast(M^\bullet_{c+1} \otimes_{R_{c+1}} R_c)_\text{ord} \cong H^\ast(M^\bullet_c)_\text{ord}, \]
the decomposition into ordinary parts taken with respect to the operator $t_c$. Then we can find a minimal complex $F^\bullet_\infty$ of $R$-modules, together with morphisms 
\begin{gather*} g_c : F^\bullet_\infty \otimes_R R_c \to M^\bullet_c,\\
g_c' : M^\bullet_c \to F^\bullet_\infty \otimes_R R_c,
\end{gather*}
all satisfying the following conditions:
\begin{enumerate}
\item We have equalities $g'_c g_c = 1$, $f_{c+1} \circ (g_{c+1} \otimes_{R_{c+1}} R_c) = g_c$ and $g_c' \circ f_{c+1} = (g'_{c+1} \otimes_{R_{c+1}} R_c)$ of morphisms in $\mathbf{D}(R_c)$.
\item  $g_c g'_c$ is an idempotent in  $\End_{\mathbf{D}(R_c)}(M^\bullet_c)$. 
\item The induced map on cohomology of $g_c g'_c$ is projection onto $H^\ast(M^\bullet_c)_\text{ord}$ along $H^\ast(M^\bullet_c)_{\text{non-ord}}$. In particular, there are commutative squares for every $c \geq 1$:
\[ \xymatrix{ H^\ast(F^\bullet_\infty \otimes_R R_{c+1}) \ar[r]^-\cong \ar[d] & H^\ast(M^\bullet_{c+1})_\text{ord} \ar[d] \\
H^\ast(F^\bullet_\infty \otimes_R R_c) \ar[r]^-\cong & H^\ast(M^\bullet_c)_\text{ord}.} \]
\end{enumerate}
The pair $(F^\bullet_\infty, (g_c)_{c \geq 1}, (g'_c)_{c \geq 1})$ is unique in the following sense: if $(G^\bullet_\infty, (h_c)_{c \geq 1}, (h'_c)_{c \geq 1})$ is another tuple satisfying points 1--3, then we can find an isomorphism $j : F^\bullet_\infty \to G^\bullet_\infty$ such that $j \otimes_R R_c = h'_c \circ g_c$ in $\mathbf{D}(R_c)$, and $j$ is uniquely determined up to chain homotopy.

Finally, suppose given for each $c \geq 1$ an element $s_c \in \End_{\mathbf{D}(R_c)}(M^\bullet_c)$ such that $f_{c+1} \circ (s_{c+1} \otimes_{R_{c+1}} R_c) = s_c \circ f_{c+1}$, and $s_c t_c = t_c s_c$. Then there is a unique element $s_\infty \in \End_{\mathbf{D}(R)}(F^\bullet_\infty)$ with the property that $(s_\infty \otimes_R R_c) = g_c s_c g'_c$ in $\mathbf{D}(R_c)$ for each $c \geq 1$. In particular, the action of $s_\infty \otimes_R R_c$ on the direct summand 
\[ H^\ast(F^\bullet_\infty \otimes_R R_c) = H^\ast(M^\bullet_c)_\text{ord} \subset H^\ast(M^\bullet_c) \]
agrees with the restriction of the element $s_c$ to this subspace.
\end{proposition}
\begin{proof}
After replacing each $M^\bullet_c$ by a quasi-isomorphic complex, we can assume that each $M^\bullet_c$ is minimal. Let $T_c \in \End_R(M^\bullet_c)$ denote a representative of $t_c \in \End_{\mathbf{D}(R_c)}(M^\bullet_c)$, and define $F^\bullet_c = M^\bullet_{c, \text{ord}}$, as in Lemma \ref{lem_ordinary_part_for_complexes}, the decomposition taken with respect to $T_c$. Let $i_c : F^\bullet_c \to M^\bullet_c$ and $p_c : M^\bullet_c \to F^\bullet_c$ denote the canonical inclusion and projection, respectively. We observe that there is an integer $N_c \geq 1$ such that the equality $i_c \circ p_c = t_c^{N_c!}$ holds in $\End_{\mathbf{D}(R_c)}(M^\bullet_c)$ (since the equality $i_c \circ p_c = T_c^{N_c!}$ holds at the level of cochains). In particular, the equality $f_{c+1} t_{c+1} = t_c f_{c+1}$ implies the equality $f_{c+1} i_{c+1} p_{c+1} = i_c p_c f_{c+1}$ in $\mathbf{D}(R_c)$.

The complex $F^\bullet_c$ is minimal (being a direct summand of a minimal complex), and for each $c \geq 1$, the composite map
\[ H^\ast(F^\bullet_c) \to H^\ast(M^\bullet_c) \to H^\ast(M^\bullet_c)_\text{ord} \]
is an isomorphism. There is a morphism $a_{c+1} : F^\bullet_{c+1} \otimes_{R_{c+1}} R_c \to F^\bullet_c$, given by the formula $a_{c+1} = p_c \circ f_{c+1} \circ (i_{c+1} \otimes_{R_{c+1}} R_c)$, and we see that $a_{c+1}$ is a quasi-isomorphism of minimal complexes, hence an isomorphism. We can therefore apply Lemma \ref{lem_limit_of_good_complexes} to the complexes $F^\bullet_c$ to obtain a minimal complex $F^\bullet_\infty$ of $R$-modules, together with isomorphisms $b_c : F^\bullet_\infty \otimes_R R_c \cong F^\bullet_c$ for each $c \geq 1$. We define $g_c = i_c \circ b_c$, and $g'_c = b_c^{-1} \circ p_c$. Then $g'_c g_c = b_c^{-1} p_c i_c b_c = 1$ and $g_c g'_c = i_c p_c$ is an idempotent (even as a morphism of complexes, before passing to the derived category). 

The equality $f_{c+1} \circ (g_{c+1} \otimes_{R_{c+1}} R_c) = g_c$ is equivalent to the equality $i_c p_c f_{c+1} i_{c+1} b_{c+1} = f_{c+1} i_{c+1} b_{c+1}$, which holds because $i_c p_c f_{c+1} = f_{c+1} i_{c+1} p_{c+1}$ and $p_{c+1} i_{c+1} = 1$. The equality $g_c' \circ f_{c+1} = (g'_{c+1} \otimes_{R_{c+1}} R_c)$ is equivalent to the equality $p_c f_{c+1} = p_c f_{c+1} i_{c+1} p_{c+1}$, which holds for essentially the same reasons This completes the proof that $(F^\bullet_\infty, (g_c)_{c \geq 1}, (g'_c)_{c \geq 1})$ has the claimed properties.

If $(G^\bullet_\infty, (h_c)_{c \geq 1}, (h'_c)_{c \geq 1})$ is another minimal complex as in the statement of the proposition, then we have diagrams for every $c \geq 1$:
\[ \xymatrix{ F^\bullet_\infty \otimes_R R_c \ar[r]^-{g_{c+1} \otimes R_c} \ar[d] & M^\bullet_{c+1} \ar[r]^-{h'_{c+1} \otimes R_c} \ar[d] & G^\bullet_\infty \otimes_R R_c \ar[d] \\
F^\bullet_\infty \otimes_R R_c \ar[r]^-{g_c} & M^\bullet_{c} \ar[r]^-{h'_c} & G^\bullet_\infty \otimes_R R_c.} \]
The composites $h'_c g_c$ are quasi-isomorphisms of minimal complexes, hence isomorphisms, and the third part of Lemma \ref{lem_limit_of_good_complexes} implies that they glue to an isomorphism $F^\bullet_\infty \cong G^\bullet_\infty$ with the desired properties.

Finally, suppose given for each $c \geq 1$ an element $s_c \in \End_{\mathbf{D}(R_c)}(M^\bullet_c)$ such that $f_{c+1} \circ (s_{c+1} \otimes_{R_{c+1}} R_c) = s_c \circ f_{c+1}$, and $s_c t_c = t_c s_c$. This implies that $s_c i_c p_c = i_c p_c s_c$. Let $s'_c = g'_c s_c g_c = b_c^{-1} p_c s_c i_c b_c \in \End_{\mathbf{D}(R_c)}(F^\bullet_\infty \otimes_R R_c) = \End_{\mathbf{K}(R_c)}(F^\bullet_\infty \otimes_R R_c)$. We claim that $s'_{c+1} \otimes_{R_{c+1}} R_c = s'_c$. This is equivalent to the equality $p_c s_c i_c a_{c+1} = a_{c+1} p_{c+1} s_{c+1} i_{c+1}$. The left hand side of this equation equals
\[ p_c s_c i_c p_c f_{c+1} i_{c+1} = p_c s_c f_{c+1} i_{c+1}. \]
The right hand side equals
\[ p_c f_{c+1} i_{c+1} p_{c+1} s_{c+1} i_{c+1} = p_c f_{c+1} s_{c+1} i_{c+1}. \]
These are the same, because of the relation $s_c f_{c+1} = f_{c+1} s_{c+1}$. The last part of Lemma \ref{lem_limit_of_good_complexes} now implies the existence and uniqueness of the desired element $s_\infty$.
\end{proof}
\section{Abstract patching}\label{sec_abstract_patching}

Let $p$ be a prime, $E$ a finite extension of $\bbQ_p$, and let $\cO$ be its ring of integers. We write $\lambda \subset \cO$ for the unique maximal ideal and $k = \cO/\lambda$.
\begin{proposition}\label{prop_main_patching_argument}
Fix $\Lambda \in \CNL_\cO$ and $q \geq 1$, and define $S_\infty = \Lambda \llbracket S_1, \dots, S_q \rrbracket$ and $\fra = \ker(S_\infty \to \Lambda)$. Let $S_\infty \supset I_1 \supset I_2 \dots$ be a decreasing sequence of open ideals of $S_\infty$ such that $\cap_{i \geq 1} I_i = 0$. If $N \geq 1$, then we set $S_N = S_\infty/I_N$. Suppose given the following data.
\begin{enumerate}
\item A good complex $C^\bullet_0$ of $\Lambda$-modules concentrated in degrees $[0, d]$ for some $d \geq 0$.
\item An object $R_0 \in \CNL_\Lambda$ together with a map $R_0 \to \End_{\mathbf{D}(\Lambda)}(C^\bullet_0)$.
\item For every integer $N \geq 1$, a good complex $C^\bullet_N$ of $S_N$-modules, together with an isomorphism $f_N : C^\bullet_N / (\fra) \to C^\bullet_0 \otimes_{S_\infty} S_N$. \(We consider $C^\bullet_0$ as an $S_\infty$-module via the augmentation homomorphism $S_\infty \to \Lambda$.\)
\item For every integer $N \geq 1$, an object $R_N \in \CNL_{S_N}$ and maps $R_N \to \End_{\mathbf{D}(S_N)}(C^\bullet_N)$ and $R_N \to R_0/(I_N)$ of $S_N$-algebras, such that the following diagram commutes:
\[ \xymatrix{ R_N \ar[r] \ar[d] & \End_{\mathbf{D}(S_N)}(C^\bullet_N) \ar[d] \\
R_0/(I_N) \ar[r]  & \End_{\mathbf{D}(\Lambda/(I_N))}(C^\bullet_0 \otimes_{S_\infty} S_N). } \]
\item A complete Noetherian $\Lambda$-algebra $R_\infty$ and for each integer $N \geq 1$, a surjective map of $\Lambda$-algebras $g_N : R_\infty \to R_N$. 
\end{enumerate}
Then we can find the following data.
\begin{enumerate}
\item A good complex $C^\bullet_\infty$ of $S_\infty$-modules, together with an isomorphism $f_\infty : C^\bullet_\infty \otimes_{S_\infty} \Lambda \to C^\bullet_0$.
\item A homomorphism $S_\infty \to R_\infty$ in $\CNL_\Lambda$, making $R_\infty$ into an $S_\infty$-algebra.
\item A commutative diagram of $S_\infty$-algebras
\[ \xymatrix{ R_\infty \ar[r] \ar[d] & \End_{\mathbf{D}(S_\infty)}(C^\bullet_\infty) \ar[d]^{- \otimes_{S_\infty} \Lambda} \\
R_0 \ar[r] & \End_{\mathbf{D}(\Lambda)}(C^\bullet_0). } \]
\end{enumerate}
\end{proposition}
\begin{proof}
We patch. By Chevalley's theorem, we can assume without loss of generality that $I_N = \ffrm_{S_\infty}^{m_N}$, for a sequence $m_N$ of integers which tend to infinity as $N \to \infty$. For simplicity, let us assume that in fact $I_N = \ffrm_{S_\infty}^N$. We define $s = \dim_k H^\ast(C^\bullet_0 \otimes_\Lambda \Lambda/\ffrm_\Lambda)$. If $t \in \ffrm_{R_N}$, then $t$ acts nilpotently on 
\[ H^\ast(C^\bullet_N \otimes_{S_N} k) \cong H^\ast(C^\bullet_0 \otimes_\Lambda k), \]
hence $t^s$ acts as 0 on $H^\ast(C^\bullet_N \otimes_{S_N} k)$. The existence of the spectral sequence of a filtered complex implies that $t^{Ns}$ then acts trivially on $H^\ast(C^\bullet_N \otimes_{S_\infty} S_\infty/\ffrm_{S_\infty}^N)$, hence (Lemma \ref{lem_nilpotents_in_derived_category}) the image of $t^{(d+1)Ns}$ is 0 in 
\[ \End_{\mathbf{D}(S_\infty/\ffrm_{S_\infty}^N)}(C^\bullet_N \otimes_{S_\infty} S_\infty/\ffrm_{S_\infty}^N). \]
In particular, if the Zariski tangent space of $R_\infty$ has $k$-dimension $g$, then the image of $\ffrm_{R_N}^{(d+1)Nsg}$ in $\End_{\mathbf{D}(S_\infty/\ffrm_{S_\infty}^N)}(C^\bullet_N \otimes_{S_\infty} S_\infty/\ffrm_{S_\infty}^N)$ is 0.

 We define a patching datum of level $N \geq 1$ to be a tuple $(D^\bullet, \psi, R, \eta_0, \eta_1, \eta_2)$, where:
\begin{itemize}
\item $D^\bullet$ is a good complex of $S_N$-modules.
\item $\psi$ is an isomorphism $D^\bullet \cong C^\bullet_0 \otimes_{S_\infty} S_N$ of complexes of $S_\infty$-modules.
\item $R$ is an object of $\CNL_{S_\infty}$ equipped with a surjection $\eta_0 : R_\infty \to R$ in $\CNL_\Lambda$, and maps $\eta_1 : R \to \End_{\mathbf{D}(S_N)}(D^\bullet)$ and $\eta_2 : R \to R_0/\ffrm_{R_0}^{(d+1)Nsg}$ of $S_N$-algebras. Moreover, $\ffrm_R^{(d+1)NSg} = 0$.
\item The following diagram is commutative:
\[ \xymatrix{ R \ar[r] \ar[d]^-{\eta_2} \ar[r]^-{\eta_1} & \End_{\mathbf{D}(S_\infty/\ffrm_{S_\infty}^N)}(D^\bullet) \ar[d] \\
R_0/\ffrm_{R_0}^{(d+1)Nsg} \ar[r] & \End_{\mathbf{D}(\Lambda/\ffrm_\Lambda^N)}(C^\bullet_0 \otimes_{\Lambda} \Lambda/\ffrm_\Lambda^N).} \]
\end{itemize}
We make the collection of patching data of level $N$ into a category $\mathrm{Patch}_N$ as follows: a morphism $\alpha : (D^\bullet, \psi, R, \eta_0, \eta_1,\eta_2) \to (E^\bullet, \psi', R', \eta_0', \eta_1', \eta_2')$ is a pair $\alpha = (f, g)$, where:
\begin{itemize}
\item $f : D^\bullet \to E^\bullet$ is an isomorphism such that $\psi' (f \otimes_{S_\infty} S_\infty/(\ffrm_{S_\infty}^N, \fra)) \psi^{-1}$ is the identity.
\item $g : R \to R'$ is an isomorphism in $\CNL_{S_\infty}$ that intertwines $\eta_0$ and $\eta_0'$, $\eta_1$ and $\eta_1'$, and $\eta_2$ and $\eta_2'$.
\end{itemize}
Evidently the category $\mathrm{Patch}_N$ is a groupoid, i.e.\ every morphism is an isomorphism. There is a collection of functors $(F_N : \mathrm{Patch}_{N+1} \to \mathrm{Patch}_N)_{N \geq 1}$, which assign to a tuple $(D^\bullet, \psi, R, \eta_0, \eta_1, \eta_2)$ the tuple $F_N(D^\bullet, \psi, \eta_0, \eta_1, \eta_2) = (E^\bullet, \psi', R', \eta_0', \eta_1', \eta_2')$ given as follows:
\begin{itemize}
\item $E^\bullet = D^\bullet \otimes_{S_\infty} S_\infty/\ffrm_{S_\infty}^N$.
\item $\psi'$ is the composite 
\[ E^\bullet \otimes_{S_\infty} S_\infty/(\ffrm_{S_\infty}^N + \fra)  \cong D^\bullet \otimes_{S_\infty} S_\infty/(\ffrm_{S_\infty}^N + \fra) \cong C^\bullet \otimes_{S_\infty} S_\infty/(\ffrm_{S_\infty}^N + \fra). \]
\item $R' = R/\ffrm_R^{(d+1)Nsg}$.
\item $\eta_0', \eta_1'$ and $\eta_2'$ are the obvious maps.
\end{itemize}
The set of isomorphism classes of patching data of level $N$ is finite. Indeed, it suffices to note that the cardinality of the ring $R$ in the tuple $(D^\bullet, \psi, R,\eta_0, \eta_1, \eta_2)$ is bounded solely in terms of $N$ (since $R_\infty/\ffrm_{R_\infty}^{(d+1)Nsg}$ is a ring of finite cardinality), and $\End_{\mathbf{D}(S_\infty/\ffrm_{S_\infty}^N)}(D^\bullet)$ has cardinality bounded above by $\# S_\infty/\ffrm_{S_\infty}^N \cdot (\dim_k D^\bullet \otimes_{S_\infty} k)^2$.

The set of isomorphism classes is also non-empty. Indeed, for each $M \geq 1$, we can define a patching datum $D(M, M) = (D^\bullet, \psi, R, \eta_0, \eta_1, \eta_2) \in \mathrm{Patch}_M$ as follows:
\begin{itemize}
\item $D^\bullet = C^\bullet_M \otimes_{S_\infty} S_\infty/\ffrm_{S_\infty}^M$.
\item $\psi = f_M$.
\item $R = R_M/\ffrm_{R_M}^{(d+1)Msg}$.
\item $\eta_0, \eta_1$ and $\eta_2$ are the obvious maps.
\end{itemize}
If $M \geq N \geq 1$, then we define $D(M, N) \in \mathrm{Patch}_N$ to be the image of $D(M, M)$ under the composite functor $F_N \circ F_{N+1} \circ \dots \circ F_{M-1}$. After diagonalization, we can find an increasing sequence $(M_N)_{N \geq 1}$ of integers, together with a system of isomorphisms $\alpha_N : F_N(D(M_{N+1}, N+1)) \to D(M_N, N)$. Passing to the inverse limit with respect to these isomorphisms, we obtain a tuple $(D_\infty^\bullet, \psi_\infty, R^\infty, \eta_0, \eta_1, \eta_2)$, where:
\begin{itemize}
\item $D_\infty^\bullet$ is a good complex of $S_\infty$-modules. 
\item $\psi_\infty$ is an isomorphism $D_\infty^\bullet \otimes_{S_\infty} S_\infty/(\fra) \cong C_0^\bullet$.
\item $R^\infty$ is an object of $\CNL_{S_\infty}$, and $\eta_0$ is a surjective homomorphism $R_\infty \to R^\infty$ in $\CNL_\Lambda$. 
\item $\eta_1 : R^\infty \to \End_{\mathbf{D}(S_\infty)}(D^\bullet_\infty)$ is a homomorphism of $S_\infty$-algebras.
\item $\eta_2 : R^\infty \to R_0$ is a homomorphism of $S_\infty$-algebras.
\item The following diagram is commutative:
\[ \xymatrix{ R^\infty \ar[r] \ar[d]^-{\eta_2} \ar[r]^-{\eta_1} & \End_{\mathbf{D}(S_\infty)}(D^\bullet_\infty) \ar[d] \\
R_0 \ar[r] & \End_{\mathbf{D}(\Lambda)}(C^\bullet_0).} \]
\end{itemize}
It follows that $\eta_1$ is injective. Since $S_\infty$ is formally smooth over $\Lambda$, we can lift the homomorphism $S_\infty \to R^\infty$ to a homomorphism $S_\infty \to R_\infty$. The proof is now completed on taking $C^\bullet_\infty = D^\bullet_\infty$.
\end{proof}
To illustrate the use of Proposition \ref{prop_main_patching_argument}, we state the following result, which will not be used later. It serves as a prototype for the proof of our main result, Theorem \ref{thm_r_equals_t}.
\begin{corollary}
With notation as in Proposition \ref{prop_main_patching_argument}, suppose given integers $n \geq l_0$ such that $\Lambda$ is a regular local ring of dimension $n$ and $R_\infty$ is a formally smooth $\Lambda$-algebra of dimension $\dim R_\infty = \dim S_\infty - l_0 = n + q - l_0$. Suppose moreover that $H^\ast(C^\bullet_0) \neq 0$ and $C^\bullet_0$ is concentrated in degrees $[0, l_0]$. Then the map $R_0 \to \End_{\mathbf{D}(\Lambda)}(C^\bullet_0)$ is injective and $H^{l_0}(C^\bullet_0)$ is a free $R_0$-module.
\end{corollary}
\begin{proof}
There is a commutative diagram of $S_\infty$-algebras:
\[ \xymatrix{ R_\infty \ar[r] \ar[d] & \End_{\mathbf{D}(S_\infty)}(C^\bullet_\infty) \ar[d]^{- \otimes_{S_\infty} \Lambda} \\
R_0 \ar[r] & \End_{\mathbf{D}(\Lambda)}(C^\bullet_0). } \]
By Lemma \ref{lem_dimension_theory}, we have
\[ \dim_{S_\infty} H^\ast(C^\bullet_\infty) = \dim_{R_\infty} H^\ast(C^\bullet_\infty) \leq \dim R_\infty = \dim S_\infty - l_0. \]
It follows from Lemma \ref{lem_dimension_criterion_for_exactness} that equality holds here, and the cohomology groups $H^i(C^\bullet_\infty)$ are non-zero if and only if $i = l_0$. We also see that
\[ \depth_{R_\infty} H^{l_0}(C^\bullet_\infty) \geq \depth_{S_\infty} H^{l_0}(C^\bullet_\infty) = \dim R_\infty, \]
so the Auslander-Buchsbaum formula implies that $H^{l_0}(C^\bullet_\infty)$ is in fact a free $R_\infty$-module. It follows that $H^{l_0}(C^\bullet_\infty) \otimes_{S_\infty} \Lambda \cong H^{l_0}(C^\bullet_0)$ is a free $R_\infty/(\fra)$-module, hence $R_\infty/(\fra) \cong R_0$ and $H^{l_0}(C^\bullet_0)$ is a free $R_0$-module. 
\end{proof}
\section{Galois deformation theory}\label{sec_galois_deformation_theory}

Let $F$ be a number field, let $p$ be a prime, and $S_p$ denote the set of places of $F$ dividing $p$. We fix a coefficient field $E$ together with a continuous, absolutely irreducible representation $\overline{\rho} : G_F \to \GL_n(k)$ and a continuous character $\mu : G_F \to \cO^\times$ which lifts $\det \overline{\rho}$.  We also fix a finite set $S$ of finite places of $F$, containing $S_p$ and the places at which $\overline{\rho}$ and $\mu$ are ramified. We assume for simplicity that $p > n$ and $n \geq 2$. In particular, the prime $p$ is odd.

For each $v \in S$, we fix a ring $\Lambda_v \in \CNL_\cO$ and we define $\Lambda = \widehat{\otimes}_{v \in S} \Lambda_v$, the completed tensor product being over $\cO$. Then $\Lambda \in \CNL_\cO$. Let $v \in S$. We write $\cD_v^\square : \CNL_{\Lambda_v} \to \mathrm{Sets}$ for the functor that associates to $R \in \mathrm{CNL}_{\Lambda_v}$ the set of all continuous homomorphisms $r : G_{F_v} \to \GL_n(R)$ such that $r \mod \ffrm_R = \overline{\rho}|_{G_{F_v}}$ and $\det r$ agrees with the composite $G_{F_v} \to \cO^\times \to R^\times$ given by $\mu|_{G_{F_v}}$ and the structural homomorphism $\cO \to R$. 
It is easy to see that $\cD_v^\square$ is represented by an object $R_v^{\square} \in \CNL_{\Lambda_v}$ (for example, by applying Schlessinger's criterion as in \cite{Maz89}).
\begin{definition}
Let $v \in S$. A local deformation problem for $\overline{\rho}|_{G_{F_v}}$ is a subfunctor $\cD_v \subset \cD_v^\square$ satisfying the following conditions:
\begin{itemize}
\item $\cD_v$ is represented by a quotient $R_v$ of $R_v^\square$.
\item For all $R \in \CNL_{\Lambda_v}$, $a \in \ker(\GL_n(R) \to \GL_n(k))$ and $r \in \cD_v(R)$, we have $a r a^{-1} \in \cD_v(R)$.
\end{itemize}
\end{definition}

\begin{definition}
A global deformation problem is a tuple
\[ \cS = (\overline{\rho}, \mu, S, \{ \Lambda_v \}_{v \in S}, \{ \cD_v \}_{v \in S}), \]
where:
\begin{itemize}
\item The objects $\overline{\rho} : G_F \to \GL_n(k)$, $\mu$, $S$ and $\{ \Lambda_v \}_{v \in S}$ are as at the beginning of this section.
\item For each $v \in S$, $\cD_v$ is a local deformation problem for $\overline{\rho}|_{G_{F_v}}$.
\end{itemize}
\end{definition}
\begin{definition}
Let $\cS = (\overline{\rho}, \mu, S, \{ \Lambda_v \}_{v \in S}, \{ \cD_v \}_{v \in S})$ be a global deformation problem. Let $R \in \CNL_{\Lambda}$, and let $\rho : G_F \to \GL_n(R)$ be a continuous homomorphism. We say that $\rho$ is of type $\cS$ if it satisfies the following conditions:
\begin{itemize}
\item $\rho$ is unramified outside $S$.
\item $\det \rho = \mu$. More precisely, the homomorphism $\det \rho : G_F \to R^\times$ agrees with the composite $G_F \to \cO^\times \to R^\times$ induced by $\mu$ and the structural homomorphism $\cO \to R$.
\item For each $v \in S$, the restriction $\rho|_{G_{F_v}}$ lies in $\cD_v(R)$, where we give $R$ the natural $\Lambda_v$-algebra structure arising from the homomorphism $\Lambda_v \to \Lambda$.
\end{itemize}
We say that two liftings $\rho_1, \rho_2 : G_F \to \GL_n(R)$ are strictly equivalent if there exists a matrix $a \in \ker(\GL_n(R) \to \GL_n(k))$ such that $\rho_2 = a \rho_1 a^{-1}$. 
\end{definition}
It is easy to see that strict equivalence preserves the property of being of type $\cS$. We write $\cD^\square_\cS$ for the functor $\CNL_{\Lambda} \to \Sets$ which associates to $R \in \CNL_{\Lambda}$ the set of liftings $\rho : G_F \to \GL_n(R)$ which are of type $\cS$. We write $\cD_\cS$ for the functor $\CNL_{\Lambda} \to \Sets$ which associates to $R \in \CNL_{\Lambda}$ the set of strict equivalence classes of liftings of type $\cS$.

\begin{definition} If $T \subset S$ and $R \in \mathrm{CNL}_\Lambda$, then we define a $T$-framed lifting of $\overline{\rho}$ to $R$ to be a tuple $(\rho, \{ \alpha_v \}_{v \in T})$, where $\rho : G_F \to \GL_n(R)$ is a lifting and for each $ v\in T$, $\alpha_v$ is an element of $\ker(\GL_n(R) \to \GL_n(k))$. Two $T$-framed liftings $(\rho_1, \{ \alpha_v \}_{v \in T})$ and $(\rho_n, \{ \beta_v \}_{v \in T})$ are said to be strictly equivalent if there is an element $a \in \ker(\GL_n(R) \to \GL_n(k))$ such that $\rho_2 = a \rho_1 a^{-1}$ and $\beta_v = a \alpha_v$ for each $v \in T$.
\end{definition}
We write $\cD^T_\cS$ for the functor $\CNL_{\Lambda} \to \Sets$ which associates to $R \in \CNL_{\Lambda}$ the set of strict equivalence classes of $T$-framed liftings $(\rho, \{ \alpha_v \}_{v \in T})$ to $R$ such that $\rho$ is of type $\cS$. 
\begin{theorem}
Let $\cS = (\overline{\rho}, \mu, S, \{ \Lambda_v \}_{v \in S}, \{ \cD_v \}_{v \in S})$ be a global deformation problem. Then the functors $\cD_\cS$, $\cD_\cS^\square$ and $\cD_\cS^T$ are representable by objects $R_\cS$, $R_\cS^\square$ and $R_\cS^T$, respectively, of $\CNL_{\Lambda}$. 
\end{theorem}
\begin{proof}
This is well-known; see \cite[Appendix 1]{Gou01} for a proof that $\cD_\cS$ is representable. The representability of the functors $\cD_\cS^\square$ and $\cD_\cS^T$ can be deduced easily from this.
\end{proof}
Let $\cS = (\overline{\rho}, \mu, S, \{ \Lambda_v \}_{v \in S}, \{ \cD_v \}_{v \in S})$ be a global deformation problem, and for each $v \in S$ let $R_v \in \CNL_{\Lambda_v}$ denote the representing object of $\cD_v$. We write $A_\cS^T = \widehat{\otimes}_{v \in T} R_v$ for the completed tensor product, taken over $\cO$, of the rings $R_v$. The ring $A^T_\cS$ has a canonical $\Lambda_T$-algebra structure, where $\Lambda_T = \widehat{\otimes}_{v \in T} \Lambda_v$; it is easy to see that $A^T_\cS$ represents the functor $\mathrm{CNL}_{\Lambda_T} \to \mathrm{Sets}$ which associates to a $\Lambda_T$-algebra $R$ the set of tuples $(\rho_v)_{v \in T}$, where for each $v \in T$, $\rho_v : G_{F_v} \to \GL_n(R)$ is a lifting of $\overline{\rho}|_{G_{F_v}}$ such that $\rho_v \in \cD_v(R)$ when we give $R$ the $\Lambda_v$-algebra structure arising from the homomorphism $\Lambda_v \to \Lambda_T \to R$.

The natural transformation $(\rho, \{ \alpha_v \}_{v \in T}) \mapsto ( \alpha_v^{-1} \rho|_{G_{F_v}} \alpha_v)_{v \in T}$ induces a canonical map $A^T_\cS \to R_\cS^T$, which is a homomorphism of $\Lambda_T$-algebras. We will generally use this construction only when $\Lambda_v = \cO$ for each $v \in S - T$, in which case there is a canonical isomorphism $\Lambda_T \cong \Lambda$.

\begin{lemma}\label{lem_framed_and_unframed_deformations}
Let $\cS$ be a global deformation problem, and let $\rho_\cS : G_F \to \GL_n(R_\cS)$ be a representative of the universal deformation. Choose $v_0 \in T$, and let $\cT = \cO\llbracket \{ X_{v, i, j} \}_{v \in T, 1 \leq i, j \leq n} \rrbracket/(X_{v_0, 1, 1})$. There is a canonical isomorphism $R_\cS^T \cong R_\cS \widehat{\otimes}_\cO \cT$.
\end{lemma}
\begin{proof}
There is a $T$-framed lifting over $R_\cS \widehat{\otimes}_\cO \cT$ given by the tuple $(\rho_\cS, \{ (X_{v, i, j})_{i, j} \}_{v \in T})$. It is easy to check that the induced map $R_\cS^T \to R_\cS \widehat{\otimes}_\cO \cT$ is an isomorphism.
\end{proof}

\subsection{Galois cohomology}\label{sec_galois_cohomology}

Let $\cS = (\overline{\rho}, \mu, S, \{ \Lambda_v \}_{v \in S}, \{ \cD_v \}_{v \in S})$ be a global deformation problem. We write $\ad \overline{\rho}$ for the $k[G_F]$-module given by the space $M_n(k)$ of $n \times n$ matrices, with $G_F$ acting via the adjoint (conjugation) action of $\overline{\rho}$. We write $\ad^0 \overline{\rho} \subset \ad \overline{\rho}$ for the $k[G_F]$-submodule consisting of matrices of trace 0. For each $v \in S$, let $R_v$ denote the representing object of $\cD_v$ and $I_v = \ker (R_v^\square \to R_v)$. There are canonical isomorphisms
\begin{equation}\label{eqn_defn_of_local_cocycle_space}
Z^1(F_v, \ad^0 \overline{\rho}) \cong \Hom_k(\ffrm_{R_v^\square}/(\ffrm_{R_v^\square}^2, \ffrm_{\Lambda_v}), k) \cong \Hom_{\CNL_{\Lambda_v}}(R^\square_v, k[\epsilon]/(\epsilon^2)),
\end{equation}
where we write $Z^1(F_v, \ad^0 \overline{\rho})$ for the space of continuous 1-coycles $\phi : G_{F_v} \to \ad^0 \overline{\rho}$. The pre-image in $Z^1(F_v, \ad^0 \overline{\rho})$ of $\Hom_{\CNL_{\Lambda_v}}(R_v, k[\epsilon]/(\epsilon^2)) \subset \Hom_{\CNL_{\Lambda_v}}(R^\square_v, k[\epsilon]/(\epsilon^2))$ is a $k$-vector subspace, which we call $\cL^1_v$. It contains the space of 1-coboundaries, and we write $\cL_v$ for the image of $\cL^1_v$ in $H^1(F_v, \ad^0 \overline{\rho})$.

Now fix a subset $T \subset S$, which may be empty, and suppose that $\Lambda_v = \cO$ for each $v \in S - T$. Then there is a canonical isomorphism $\Lambda_T \cong \Lambda$, and the map $A_\cS^T \to R_\cS^T$ is a morphism of $\Lambda$-algebras. In this case, we define (following \cite[\S 2]{Clo08}) a complex $C_{\cS, T}^\bullet(\ad^0 \overline{\rho})$ by the formula
\[ C_{\cS, T}^i(\ad \overline{\rho}) = \left\{ \begin{array}{ll} C^0(F_S/F, \ad \overline{\rho}) & i = 0\\
C^1(F_S/F, \ad^0 \overline{\rho}) \oplus_{v \in T} C^0(F_v, \ad \overline{\rho}) & i = 1\\
C^2(F_S/F, \ad^0 \overline{\rho}) \oplus_{v \in S-T} C^1(F_v, \ad^0 \overline{\rho})/\cL^1_v & i =2\\
C^i(F_S/F, \ad^0 \overline{\rho}) \oplus_{v \in S} C^{i-1}(F_v, \ad^0 \overline{\rho}) & \text{otherwise. }
\end{array}\right. \]
The boundary map is given by the formula
\begin{gather*} C^i_{\cS, T}(\ad^0 \overline{\rho}) \to C^{i+1}_{\cS, T}(\ad^0 \overline{\rho}) \\ 
(\phi, (\psi_v)_v) \mapsto (\partial \phi, (\phi|_{G_{F_v}} - \partial \psi_v)_v). 
\end{gather*} 
There is a long exact sequence of cohomology groups
\begin{equation}\label{eqn_long_exact_sequence_in_taylor_cohomology}
\xymatrix@R-2pc{ 
0 \ar[r] & H^0_{\cS, T}(\ad^0 \overline{\rho}) \ar[r] & H^0(F_S/F, \ad \overline{\rho}) \ar[r] & \oplus_{v \in T} H^0(F_v, \ad \overline{\rho}) \\
\ar[r] &  H^1_{\cS, T}(\ad^0 \overline{\rho}) \ar[r] & H^1(F_S/F, \ad^0 \overline{\rho}) \ar[r] & \oplus_{v \in T} H^1(F_v, \ad^0 \overline{\rho}) \oplus_{v \in S-T} H^1(F_v, \ad^0 \overline{\rho})/\cL_v \\
 \ar[r] &  H^2_{\cS, T}(\ad^0 \overline{\rho}) \ar[r] & H^2(F_S/F, \ad^0 \overline{\rho}) \ar[r] & \oplus_{v \in S} H^2(F_v, \ad^0 \overline{\rho}) \\
\ar[r] &  \dots}
\end{equation}
and consequently an equation
\begin{equation}\label{eqn_equality_of_euler_characteristics}
\chi_{\cS, T}(\ad^0 \overline{\rho}) = \chi(F_S/F, \ad^0 \overline{\rho}) - \sum_{v \in S} \chi(F_v, \ad^0 \overline{\rho}) + \sum_{v \in S - T} (\ell_v - h^0(F_v, \ad^0 \overline{\rho})) - 1 + \# T
\end{equation}
relating the Euler characteristics of these complexes (which are all finite; see \cite[Ch. 1, Corollary 2.3]{Mil06} and \cite[Ch. 1, Corollary 4.15]{Mil06}). We also define a group that plays the role of the dual Selmer group in this setting. Since $p > n$, there is a perfect duality of Galois modules
\begin{gather}\label{eqn_duality_of_adjoint_representation}
\ad^0 \overline{\rho} \times \ad^0 \overline{\rho}(1) \to k(\epsilon) \\
(X, Y) \mapsto \tr X Y. \nonumber
\end{gather}
In particular, this induces for each finite place $v$ of $F$ a perfect duality between the groups $H^1(F_v, \ad^0 \overline{\rho})$ and $H^1(F_v, \ad^0 \overline{\rho}(1))$ (by Tate duality; see \cite[Ch. 1, Corollary 2.3]{Mil06} again). We write $\cL_v^\perp \subset H^1(F_v, \ad^0 \overline{\rho}(1))$ for the annihilator under this pairing of $\cL_v$, and we define
\begin{equation} H^1_{\cS, T}(\ad^0 \overline{\rho}(1)) = \ker \left( H^1(F_S/F, \ad^0 \overline{\rho}(1)) \to \prod_{v \in S - T} H^1(F_v, \ad^0 \overline{\rho}(1)) / \cL_v^\perp \right). 
\end{equation}
\begin{proposition}\label{prop_presenting_global_deformation_ring} Let $\cS = (\overline{\rho}, \mu, S, \{ \Lambda_v \}_{v \in S}, \{ \cD_v \}_{v \in S})$ be a global deformation problem, and let $T \subset S$ be a non-empty subset. Suppose that $\Lambda_v = \cO$ for each $v \in S - T$. 
\begin{enumerate} \item  There is a surjection $A_\cS^T\llbracket X_1, \dots, X_g \rrbracket \to R_\cS^T$ of $A_\cS^T$-algebras, where $g = h^1_{\cS, T}(\ad^0 \overline{\rho})$. If for each $v \in S - T$, the ring $R_v$ is formally smooth over $\cO$, then the kernel can be generated by $r$ elements, where $r = h^2_{\cS, T}(\ad^0 \overline{\rho})$.
\item There is are equalities $h^2_{\cS, T}(\ad^0 \overline{\rho}) = h^1_{\cS, T}(\ad^0 \overline{\rho}(1))$ and 
\[ h^1_{\cS, T}(\ad^0 \overline{\rho}) = h^1_{\cS, T}(\ad^0 \overline{\rho}(1))+\sum_{v \in S - T} \left( \ell_v - h^0(F_v, \ad^0 \overline{\rho}) \right) - h^0(F, \ad^0 \overline{\rho}(1)) - \sum_{v | \infty} h^0(F_v, \ad^0 \overline{\rho}) - 1 + \# T. \]
\end{enumerate}
\end{proposition}
\begin{proof}
The global analogue of (\ref{eqn_defn_of_local_cocycle_space}) is the chain of isomorphisms
\begin{equation}
H^1_{\cS, T}(\ad^0 \overline{\rho}) \cong \Hom_k(\ffrm_{R_\cS^T}/(\ffrm_{R_\cS^T}^2, \ffrm_{A_{\cS}^T}), k) \cong \Hom_{\mathrm{CNL}_{\Lambda_T}}(R_{\cS^T}/(\ffrm_{A_{\cS}^T}), k[\epsilon]/(\epsilon^2)).
\end{equation}
We explain the first isomorphism. A $T$-framed lifting of $\overline{\rho}$ to $k[\epsilon]/(\epsilon^2)$ can be written in the form $((1+\epsilon \phi)\overline{\rho}, (1 + \epsilon \alpha_v)_{v \in T})$, with $\phi \in Z^1(F_S/F, \ad^0 \overline{\rho})$. The condition that it be of type $\cS$ is equivalent to the condition $\phi|_{G_{F_v}} \in \cL^1_v$ for $v \in S$. The condition that it give the trivial lifting at $v \in T$ is equivalent to the condition
\[ (1 - \epsilon \alpha_v) (1 + \epsilon \phi|_{G_{F_v}})\overline{\rho}|_{G_{F_v}} (1 + \epsilon \alpha_v) = \overline{\rho}|_{G_{F_v}}. \]
Putting it another way, the $T$-framed liftings of $\overline{\rho}$ to $k[\epsilon]/(\epsilon^2)$ are in bijection with the tuples $(\phi, \{ \alpha_v \}_{v \in T})$, where $\phi \in Z^1(F_S/F, \ad^0 \overline{\rho})$, $\alpha_v \in \ad \overline{\rho}$, and for each $v \in T$ we have the equality
\[ \phi|_{G_{F_v}} = (\ad \overline{\rho}|_{G_{F_v}} - 1) \alpha_v. \]
Two tuples $(\phi, \{ \alpha_v \}_{v \in T})$ and $(\phi', \{ \alpha'_v \}_{v \in T})$ give rise to strictly equivalent $T$-framed liftings if and only if there exists $t \in \ad \overline{\rho}$ satisfying
\begin{gather*}
\phi' = \phi + (1 - \ad \overline{\rho})t, \\
\alpha'_v = \alpha_v + t
\end{gather*}
for each $v \in T$. This completes the calculation of the tangent space. The rest of the first part is a standard argument in obstruction theory, which we omit.

For the second part, we recall that $H^i(F_S/F, \ad^0 \overline{\rho}) = 0$ if $i \geq 3$ (by \cite[Ch. 1, Theorem 4.10]{Mil06}, and since $p$ is odd), while Tate's local and global Euler characteristic formulae (see \cite[Ch. 1, Theorem 2.8]{Mil06} and \cite[Ch. 1, Theorem 5.1]{Mil06}, respectively) give 
\begin{gather*} \sum_{v \in S}\chi(F_v, \ad^0 \overline{\rho}) = n^2 [F : \bbQ], \\
\chi(F_S/F, \ad^0 \overline{\rho}) = n^2 [ F : \bbQ] - \sum_{v | \infty} h^0(F_v, \ad^0 \overline{\rho}),
\end{gather*}
hence 
\begin{equation}\label{eqn_taylor_euler_characteristic}
\chi_{\cS, T}(\ad^0 \overline{\rho}) = - \sum_{v | \infty} h^0(F_v, \ad^0 \overline{\rho}) + \sum_{v \in S - T} (\ell_v - h^0(F_v, \ad^0 \overline{\rho})) + 1 - \# T
\end{equation}
(use (\ref{eqn_equality_of_euler_characteristics})).
We now observe that there are exact sequences
\begin{equation*}
\xymatrix@R-2pc{& & H^1(F_S/F, \ad^0 \overline{\rho}) \ar[r] & \oplus_{v \in T} H^1(F_v, \ad^0 \overline{\rho}) \oplus_{v \in S-T} H^1(F_v, \ad^0 \overline{\rho})/\cL_v \\
 \ar[r] &  H^2_{\cS, T}(\ad^0 \overline{\rho}) \ar[r] & H^2(F_S/F, \ad^0 \overline{\rho}) \ar[r] & \oplus_{v \in S} H^2(F_v, \ad^0 \overline{\rho})  \\
\ar[r] &  H^3_{\cS, T}(\ad^0 \overline{\rho}) \ar[r] & 0}
\end{equation*}
and
\begin{equation*}
\xymatrix@R-2pc{ & & H^1(F_S/F, \ad^0 \overline{\rho}) \ar[r] & \oplus_{v \in T} H^1(F_v, \ad^0 \overline{\rho}) \oplus_{v \in S - T} H^1(F_v, \ad^0 \overline{\rho})/\cL_v \\
\ar[r] & H^1_{\cS, T}(\ad^0 \overline{\rho}(1))^\vee \ar[r] & H^2(F_S/F, \ad^0 \overline{\rho}) \ar[r] & \oplus_{v \in S} H^2(F_v, \ad^0 \overline{\rho}) \\
\ar[r] & H^0(F_S/F, \ad^0 \overline{\rho}(1))^\vee\ar[r] & 0.}
\end{equation*}
(The first sequence is part of (\ref{eqn_long_exact_sequence_in_taylor_cohomology}), while the second is part of the Poitou-Tate exact sequence (see \cite[Ch. 1, Proposition 4.10]{Mil06}).) Comparing these two exact sequences, we see obtain
\begin{gather*}
 h^2_{\cS, T}(\ad^0 \overline{\rho}) = h^1_{\cS, T}(\ad^0 \overline{\rho}(1)),\\
 h^3_{\cS, T}(\ad^0 \overline{\rho}) = h^0(F_S/F, \ad^0 \overline{\rho}(1)),
\end{gather*}
and so (\ref{eqn_taylor_euler_characteristic}) gives
\begin{equation*}
h^1_{\cS, T}(\ad \overline{\rho}) = h^1_{\cS, T}(\ad^0 \overline{\rho}(1)) - h^0(F_S/F, \ad^0 \overline{\rho}(1)) - \sum_{v | \infty} h^0(F_v, \ad^0 \overline{\rho}) + \sum_{v \in S - T} (\ell_v - h^0(F_v, \ad^0 \overline{\rho})) - 1 + \# T
\end{equation*}
(we have $h^0_{\cS, T}(\ad^0 \overline{\rho}) = 0$, since $T$ is non-empty by assumption). Re-arranging this equation completes the proof.
\end{proof}
We say that $\overline{\rho}$ is totally odd if for each complex conjugation $c \in G_F$, we have 
\[ \overline{\rho}(c) \sim \diag(\underbrace{1, \dots, 1}_a, \underbrace{-1, \dots, -1}_b), \]
with $|a - b| \leq 1$. If $F$ is totally complex, then this condition is empty!
\begin{corollary}
Suppose further that the following conditions hold.
\begin{enumerate}
\item The representation $\overline{\rho}$ is totally odd.
\item The representation $\overline{\rho}$ is not isomorphic to its twist $\overline{\rho} \otimes \epsilon$ by the cyclotomic character.
\item For each $v \in S - T$, $R_v$ is formally smooth over $\cO$ of dimension $n^2$.
\item $T$ is non-empty.
\end{enumerate}
Then $R_\cS^T$ is a quotient of a power series ring over $A_\cS^T$ in 
\[ g = h^1_{\cS, T}(\ad^0 \overline{\rho}(1)) - n(n-1)[F : \bbQ]/2 - l_0 -1 + \# T\]
variables, where $l_0$ is the `defect' (depending only on $F$ and $n$) defined in \S \ref{sec_completed_cohomology_setup} below. The kernel of this surjection can be generated by at most $r = h^1_{\cS, T}(\ad \overline{\rho}(1))$ elements. In particular we have:
\[ \dim R_\cS^T \geq \dim A_{\cS}^T + g - r = \dim A_{\cS}^T - n(n-1)[F : \bbQ]/2 - l_0 + 1 - \# T, \]
hence
\[ \dim R_\cS \geq (\dim A_{\cS}^T - (n^2 - 1)\# T) - n(n-1)[F : \bbQ]/2 - l_0. \]
\end{corollary}

\subsection{Local deformation problems}

\subsubsection{Ordinary deformations}\label{sec_ordinary_local_deformation_problem}

Let $v \in S_p$, and suppose that $\Lambda_v = \cO \llbracket \cO_{F_v}^\times(p) \times \dots \times \cO_{F_v}^\times(p) \rrbracket$ ($n-1$ times). We now define a functor of ordinary deformations. For simplicity, we assume that we are in the special case where $\overline{\rho}|_{G_{F_v}}$ is trivial and $[F_v : \bbQ_p] > n(n-1)/2 + 1$. In this case, we have the following result:
\begin{proposition}\label{prop_ordinary_lifting_rings}
There exists an $\cO$-flat, reduced quotient $R_v^\triangle$ of $R_v^\square$ satisfying the following conditions:
\begin{enumerate}
\item $R_v^\triangle$ defines a local deformation problem.
\item Let $R \in \CNL_{\Lambda_v}$ be a domain, and let $\overline{K}$ be an algebraic closure of $K = \Frac R$. Let $\rho$ be a lifting of $\overline{\rho}|_{G_{F_v}}$ corresponding to a homomorphism $f : R_v^\square \to R$. Then $f$ factors through the quotient map $R_v^\square \to R_v^\triangle$ if and only if there exists an increasing filtration
\begin{equation} 0 = \Fil^0_v \subset \Fil^1_v \subset \dots \subset \Fil^n_v = \overline{K}^n 
\end{equation}
of $\rho \otimes_R \overline{K}$ with the following property: each $\gr^i \Fil^\bullet_v = \Fil^i_v/\Fil^{i-1}_v$, $i = 1, \dots, n-1$, is 1-dimensional, and the corresponding character $\psi^i_v : G_{F_v} \to \overline{K}^\times$ satisfies $\psi^i_v|_{I_{F_v}} = \chi^i_v$, where $\chi^i_v$ is the pushforward of the universal character $I_{F_v} \to \Lambda_v^\times \to R^\times$.
\item The structural map $\Spec R_v^\triangle \to \Spec \Lambda_v$ induces a bijection on irreducible components. In particular, if there exists an embedding $F_v \hookrightarrow E$ then each irreducible component of $\Spec R_v^\triangle$ is geometrically irreducible. Each irreducible component of $\Spec R_v^\triangle$ has dimension $1 + (n^2 - 1) + [F_v : \bbQ_p]( n(n+1)/2 - 1)$.
\end{enumerate}
\end{proposition}
\begin{proof}
See \cite[\S 3.3.2]{Tho14}. We note that in \emph{loc. cit.} we consider the analogous situation where the determinant is not fixed, but it is easy to pass from this to the situation considered here.
\end{proof}

\subsubsection{Taylor--Wiles deformations}\label{sec_definition_of_taylor_wiles_deformations}

Suppose that $q_v \equiv 1 \text{ mod }p$, and $\overline{\rho}|_{G_{F_v}}$ is unramified. Let $\Delta_v = k(v)^\times(p)^{n-1}$ and $\Lambda_v = \cO[\Delta_v]$, and suppose that $\overline{\rho}(\Frob_v)$ has $n$ distinct eigenvalues $\gamma_{v, 1}, \dots, \gamma_{v, n} \in k$. We define $\cD_v^\text{TW}$ to be the functor of liftings over $R \in \CNL_{\Lambda_v}$
\[ r \sim C_1 \oplus \dots \oplus C_n, \]
for continuous characters $C_1, \dots, C_n : G_{F_v} \to R^\times$ satisfying the following conditions: for each $i = 1, \dots, n-1$, we have $(C_i \text{ mod }\ffrm_R)(\Frob_v) = \gamma_{v, i}$ and $C_i|_{I_{F_v}}$ agrees, on composition with the Artin map, with the $i^\text{th}$ canonical character $k(v)^\times(p) \to R^\times$. (This character exists because $R$ is a $\Lambda_v$-algebra, by assumption.) We observe that the functor $\cD_v^\text{TW}$ depends on the choice of ordering of the eigenvalues of $\overline{\rho}(\Frob_v)$. The functor $\cD_v^\text{TW}$ is represented by a formally smooth $\Lambda_v$-algebra.

Suppose that $\cS = (\overline{\rho}, \mu, S, \{ \Lambda_v \}_{v \in S}, \{ \cD_v \}_{v \in S})$ is a deformation problem. Let $Q$ be a set of places disjoint from $S$, such that for each $v \in Q$, $q_v \equiv 1 \text{ mod }p$ and $\overline{\rho}(\Frob_v)$ has distinct eigenvalues $\gamma_{v, 1}, \dots, \gamma_{v, n} \in k$. We refer to the tuple $(Q, (\gamma_{v, 1}, \dots, \gamma_{v, n})_{v \in Q})$ as a Taylor--Wiles datum, and define the augmented deformation problem
\[ \cS_Q = (\overline{\rho}, \mu, S \cup Q, \{ \Lambda_v \}_{v \in S} \cup \{ \cO[\Delta_v] \}_{v \in Q}, \{ \cD_v \}_{v \in S} \cup \{ \cD_v^\text{TW} \}_{v \in Q}). \]
Let $\Delta_Q = \prod_{v \in Q} \Delta_v = \prod_{v \in Q} k(v)^\times(p)$. Then $R_{\cS_Q}$ is naturally a $\cO[\Delta_Q]$-algebra. If $\fra_Q \subset \cO[\Delta_Q]$ is the augmentation ideal, then there is a canonical isomorphism $R_{\cS_Q}/(\fra_Q) \cong R_\cS$. 

\subsection{Auxiliary places}

We continue with the notation established at the beginning of \S \ref{sec_galois_deformation_theory}. In particular, $p = \mathrm{char }k > n$.
\begin{definition}\label{def_huge_image}
We say that a subgroup $H \subset \GL_n(k)$ is enormous if it satisfies the following conditions:
\begin{enumerate}
\item $H$ has no non-trivial $p$-power order quotient.
\item $H^0(H, \ad^0) = H^1(H, \ad^0) = 0$ (for the adjoint action of $H$).
\item For all simple $k[H]$-submodules $W \subset \ad^0$, we can find $h \in H$ with $n$ distinct eigenvalues, all lying in $k$, and an element $\alpha \in k$ such that $\alpha$ is an eigenvalue of $H$ and $\tr e_{h, \alpha} W \neq 0$. \(By definition, $e_{h, \alpha} \in M_n(k) = \ad$ is the unique $h$-equivariant projection onto the $\alpha$-eigenspace of $h$.\)
\end{enumerate}
\end{definition}
\begin{remark} \begin{enumerate} \item A enormous subgroup is big (cf. \cite[Definition 2.5.1]{Clo08}).
\item In the first version of this paper, we used the adjective `huge' instead of `enormous' to describe a subgroup satisfying the conditions of Definition \ref{def_huge_image}. We have made the change in order to be consistent with \cite{Cal13}, where the same definition is employed for a similar purpose (cf. \cite[\S 9.2]{Cal13}).
\end{enumerate}
\end{remark}
\begin{lemma}\label{lem_existence_of_taylor_wiles_primes}
Let $(\overline{\rho}, \mu, S, \{ \Lambda_v \}_{v \in S}, \{ \cD_v \}_{v \in S})$ be a global deformation problem, and let $T \subset S$. Suppose that $\overline{\rho}({G_{F(\zeta_p)}})$ is enormous, and let $q \geq h^1_{\cS, T}(\ad^0 \overline{\rho}(1))$. Then for every $N \geq 1$ there exists a Taylor--Wiles datum $(Q_N, (\gamma_{v, 1}, \dots, \gamma_{v, n})_{v \in Q_N})$ satisfying the following conditions:
\begin{enumerate}
\item $\# Q_N = q$.
\item For each $v \in Q_N$, $q_v \equiv 1 \text{ mod }p^N$.
\item $h^1_{\cS_{Q_N}, T}(\ad^0 \overline{\rho}(1)) = 0$.
\end{enumerate}
\end{lemma}
\begin{proof}
Fix $N \geq 1$. By the usual arguments (cf. \cite[Proposition 2.5.9]{Clo08}), it is enough to find a set $Q$ of places of $F$, disjoint from $S$, and satisfying the following conditions:
\begin{itemize}
\item Each place $v \in Q$ splits in $F(\zeta_{p^N})$.
\item If $v \in Q$, then $\overline{\rho}(\Frob_v)$ has distinct eigenvalues $\gamma_{v, 1}, \dots, \gamma_{v, n}$.
\item The natural map $H^1_{\cS, T}( \ad^0 \overline{\rho}(1)) \to \oplus_{v \in Q} H^1(F_v, \ad^0 \overline{\rho}(1))$ is injective.
\end{itemize}
By the Chebotarev density theorem, and induction, it even suffices to find for each cocycle $\varphi$ representing a non-zero element of $H^1_{\cS, T}(\ad^0 \overline{\rho}(1))$, an element $\sigma \in G_{F(\zeta_{p^N})}$ satisfying the following conditions:
\begin{itemize}
\item $\overline{\rho}(\sigma)$ has distinct eigenvalues $\alpha_1, \dots, \alpha_n$.
\item There is $i = 1, \dots, n$ such that $\tr e_{\sigma, \alpha_i} \varphi(\sigma) \neq 0$.
\end{itemize}
Suppose that $\varphi$ represents a non-zero element of $H^1(F_S/F, \ad^0 \overline{\rho}(1))$. It follows from the definition of `enormous' that the image of $\varphi$ in $H^1(F(\zeta_{p^N}), \ad^0 \overline{\rho}(1))$ is non-zero, and this image is represented by a $G_F$-equivariant homomorphism $f : F(\zeta_{p^N}) \to \ad^0 \overline{\rho}(1)$. Let $\sigma_0 \in G_{F(\zeta_{p^N})}$ be any element such that $\overline{\rho}(\sigma_0)$ has distinct eigenvalues $\alpha_1, \dots, \alpha_n$. If $\tr e_{\sigma_0, \alpha_i} \varphi(\sigma_0) \neq 0$ for some $i$, then we're done. Otherwise, we can assume that $\tr e_{\sigma_0, \alpha} \varphi(\sigma_0) = 0$ for each $i$. 

Using that $H^1(\overline{\rho}(G_{F(\zeta_{p^N})}, \ad^0 \overline{\rho}(1)) = 0$, we see that the image of $\varphi$ in $H^1(L, \ad^0 \overline{\rho}(1))$ is non-zero, where $L/F(\zeta_{p^N})$ is the extension cut out by $\overline{\rho}$. We thus choose any element $\tau \in G_L$ with $\tr e_{\sigma_0, \alpha_i} \varphi(\tau) \neq 0$, and set $\sigma = \tau \sigma_0$. Since $\overline{\rho}(\sigma) = \overline{\rho}(\sigma_0)$ and $\varphi(\sigma) = \varphi(\sigma_0) + \varphi(\tau)$, this element does the job. 
\end{proof}
We now construct some examples of enormous subgroups by induction from a character. Let $M/F$ be a degree $n$ Galois extension, and let $\overline{\chi} : G_M \to k^\times$ be a continuous character. Suppose that for all $g \in G_F - G_M$, we have $\overline{\chi}^g \neq \overline{\chi}$. Then $\overline{\rho} = \Ind_{G_M}^{G_F} \overline{\chi}$ is absolutely irreducible. Let $H$ denote the image of $\overline{\rho}$. Then $H$ has order prime to $p$, and $\overline{\rho}$ will be enormous if and only if the following condition is satisfied:
\begin{itemize}
\item For all simple $k[H]$-submodules $W \subset \ad^0$, we can find $h \in H$ with $n$ distinct eigenvalues and $\alpha \in k$ such that $\alpha$ is an eigenvalue of $H$ and $\tr e_{h, \alpha} W \neq 0$. 
\end{itemize}
\begin{lemma}\label{lem_generic_induced_representations_are_huge}
With assumptions as above, choose a decomposition $G_F = \sqcup_{i=1}^n g_i G_M$, and make the following further assumptions:
\begin{enumerate}
\item For all pairs $(i, j) \neq (i', j')$ with $i \neq j$, $i' \neq j'$, we have $\chi^{g_i}/\chi^{g_j} \neq \chi^{g_{i'}}/\chi^{g_{j'}}$.
\item For each $i = 1, \dots, n$, there is an element $h \in g_i G_M$ such that $\overline{\rho}$ has distinct eigenvalues.
\end{enumerate} Then $\overline{\rho}$ has enormous image.
\end{lemma}
\begin{proof}
A basis of $\Ind_{G_M}^{G_F} \overline{\chi} = k[G_F] \otimes_{k[G_M]} k(\overline{\chi})$ is given by the elements $e_1, \dots, e_n$, $e_i = g_i \otimes 1$. We can decompose $\ad^0 \overline{\rho} = M_0 \oplus M_1$, where $M_0$ is the submodule of diagonal matrices and $M_1$ is the submodule of matrices with all diagonal entries equal to 0. If $W \subset \ad^0 \overline{\rho}$ is a simple submodule, then either $W \subset M_0$ or $W \subset M_1$. 

In the case $W \subset M_0$, let $h \in G_M$ be an element such that $\overline{\rho}(h)$ has distinct eigenvalues. It is then clear that $\tr e_{h, \alpha} W \neq 0$ for some choice of eigenvalue $\alpha$ of $H$. In the case $W \subset M_1$, we see that $W$ decomposes as $k[G_M]$-module into a sum of pairwise non-isomorphic 1-dimensional representations. It follows that we can find an element $X \in W$ such that $X$ has exactly one non-zero entry (with respect to the basis $e_1, \dots, e_n$). We will show that there exists $h \in G_F$ with pairwise distinct eigenvalues and such that $\tr h X \neq 0$. This will complete the proof of the lemma. Indeed, the element $h$ generates a commutative subalgebra $k[h] \subset M_n(k)$ of dimension $n$ over $k$, and there is an isomorphism $k[h] \cong \prod_\alpha k$, the product being over the eigenvalues $\alpha \in k$ of $h$; the idempotents corresponding to this decomposition are exactly the equivariant projectors $e_{h, \alpha}$ onto the eigenspaces of $h$. If $\tr h X \neq 0$, then $\tr e_{h, \alpha} W \neq 0$ for some $\alpha$, and this is what we need to show.

Consider an element of the form $h = g_j \sigma$, for some $\sigma \in G_M$. Then we have $h e_i = g_j \sigma g_i \otimes 1 = g_k \otimes \overline{\chi}(g_k^{-1} g_j g_i) \overline{\chi}(g_i^{-1} \sigma g_i)$ for some $k$ uniquely determined by $i$ and $j$. In particular, the matrix $\overline{\rho}(h)$ is in the normalizer of the diagonal matrix torus of $\GL_n(k)$. Since $G_F$ acts transitively on the set $G_F / G_M$, we can choose $j$ so that $g_j X$ has (exactly one) non-zero diagonal entry. Since $\overline{\rho}(\sigma)$ is diagonal, we have 
\[ \tr \overline{\rho}(g_j \sigma) X = \tr \overline{\rho}(g_j \sigma g_j^{-1}) \overline{\rho}(g_j) X \neq 0 \]
for all $\sigma \in G_M$. By assumption, we can choose $\sigma \in G_M$ so that $h = g_j \sigma$ has distinct eigenvalues. This completes the proof.
\end{proof}
\begin{lemma}\label{lem_existence_of_generic_characters}
With assumptions as above, let $\bbG_m$ denote the standard 1-dimensional torus over $\overline{\bbF}_p$, and let $Z \subset \bbG_m^n$ denote a proper Zariski closed subset. Choose a decomposition $G_F = \sqcup_{i=1}^n g_i G_M$. We assume that $g_1 = 1$. Then for all $N \geq 1$, we can find the following:
\begin{enumerate}
\item A prime $l > N$.
\item A continuous character $\overline{\chi} : G_M \to \overline{\bbF}_p^\times$ of $l$-power order.
\item An element $\sigma \in G_M$ such that $(\overline{\chi}^{g_1}(\sigma), \dots, \overline{\chi}^{g_n}(\sigma)) \not\in Z$.
\end{enumerate}
\end{lemma}
\begin{proof}
Fix a place $v$ of $F$ which splits in $M$ and a rational prime $l> \sup(N, p)$. Let $w_1$ be a place of $M$ above $v$, and let $w_i = g_i(w_1)$. The set of $l$-power roots of unity in $\bbG_m^n(\overline{\bbF}_p)$ is Zariski dense. We can therefore choose $l$-power roots of unity $\zeta_1, \dots, \zeta_n \in \overline{\bbF}_p^\times$ such that $(\zeta_1, \dots, \zeta_n) \not\in Z$. By the Grunwald-Wang theorem, we can find a character $\overline{\chi} : G_M \to \overline{\bbF}_p^\times$ unramified at $w_1, \dots, w_n$, of $l$-power order, and such that $\overline{\chi}(\Frob_{w_i}) = \zeta_i$. 

We then have $\overline{\chi}^{g_i}(\Frob_{w_1}) = \overline{\chi}(g_i \Frob_{w_1} g_i^{-1}) = \overline{\chi}(\Frob_{w_i}) = \zeta_i$. The lemma follows on taking $\sigma = \Frob_{w_1}$.
\end{proof}
\begin{lemma}\label{lem_many_huge_characters}
Let $M/F$ be a degree $n$ Galois extension. Then there exist infinitely many choices of $\overline{\chi} : G_M \to \overline{\bbF}_p^\times$ such that $\overline{\rho} = \Ind_{G_M}^{G_F} \overline{\chi}$ has enormous image.
\end{lemma}
\begin{proof}
We use the previous 2 lemmas. Choose a decomposition $G_F = \sqcup_{i=1}^n g_i G_M$ with $g_1 = 1$.
Let $Z \subset \bbG_m^n$ be the set of elements $(x_1, \dots, x_n)$ such that there exist disjoint non-empty subsets $I, J \subset \{ 1, \dots, n \}$ such that $\prod_{i \in I} x_i = \prod_{j \in J} x_j$. By Lemma \ref{lem_existence_of_generic_characters}, we can find (after possibly enlarging $k$) a prime $l > p$, a character $\overline{\chi} : G_M \to k^\times$ of $l$-power order, and an element $\sigma \in G_M$ such that $(\overline{\chi}^{g_1}(\sigma), \dots, \overline{\chi}^{g_n}(\sigma)) \not\in Z$. In particular, the characters $\overline{\chi}^{g_1}, \dots, \overline{\chi}^{g_n}$ are pairwise distinct, so $\overline{\rho} =  \Ind_{G_M}^{G_F} \overline{\chi}$ is absolutely irreducible, and their ratios are pairwise distinct, so the first condition of Lemma \ref{lem_generic_induced_representations_are_huge} is satisfied.

Let $L/F$ be the extension cut out by $\overline{\rho}$. Then $\Gal(L/F)$ sits in a short exact sequence
\[ \xymatrix@1{ 1 \ar[r] & \Gal(L/M) \ar[r] & \Gal(L/F) \ar[r] & \Gal(M/F) \ar[r] & 1,} \]
with $\Gal(L/M)$ abelian of $l$-power order. This extension is classified by an element of the cohomology group $H^2(\Gal(M/F), \Gal(L/M))$, which is trivial (since $l > n = \# \Gal(M/F)$). It follows that this extension is split. Choose a splitting $\Gal(M/F) \hookrightarrow \Gal(L/F)$; multiplying the elements $g_i$ on the right by elements of $G_M$, we can assume that they all lie in the image of this map. This does not change the characters $\overline{\chi}^{g_i}$, or the condition on the element $\sigma$.

We again use the basis $e_1, \dots, e_n$ of $\overline{\rho}$, where $e_i = g_i \otimes 1$. The action of an element $g_j \sigma \in g_j G_M$ is now given by $g_j \sigma e_i = g_j \sigma g_i \otimes 1 = g_j g_i \otimes \overline{\chi}(g_i^{-1} \sigma g_i)$. Let $I_1, \dots, I_s \subset \{ 1, \dots n \}$ be the cycle decomposition of $g_j$ (i.e.\ the orbits of $g_j$ on $\{ g_1, \dots, g_n \}$ by left multiplication). It follows from our construction that the elements $\beta_k = \prod_{i \in I_k} \overline{\chi}^{g_i}(\sigma)$, $k = 1, \dots, s$ are pairwise distinct. This implies that the element $\overline{\rho}(g_j \sigma)$ has distinct eigenvalues; indeed, these eigenvalues are exactly the distinct $(\# I_k )^\text{th}$ roots of the elements $\beta_k$, $k = 1, \dots, s$. It follows that the second assumption of Lemma \ref{lem_generic_induced_representations_are_huge} is satisfied, and we deduce that $\overline{\rho}$ has enormous image, as required.
\end{proof}

We now change notation slightly, with a view to the applications in \S \ref{sec_potential_automorphy_and_leopoldt} below. Let $K/\bbQ$ be a Galois totally real field of degree $n$, and let $F$ be an imaginary CM field which is linearly disjoint over $\bbQ$ from $K$. Set $M = K \cdot F$. Then $M/F$ is a Galois extension of imaginary CM fields of degree $n$. 

\begin{lemma}\label{lem_existence_of_odd_huge_representations}
Let $F^+ \subset F$ be the maximal totally real subfield, and let $c \in G_{F^+}$ be a choice of complex conjugation. There exists a continuous character $\overline{\chi} : G_M \to k^\times$ such that $\overline{\rho} = \Ind_{G_M}^{G_F} \overline{\chi}$ satisfies the following conditions:
\begin{enumerate}
\item $\overline{\rho}$ has enormous image. In particular, $\overline{\rho}$ is absolutely irreducible.
\item The pair $(\overline{\rho}, \epsilon^{1-n} \delta_{F/F^+}^n)$ is polarized, in the sense of \cite[\S 2.1]{Bar14}. In other words, there exists a non-degenerate symmetric bilinear form  $\langle \cdot, \cdot \rangle : \overline{\rho} \times \overline{\rho} \to k$ such that for all $\sigma \in G_F$, $x, y \in \overline{\rho}$, we have $\langle \overline{\rho}(\sigma)x, \overline{\rho}(\sigma^c) y \rangle = \epsilon^{1-n}(\sigma) \langle x, y \rangle$.
\end{enumerate}
\end{lemma}
\begin{proof}
The fields $F$ and $K \cdot F^+$ are linearly disjoint over $F^+$, from which it follows that $\Gal(M/F^+) \cong \Gal(F/F^+) \times \Gal(K/\bbQ)$. The maximal totally real subfield of $M$ is $M^+ = F^+ \cdot K$, and we have $c \in G_{M^+} \subset G_{F^+}$. Let $l > p$ be a rational prime which splits in $M$, let $v$ be a place of $F$ above $l$, and let $w = w_1, \dots, w_n$ be the places of $M$ above $v$. Choose a surjection $I_{M_{w_1}}^\text{ab} \to \bbZ_l$, and let $\sigma_1 \in I_{M_{w_1}}$ be an element generating the image of this surjection. Let $1 = g_1, \dots, g_n$ be a transversal of $G_M$ in $G_F$, and let $\sigma_i = \sigma^{g_i} \in G_M$. 

By \cite[Lemma A.2.4]{Bar14}, for any tuple $(\zeta_1, \dots, \zeta_n) \in \bbG_m^n$ of $l$-power roots of unity in $\overline{\bbF}_p^\times$, we can find a continuous character $\overline{\theta} : G_M \to \overline{\bbF}_p^\times$ such that $\overline{\theta} \overline{\theta}^c = 1$ and $\overline{\theta}(\sigma_i) = \zeta_i$, for each $i = 1, \dots, n$. After possibly replacing $\overline{\theta}$ by $\overline{\theta}^M$, for some integer $M$ prime to $l$, we can assume further that $\overline{\theta}$ has $l$-power order. Arguing as in the proof of Lemma \ref{lem_many_huge_characters}, we see that we can even choose $\overline{\theta}$ so that $\overline{\rho} = \Ind_{G_M}^{G_F} \overline{\theta}$ is absolutely irreducible and has enormous image, and the elements $\overline{\theta}(\sigma_i)$ are pairwise distinct. We can also find a character $\overline{\omega} : G_F \to k^\times$ such that $\overline{\omega} \overline{\omega}^c = \epsilon^{1-n}$. Then the representation $\overline{\rho} = \overline{\omega} \otimes \Ind_{G_M}^{G_F} \overline{\theta} \cong \Ind_{G_M}^{G_F} \overline{\omega} \overline{\theta}$ has enormous image, being a twist of a representation with enormous image, so we take $\overline{\chi} = \overline{\theta} \overline{\omega}$.

It remains to check that the representation $(\overline{\rho}, \epsilon^{1-n} \delta_{F/F^+}^n)$ is polarized. Since $\overline{\theta} \overline{\theta}^c = 1$, we have isomorphisms
\[ ( \Ind_{G_M}^{G_F} \overline{\theta} )^c \cong \Ind_{G_M}^{G_F} \overline{\theta}^{-1} \cong ( \Ind_{G_M}^{G_F} \overline{\theta} ), \]
so there exists a non-degenerate bilinear form $\langle \cdot, \cdot \rangle : \Ind_{G_M}^{G_F} \overline{\theta} \times \Ind_{G_M}^{G_F} \overline{\theta} \to k$, unique up to scalar, such that $\langle \sigma x, \sigma^c y \rangle = \langle x, y \rangle$ for all $x ,y \in  \Ind_{G_M}^{G_F} \overline{\theta}$. To finish the proof of the lemma, it is enough to show that this pairing is in fact symmetric (and not anti-symmetric). 

Let us write $V = \Ind_{G_M}^{G_F} \overline{\theta}$, and let $L/M$ be the extension cut out by $\overline{\theta}$. Then we can find an isomorphism $\Gal(L/F) \cong \Gal(L/M) \rtimes \Gal(M/F)$, and we choose coset representatives $\Gal(L/F) = \sqcup_{j=1}^n g_j \Gal(M/F)$ with $g_j \in \Gal(M/F)$. We fix a basis $e_1, \dots, e_n$ of $V$ with $e_j = g_j \otimes 1$.

If $\sigma \in G_M$, then $\langle \sigma e_i, \sigma^c e_j \rangle = \overline{\theta}^{g_i}(\sigma) \overline{\theta}^{g_jc}(\sigma) \langle e_i, e_j \rangle = \overline{\theta}^{g_i}/\overline{\theta}^{g_j}(\sigma) \langle e_i, e_j \rangle$ (because $\overline{\theta}^c = \overline{\theta}^{-1}$ and $\Gal(M/F^+) \cong \Gal(M/F) \times \Gal(F/F^+)$). The character $\overline{\theta}^{g_i}/\overline{\theta}^{g_j}$ is non-trivial if $i \neq j$, so it follows that $\langle e_i, e_j \rangle = 0$ if $i \neq j$. This is implies that the non-degenerate bilinear form $\langle \cdot, \cdot \rangle$ is in fact symmetric, and completes the proof of the lemma.
\end{proof}

\section{Interlude on $q$-adic Iwahori Hecke algebras}\label{sec_q_adic_hecke_algebras}

Let $l$ be a prime, and let $F$ be a finite extension of $\bbQ_l$. Let $p \neq l$ be an odd prime, and let $E$ be a finite extension of $\bbQ_p$. Let $\cO_F$ be the ring of integers of $F$, with maximal ideal $\ffrm_F$ and residue field $k_F$. Let $\cO$ be the ring of integers of $E$, with maximal ideal $\lambda$ and residue field $k$. Let $q = \# k_F$. 

In this section, we consider some aspects of the smooth representation theory of the group $G = \GL_n(F)$ over $\cO$ under the assumption $q \equiv 1 \text{ mod }p$ and $p > n$. The results of this section will be used in \S \ref{sec_ordinary_completed_cohomology}.  By definition, a smooth $\cO[G]$-module $M$ is an $\cO[G]$-module such that every element of $M$ is fixed by an open compact subgroup $U \subset G$. An admissible $\cO[G]$-module $M$ is a smooth $\cO[G]$-module such that for every open compact subgroup $U \subset G$, $M^U$ is a finite $\cO$-module. We suppose fixed a square root $q^{1/2}$ of $q$ in $\cO$; this exists, by Hensel's lemma. 

We first introduce some important subgroups of $G$. Let $\cK = \GL_n(\cO_F)$, and let $\cB \subset \cK$ denote the standard (upper-triangular mod $\ffrm_F$) Iwahori subgroup. Then $[ \cK: \cB]\equiv n! \text{ mod }p$ lies in $\cO^\times$. We write $\cH_\cB$ for the convolution algebra of compactly supported $\cB$-biinvariant functions $f : G \to \cO$, identity element given by $e_\cB = [\cB]$ (i.e.\ the characteristic function of $\cB$). Then $\cH_\cB = \cH(G, \cB) \otimes_\bbZ \cO$, in the notation of \S \ref{sec_hecke_operators} below. More generally, if $U$ is any open compact subgroup of $G$ then we write $\cH_U$ for the convolution algebra of compactly supported $\cB$-biinvariant functions $f : G \to \cO$, identity element given by $e_U = [U]$.

We write $T \subset B \subset G$ for the standard maximal torus and Borel subgroup, respectively. We write $\Phi \subset X^\ast(T)$ for the set of roots, and $R \subset X^\ast(T)$ for the set of simple roots given by $R = \{ \alpha_1, \dots, \alpha_{n-1} \}$, where $\alpha_i(t_1, \dots, t_n) = t_i/t_{i+1}$. Then the Weyl group $W = N_G(T)/T$ is generated by the simple reflections $s_{\alpha_i} = s_i$, $i = 1, \dots, n-1$.

We now introduce the Bernstein presentation of the algebra $\cH_B$, as exposed in \cite{Hai10}. The Bernstein presentation is an algebra isomorphism
\[ \cH_\cB \cong \cO[X_\ast(T)] \widetilde{\otimes}_\cO \cO[\cB \backslash \cK /\cB], \]
where the terms are as follows:
\begin{itemize}
\item $\cO[X_\ast(T)]$ is the group algebra of the free abelian group $X_\ast(T)$. The embedding $\cO[X_\ast(T)] \hookrightarrow \cH_B$ is defined as follows. If $\lambda \in X_\ast(T)_+$ is a dominant cocharacter, then we send $e_{\lambda} \mapsto q^{-l(\lambda)/2} [\cB \lambda(\varpi_F) \cB]$, where $l$ is the usual length function on the extended affine Weyl group. This defines a homomorphism $\cO[X_\ast(T)_+] \to \cH_\cB$ which extends uniquely to an embedding $\cO[X_\ast(T)] \hookrightarrow \cH_B$. If $\lambda \in X_\ast(T)$ is any element, then we write $e_\lambda \in \cO[X_\ast(T)]$ for the corresponding group algebra element and $\theta_{\lambda}$ for its image in $\cH_B$.
\item $\cO[\cB \backslash \cK / \cB]$ is the convolution algebra of $\cB$-bi-invariant functions $f : \cK \to \cO$, with its natural structure as subalgebra of $\cH_B$. If $w \in W = N_G(T)/T \cong S_n$, then we write $T_w = [\cB \dot{w} \cB]$, where $\dot{w} \in \cK$ is any representative of $w$.
\item The twisted tensor product $\widetilde{\otimes}$ is the usual tensor product as $\cO$-modules, with a twisted multiplication law characterized on basis elements by the equality (for simple $s_\alpha \in W$, $\lambda \in X_\ast(T)$):
\begin{equation}\label{eqn_bernstein_presentation} T_{s_\alpha} \theta_\lambda = \theta_{s_\alpha(\lambda)} T_{s_\alpha} + (q-1) \frac{\theta_{s_\alpha(\lambda)} - \theta_\lambda}{1 - \theta_{-\alpha^\vee}}. 
\end{equation}
The fraction, a priori an element of $\Frac \cO[X_\ast(T)]$, in fact lies in $\cO[X_\ast(T)]$.
\end{itemize}
We can use this presentation to deduce the following result. (Fabian Januszewski has pointed out to us that it could also be deduced more directly from the Iwahori--Matsumoto presentation.)
\begin{lemma}\label{lem_isomorphism_of_mod_p_hecke_algebra}
There is an isomorphism
\[ \cH_\cB \otimes_\cO k \cong k[X_\ast(T) \rtimes W]. \]
\end{lemma}
\begin{proof}
We first observe that the natural $k$-vector space map $k[W] \to k[\cB \backslash \cK / \cB]$ given on basis elements by $w \mapsto T_w$ is an algebra isomorphism. This can be deduced from studying the presentations of each algebra: $W$ is the free group on generators $s_\alpha$, $\alpha \in R$, subject to the relations $(s_\alpha s_\beta)^{m_{\alpha \beta}} = 1$, for appropriate coefficients $m_{\alpha \beta} \in \{ 1, 2, \infty\}$, while $k[\cB \backslash \cK / \cB]$ can be presented as the free algebra on elements $T_{i} = T_{s_i}$, $i = 1, \dots, n-1$, subject to the relations 
\begin{gather*}
T_{i}^2 = q + (q-1)T_i\text{ } (i = 1, \dots, n-1),\\
T_i T_{i+1} T_i = T_{i+1} T_i T_{i+1}\text{ } (i = 1, \dots, n-2) \\
T_i T_j = T_j T_i\text{ }(i, j = 1, \dots, n-1 \text{ and }|i - j|\geq 2).
\end{gather*}
Using the fact that $q = 1$ in $k$, we see that these two sets of relations are the same. Together with the relation (\ref{eqn_bernstein_presentation}), this shows that the map $k[X_\ast(T) \rtimes W] \to \cH_B \otimes_\cO k$, $e_\lambda w \mapsto \theta_\lambda \cdot T_w$, is an algebra isomorphism.
\end{proof}
\begin{lemma}\label{lem_splitting_of_inclusion_map_in_taylor_wiles_set_up}
Let $M$ be a smooth $\cO[G]$-module. Then the natural inclusion $M^\cK \subset M^\cB$ is canonically split, namely by the map $x \mapsto [\cK : \cB]^{-1} [\cK] x$.
\end{lemma}
\begin{proof}
The Hecke operator $[\cK] \in \cH_\cB$ induces the natural trace map $M^\cB \to M^\cK$. The composite $M^\cK \hookrightarrow M^\cB \rightarrow M^\cK$ is therefore multiplication by the index $[\cK : \cB] \in \cO^\times$.
\end{proof}
If $M$ is a smooth $k[G]$-module, then $M^\cB$ is a $k[W]$-module, by Lemma \ref{lem_isomorphism_of_mod_p_hecke_algebra}, and Lemma \ref{lem_splitting_of_inclusion_map_in_taylor_wiles_set_up} shows that $M^\cK$ is in fact the submodule of $M^\cB$ of $W$-invariants. 

Let $t_1, \dots, t_n \in \cO[X_\ast(T)]$ be the elements corresponding to the standard basis of $X_\ast(T)$. For each $i = 1, \dots, n$, let $T^i$ denote the standard unramified Hecke operator in $\cH_K$:
\[ T^i = \left[ \cK \diag(\underbrace{\varpi_F, \dots, \varpi_F}_i, \underbrace{1, \dots, 1}_{n-i}) \cK \right]. \]
Then $\cO[X_\ast(T)]^W$ is the center of $\cH_\cB$, and there is an isomorphism $\cO[X_\ast(T)]^W \cong \cH_K$ given by $x \mapsto [\cK] x$ (see \cite[\S 4.6]{Hai10}). This isomorphism sends the symmetric polynomial $e_i(t_1, \dots, t_n)$ to the element $q^{i(n-i)/2} T^i$.
\begin{lemma}\label{lem_unramified_hecke_eigenvalues_determine_eigenspaces}
Let $M$ be a $\cH_\cB$-module which is also a finite-dimensional $k$-vector space. Suppose that $[\cK]M \neq 0$ and that there are pairwise distinct elements $\gamma_1, \dots, \gamma_n \in k$ such that $T^i$ acts as the scalar $e_i(\gamma_1, \dots, \gamma_n)$ on $[\cK]M$. Then for each $w \in W$, the maximal ideal $\ffrm_w = (t_1 - \gamma_{w(1)}, \dots, t_n - \gamma_{w(n)}) \subset k[X_\ast(T)]$ is in the support of $M$.
\end{lemma}
\begin{proof}
Let $\frn = (e_1(t_1, \dots, t_n) - e_1(\gamma_1, \dots, \gamma_n), \dots, e_n(t_1, \dots, t_n) - e_n(\gamma_1, \dots, \gamma_n)) \subset k[X_\ast(T)]^W$. Then $[\cK] M \subset M[\frn]$, hence $M_\frn \neq 0$. Since $M_\frn = \oplus_{w \in W} M_{\ffrm_w}$, it follows that there exists $w \in W$ such that $M_{\ffrm_w} \neq 0$. Since $W$ acts on the set of maximal ideals of $k[X_\ast(T)]$ which are in the support of $M$, it follows that $M_{\ffrm_w} \neq 0$ for all $w \in W$. This shows the lemma.
\end{proof}
\begin{lemma}\label{lem_hecke_module_with_regular_infinitesimal_character}
Let $M$ be a $\cH_\cB$-module which is also a finite-dimensional $k$-vector space. Suppose that for each maximal ideal $\frn \subset k[X_\ast(T)]^W$ in the support of $M$, the polynomial 
\[ \sum_{i=0}^n (-1)^i q^{i(n-i)/2} e_i(t_1, \dots, t_n) X^{n-i} \in k[X_\ast(T)]^W[X] \]
has distinct roots modulo $\frn$. Then $[\cK]M \neq 0$. If there is a unique maximal ideal  $\frn \subset \cO[X_\ast(T)]^W$ in the support of $M$, then for each maximal ideal $\ffrm \subset k[X_\ast(T)]$ in the support of $M$, the maps
\[ k[W] \otimes_k M_\ffrm \to M,\text{ } w \otimes x \mapsto w \cdot x \]
and
\[ M_\ffrm \to [\cK] M, \text{ } x \mapsto [\cK] \cdot x \]
are isomorphisms.
\end{lemma}
\begin{proof}
After possibly enlarging $k$, we can assume that the roots $\gamma_1, \dots, \gamma_n$ of the polynomial 
\[ \sum_{i=0}^n (-1)^i q^{i(n-i)/2} e_i(t_1, \dots, t_n) X^{n-i} \text{ mod }\frn \in k[X] \]
 lie in $k$. Then the maximal ideal $\ffrm = (t_1 - \gamma_1, \dots, t_n - \gamma_n) \subset k[X_\ast(T)]$ lies in the support of $M$, hence $M_\ffrm \neq 0$. The operator $[\cK] = \sum_{w \in W} w$ defines an injection $M_\ffrm \hookrightarrow \oplus_{w \in W} M_{\ffrm_w}$, which shows that $[\cK] M \neq 0$. 

If $\frn$ is the unique maximal ideal of $k[X_\ast(T)]$ in the support of $M$, then every maximal ideal of $k[X_\ast(T)]$ in the support of $M$ has the form $\ffrm_w$, for some $w \in W$, and $M = \oplus_{w \in W} M_{\ffrm_w} = \oplus_{w \in W} w \cdot M_\ffrm$. The rest of the lemma now follows.
\end{proof}

\section{Ordinary completed cohomology}\label{sec_ordinary_completed_cohomology}

\subsection{Set-up}\label{sec_completed_cohomology_setup}

Let $F$ be a number field, and let $G = \GL_{n, F}$, a reductive group over $F$. We introduce the following notation:
\begin{itemize}
\item $G_\infty = G(F \otimes_\bbQ \bbR)$, and $\frg$ is the complexified Lie algebra of $G_\infty$, viewed as a real Lie group.
\item $K_\infty \subset G_\infty$ is a maximal compact subgroup.
\item $G^\infty = G(\bbA_F^\infty)$. Thus $G(\bbA_F) = G^\infty \times G_\infty$.
\item $Z \subset G$ is the center, $Z_\infty = Z(F \otimes_\bbQ \bbR)$, and $Z^\infty = Z(\bbA_F^\infty)$.
\end{itemize}
If $U \subset G^\infty$ is an open compact subgroup, then we define
\begin{equation}
X_U = G(F) \backslash G(\bbA_F) / U Z_\infty K_\infty.
\end{equation}
If $g \in G^\infty$, then $\Gamma_{U, g} = g U g^{-1} \cap G(F)$ is a discrete subgroup of $G_\infty$. Let $\overline{\Gamma}_{U, g}$ denote the image of $\Gamma_{U, g}$ in $G_\infty/Z_\infty$. Then $\overline{\Gamma}_{U, g}$ is a discrete subgroup of $G_\infty/Z_\infty = \PGL_n(F \otimes_\bbQ \bbR)$, and acts properly discontinuously on the quotient $G_\infty/K_\infty Z_\infty$. (The class of arithmetic subgroups is invariant under surjective morphisms of reductive groups.)

If $U \subset G^\infty$ is an open compact subgroup of the form $U = \prod_v U_v$, we say that $U$ is a good subgroup if it satisfies the following conditions:
\begin{itemize}
\item For each $g \in G(\bbA_F^\infty)$, the group $\Gamma_{U,g }$ is neat (hence $\overline{\Gamma}_{U, g}$ is neat, and torsion-free).
\item For each finite place $v$ of $F$, $U_v \subset \GL_n(\cO_{F_v})$.
\end{itemize}
In this connection, we have the following well-known lemma.
\begin{lemma}\label{lem_good_subgroups_act_without_fixed_points}
\begin{enumerate}
\item Good subgroups exist.
\item Let $U \subset G^\infty$ be a good subgroup. Then for all $g \in G^\infty$, the group $\overline{\Gamma}_{U, g}$ acts freely on $X = G_\infty/K_\infty Z_\infty$, and $X_U$ is endowed with the structure of smooth manifold.
\item Let $V = \prod_v V_v$ be an open compact subgroup of the good subgroup $U$. Then $V$ is also good.
\item Let $V = \prod_v V_v$ be a normal open compact subgroup of the good subgroup $U$. Then the map $X_V \to X_U$ is a Galois covering space, with Galois group $U/(V \cdot ( Z(F) \cap U ))$. (The intersection $Z(F) \cap U$ is taken inside $G^\infty$.)
\end{enumerate}
\end{lemma}
We now discuss the numerology of the manifolds $X_U$. Let $r_1,$ $r_2$ denote the number of real and complex places of $F$, respectively. We then have
\begin{equation} \dim X_U = d = \frac{r_1}{2}(n-1)(n+2) + r_2(n^2 - 1). 
\end{equation}
The so-called defect is (where the rank denotes dimension of a maximal torus in a real algebraic group)
\begin{equation}\label{eqn_calculation_of_defect} l_0 = \rank G_\infty - \rank Z_\infty K_\infty = \left\{ \begin{array}{ll} r_1 (\frac{n-2}{2}) + r_2(n-1) & n \text { even;} \\ r_1(\frac{n-1}{2}) + r_2(n-1) &n \text{ odd,} \end{array}\right. 
\end{equation}
and we set 
\begin{equation} q_0 = \frac{d-l_0}{2} = \left\{ \begin{array}{ll} r_1 (\frac{n^2}{4}) + r_2n(n-1) & n \text { even;} \\ r_1 (\frac{n^2-1}{4}) + r_2n(n-1) &n \text{ odd.} \end{array}\right. 
\end{equation}
(In particular, if $F$ is totally complex, then 
\[ d = (n^2 - 1) \frac{[F : \bbQ]}{2} \]
and
\[ l_0 = (n-1) \frac{[F : \bbQ]}{2} \]
and
\[ q_0 = n(n-1)\frac{[F : \bbQ]}{2}. \]

\subsection{Cohomology and Hecke operators}\label{sec_hecke_operators}

Let $\cJ$ denote the set of good subgroups of $G^\infty$. If $U \in \cJ$ and $M$ is a (left) $\bbZ[U]$-module on which the group $Z(F) \cap U \subset U$ acts trivially, we can define a local system $\cL_M$ on $X_U$ as the sheaf of local sections of the morphism
\[ G(F) \backslash \left[ G^\infty \times G_\infty/Z_\infty K_\infty \times M\right]  / U  \to G(F) \backslash \left[ G^\infty \times G_\infty/Z_\infty K_\infty\right] /U = X_U. \]
On the left-hand side, the group $G(F) \times U$ acts by $(\gamma, u) \cdot (g^\infty, x, m) = (\gamma g^\infty u^{-1}, \gamma x, u m)$; $M$ is endowed with the discrete topology. The cohomology groups $H^\ast(X_U, M) = H^\ast(X_U, \cL_M)$ are then defined. Let $C_{\bbA, \bullet}$ denote the complex of singular chains with $\bbZ$-coefficients, valued in $G^\infty \times G_\infty/ K_\infty Z_\infty$; it is naturally a $\bbZ[G(F) \times G^\infty]$-module. We define $C^\bullet_\bbA(U, M) = \Hom_{G(F) \times U}(C_{\bbA, \bullet}, M)$ and $C_{\bbA, \bullet}(U, M) = C_{\bbA, \bullet} \otimes_{\bbZ[G(F) \times U]} M$. (In forming this tensor product, it is necessary to view $C_{\bbA, \bullet}$ as a right $\bbZ[G(F) \times U]$-module. Writing a general singular simplex in $G^\infty \times G_\infty/ K_\infty Z_\infty$ as $h \times \sigma$, where $h \in G^\infty$ and $\sigma$ is a singular simplex in $G_\infty/ K_\infty Z_\infty$, we thus have the relation
\[ (hu \times \sigma) \otimes m = (h \times \sigma) \otimes um \]
in $C_{\bbA, \bullet}(U, M)$ for all $u \in U$, $m \in M$.)
\begin{proposition}\label{prop_adelic_complex_computes_cohomology} Let $U \in \cJ$, and let $M$ be a $\bbZ[U]$-module. Then there are canonical isomorphisms
\begin{equation}\label{eqn_adelic_complex_computes_cohomology}
H^\ast(X_U, M) \cong H^\ast(C^\bullet_\bbA(U, M)),
\end{equation}
\begin{equation}\label{eqn_adelic_complex_computes_homology}
H_\ast(X_U, M) \cong H_\ast(C_{\bbA, \bullet}(U, M)).
\end{equation}
\end{proposition}
\begin{proof}
We treat the isomorphism (\ref{eqn_adelic_complex_computes_cohomology}) in cohomology, the other case being similar. Let us write $X = G_\infty / Z_\infty K_\infty$, and choose a decomposition $G^\infty = \sqcup_i G(F) g_i U$. Each of the groups $\overline{\Gamma}_{U, g_i}$ acts freely on $X$, and we have 
\[ X_U = \sqcup_i \left( \overline{\Gamma}_{U, g_i} \backslash X \right).  \]
There are corresponding identifications 
\[ H^\ast(X_U, M) = \oplus_i H^\ast(\overline{\Gamma}_{U, g_i} \backslash X, M_i) = \oplus_i H^\ast( \Hom_{{\Gamma}_{U, g_i}}(C_{\bbR, \bullet}, M_i)), \]
where $C_{\bbR, \bullet}$ denotes the complex of singular chains in $X$ with $\bbZ$-coefficients, and $M_i$ denotes $M$ viewed as an $\bbZ[\Gamma_{U, g_i}]$-module via the formula $\gamma \cdot m = (g_i^{-1} \gamma g_i) m$ ($\gamma \in \Gamma_{U, g_i}$, $g_i^{-1} \gamma g_i \in U$). We will show that there is a natural isomorphism
\begin{equation}\label{eqn_isomorphism_of_chain_complexes} \Hom_{\bbZ[G(F) \times U]}(C_{\bbA, \bullet}, M) \cong \oplus_i \Hom_{\Gamma_{U, g_i}}(C_{\bbR, \bullet}, M_i). 
\end{equation}
This will induce an isomorphism $H^\ast(X_U, M) \cong H^\ast(C^\bullet_\bbA(U, M))$ which is easily seen not to depend on the choice of $g_i$, and therefore prove the proposition. To construct the map (\ref{eqn_isomorphism_of_chain_complexes}), we note that there is a closed embedding
\[ \sqcup_i \left( \{ g_i \} \times X \right) \hookrightarrow G^\infty \times X, \]
and that the corresponding restriction map $\Hom_{\bbZ[G(F) \times U]}(C_{\bbA, \bullet}, M) \to \oplus_i \Hom_\bbZ(C_{\bbR, \bullet}, M)$ actually takes values in $\oplus_i \Hom_{\Gamma_{U, g_i}}(C_{\bbR, \bullet}, M_i)$. To see that the induced map
\[ \Hom_{\bbZ[G(F) \times U]}(C_{\bbA, \bullet}, M) \to \oplus_i \Hom_{\Gamma_{U, g_i}}(C_{\bbR, \bullet}, M_i) \]
 is injective, we note that a singular simplex $\sigma : \Delta^j \to G^\infty \times X$ becomes constant after projection to $G^\infty$, so is $G(F) \times U$-conjugate to a simplex taking values in $\{ g_i \} \times X$, for some $i$ depending on $\sigma$. To see that it is surjective, we note that two simplices $\sigma, \sigma' : \Delta^j \to \{ g_i \} \times X$ are $G(F) \times U$-conjugate if and only if they $\Gamma_{U, g_i}$-conjugate. This completes the proof.
\end{proof}
Let $U, V \in \cJ$. Let $S$ be a finite set of finite places of $F$ such that $U_v = V_v$ if $v \in S$, and suppose that $M$ is actually a $\bbZ[G^{\infty, S} \times U_S]$-module. If $g \in G^{\infty, S}$, then we define a map $[U g V]^\ast : C^\bullet_\bbA(V, M) \to C^\bullet_\bbA(U, M)$ as follows. Fix a decomposition $U g V = \coprod_i g_i V$ ($g_i \in G^{\infty, S}$), and take $\varphi \in C^\bullet_\bbA(U, M)$. If $\sigma \in C_{\bbA, \bullet}$, then we define $([ U g V ]^\ast \varphi)(\sigma) = \sum_i g_i \varphi(g_i^{-1} \sigma)$. It is easy to check that this does not depend on the choice of double coset decomposition. If $U = V$ and we write $\cH(G^{\infty, S}, U^S)$ for the convolution algebra of compactly supported $U^S$-biinvariant functions $f : G^{\infty, S} \to \bbZ$ (with unit element $e_{U^S}$, the characteristic function of $U^S$), then the formula $e_{U^S g U^S} \mapsto [U^S g U^S]^\ast$ defines an action of $\cH(G^{\infty, S}, U^S)$ on the complex $C^\bullet_\bbA(U, M)$, and hence on the cohomology groups $H^\ast(X_U, M)$. Alternatively, if $v \not\in S$ is a finite place of $F$ then the algebra $\cH(G(F_v), U_v)$ acts on $H^\ast(X_U, M)$ in a similar way. 

Similarly, we define a map $[U g V]_\ast : C_{\bbA, \bullet}(U, M) \to C_{\bbA, \bullet}(V, M)$, given on elements of the form $(h \times \sigma) \otimes m$ ($h \in G^\infty,$ $\sigma : \Delta^j \to X,$ $m \in M$) by the formula 
\[ [U g V]_\ast((h \times \sigma) \otimes m) = \sum_i [g_i^{-1} \cdot (h \times \sigma)] \otimes g_i^{-1} m = \sum_i (h g_i \times \sigma) \otimes g_i^{-1} m. \]
It is easy to check that this is well-defined, and (in the case $U = V$) defines an action of the algebra $\cH(G^{\infty, S}, U^S)^\text{op}$ on the complex $C_{\bbA, \bullet}(U, M)$. We then have the following simple lemma.
\begin{lemma}
With notation as above, suppose that $M$ is a  $\bbZ[G^{\infty, S} \times U_S]$-module which is finite as a $\bbZ$-module. Let $M^\vee = \Hom(M, \bbQ/\bbZ)$, equipped with its natural structure of $\bbZ[G^{\infty, S} \times U_S]$-module. Then there are natural isomorphisms $C_{\bbA, \bullet}(V, M)^\vee \cong C_\bbA^\bullet(V, M^\vee)$ and $C_{\bbA, \bullet}(U, M)^\vee \cong C_\bbA^\bullet(U, M^\vee)$, and with respect to these isomorphisms we have $[U g V]^\ast = [U g V]_\ast^\vee$ (i.e. transpose). If $U = V$, then this is an isomorphism of $\cH(G^{\infty, S}, U^S)$-modules, when $C_{\bbA, \bullet}(U, M)^\vee$ is given its natural (transpose) structure of $\cH(G^{\infty, S}, U^S)$-module.
\end{lemma}
\begin{corollary}
Let $L$ be a field, and let $U$ be a good subgroup. Then the natural isomorphism $H^\ast(X_U, L) \cong \Hom(H_\ast(X_U, L), L)$ is an isomorphism of $\cH(G^\infty, U)$-modules.
\end{corollary}
We will generally consider coefficient systems of the following 2 kinds, defined in terms of a fixed prime $p$:
\begin{enumerate}
\item Let $V$ be a finite-dimensional $\bbQ_p$-vector space equipped with a continuous action of $\prod_{v | p} G(F_v)$, and on which the group $\prod_{v | p} Z(F_v)$ acts trivially. We take $M = V$, with $G^\infty$ acting on $M$ by its projection to the factor $\prod_{v | p} G(F_v)$. Then for each $U \in \cJ$, the full Hecke algebra $\cH(G^\infty, U)$ acts on the groups $H^\ast(X_U, M)$, which are $\bbQ_p$-vector spaces.
\item Let $S = S_p$ denote the set of places of $F$ dividing $p$, and let $V$ be a finite $\bbZ_p$-module equipped with a continuous action of $\prod_{v \in p} \GL_n(\cO_{F_v})$, on which the group $\prod_{v | p} \cO_{F_v}^\times$ acts trivially. We take $M = V$, with $G^{\infty, S} \times U_S$ acting on $M$ by its projection to the factor $U_S \subset \prod_{v | p} \GL_n(\cO_{F_v})$. Then for each $U \in \cJ$, the $p$-deprived Hecke algebra $\cH(G^{\infty, S}, U^S)$ acts on the groups $H^\ast(X_U, M)$, which are $\bbZ_p$-modules.
\end{enumerate}
The connection between these two situations will be exploited in \S \ref{sec_independence_of_weight} below.

We now introduce some particular open compact subgroups and associated Hecke operators. If $v$ is a finite place of $F$, then we fix a choice $\varpi_v$ of uniformizer of $\cO_{F_v}$, and define for each $i = 1, \dots, n$ a matrix $\alpha_{v, i} = \diag(\varpi_v, \dots, \varpi_v, 1, \dots, 1)$ (with exactly $i$ entries equal to $\varpi_v$). If $U \in \cJ$ and $U_v = \GL_n(\cO_{F_v})$, then the unramified Hecke operators $T_v^i = [U_v\alpha_{v, i} U_v] \in \cH(G(F_v), U_v)$ are independent of the choice of $\varpi_v$, and pairwise commute. 

If $v$ is a finite place of $F$ and $c \geq b \geq 0$ are integers, then we define an open compact subgroup $I_v(b, c)$ of $\GL_n(\cO_{F_v})$ by the formula
\begin{equation}\label{eqn_definition_of_iwahori} I_v(b, c) = \left\{ \left( \begin{array}{cccc} t_1 & \ast & \cdots & \ast \\
0 & t_2 & \ddots & \vdots \\
\vdots & \ddots & \ddots & \ast \\
0 & \cdots & 0 & t_n \end{array}\right) \text{ mod }\varpi_v^c, \text{ with }t_1 \equiv \dots \equiv t_n \text{ mod }\varpi_v^b\right\}. 
\end{equation}
Thus $I_v(0, 1)$ is the standard Iwahori subgroup of $\GL_n(\cO_{F_v})$. If $U_v = I_v(b, c)$ for some integers $0 \leq b \leq c$, then we define the operators $\mathbf{U}_v^i$ by the same formula $\mathbf{U}_v^i = [U_v \alpha_{v, i} U_v]$. In contrast to the unramified case, these operators do depend on the choice of $\varpi_v$, but this dependence will not be important in what follows. If $\alpha \in I_v(1, c)$ then we write $\langle \alpha \rangle = [ U_v \alpha U_v ]$. (We observe that $I_v(b, c)$ is a normal subgroup of $I_v(1, c)$.)

 The notation $\mathbf{U}_v^i$ is ambiguous, since it does not involve the integers $b \leq c$; however, this abuse of notation is mild, as the following lemma shows.
\begin{lemma}\label{lem_ordinary_hecke_operators} Let $v$ be a finite place of $F$, $U \in \cJ$, and suppose that $U_v = I_v(b, c)$ for some $1 \leq b \leq c$. Let $A$ be a finite abelian group, equipped with its trivial $G^\infty$-action.
\begin{enumerate}
\item The operators $\mathbf{U}_v^i,$ ($i = 1, \dots, n$) and $\langle \alpha \rangle$ ($\alpha \in I_v(1, c)$) all commute in their action on the complex $C^\bullet_\bbA(U, A)$, and hence on the groups $H^\ast(X_U, A)$. 
\item Let $U' \subset U$ be a subgroup with $U'_w = U_w$ if $w \neq v$, and $U'_v = \Iw_v(b', c')$ for some $b' \geq b$, $c' \geq b$. Then the operators $\mathbf{U}_v^i$ and $\langle \alpha \rangle$ commute with the natural maps
\[ C^\bullet_\bbA(U, A) \to C^\bullet_\bbA(U', A) \]
and
\[ H^\ast(X_U, A) \to H^\ast(X_{U'}, A). \]
\item Suppose that $c > b$. Define another subgroup $U' = \prod_w U'_w$ by the formula $U'_w = U_w$ if $w \neq v$ and $U'_v = I_v(b, c-1)$. Then the operator $\mathbf{U}_v^i$ on $C^\bullet_\bbA(U, A)$ takes image inside the subcomplex $C^\bullet_\bbA(X_{U'}, A)$. Consequently, the induced map on cohomology factors
\[ H^\ast(X_U, A) \to H^\ast(X_{U'}, A) \to H^\ast(X_U, A), \]
where the first operator is the double coset operator $[U \alpha_{v, i} U']$ and the second is the canonical pullback.
\end{enumerate}
Analogous statements hold at the level of homology groups.
\end{lemma}
\begin{proof}
The first part follows from the fact that these operators generate a commutative subalgebra of $\cH(G(F_v), U_v)$; see \cite[Proposition 2.2]{Hid95}. The proof of \emph{loc. cit.} implies that the coset representatives used to define the action of $\mathbf{U}_v^i$ can be chosen independent of $b$ and $c$, which implies that the given maps commute with the action of Hecke, cf. the proof of \cite[Lemma 2.3.3]{Ger09}. This implies the second part of the lemma. The third part follows in a similar way, by using the argument of \cite[Lemma 2.5.2]{Ger09}.
\end{proof}
\subsection{Ordinary part of homology}\label{sec_ordinary_part_of_homology}

We now fix, for the rest of \S \ref{sec_ordinary_completed_cohomology}, a prime $p$. Let $S_p$ denote the set of places of $F$ dividing $p$. Let $E$ be a finite extension of $\bbQ_p$ inside $\overline{\bbQ}_p$, which contains the images of all embeddings $F \hookrightarrow \overline{\bbQ}_p$. We write $\varpi \in \cO$ for a choice of uniformizer of the ring of integers $\cO$ of $E$, $\lambda = (\varpi)$, and $k = \cO/\lambda$. We also fix a good subgroup $U = \prod_v U_v \subset G^\infty$, with the property that $U_v = I_v(1, 1)$ for each $v \in S_p$. If $c \geq b \geq 1$ are integers, then we define a new good subgroup $U(b, c) = \prod_v U(b, c)_v \subset U$ by the formulae:
\begin{itemize}
\item If $v \in S_p$, then $U(b, c)_v = I_v(b, c)$.
\item If $v \not\in S_p$, then $U(b, c)_v = U_v$. 
\end{itemize}
We define 
\[ T(0, c) = \prod_{v \in S_p} (\cO_{F_v}/\varpi_v^c)^\times \times \dots \times (\cO_{F_v}/\varpi_v^c)^\times\text{ (}n-1\text{ times),}\]
and $T(b, c) = \ker(T(0, c) \to T(0, b))$. There is an isomorphism
\begin{gather}
U(1, c) / U(b, c) \cong T(b, c) \\ \label{eqn_remove_the_center}
\left( \begin{array}{cccc} t_{v, 1} & \ast & \cdots & \ast \\
0 & t_{v, 2} & \ddots & \vdots \\
\vdots & \ddots & \ddots & \ast \\
0 & \cdots & 0 & t_{v, n} \end{array}\right)_v \text{ mod }\varpi_v^c \mapsto (t_{v, 1} t_{v, n}^{-1}, \dots, t_{v, n-1} t_{v, n}^{-1})_v.
\end{gather}
We observe that for any $c \geq b \geq 1$, we have $Z(F) \cap U(b, c) = Z(F) \cap U(1, c) = Z(F) \cap U$. The group $U(1, c)$ acts on $X_{U(b, c)}$, and the quotient 
\[ U(1, c) / (U(b, c) \cdot (Z(F) \cap U(0, c)) = U(1, c) / U(b, c) \cong T(b, c)\]
 acts freely, by Lemma \ref{lem_good_subgroups_act_without_fixed_points}. We define $\Lambda(b, c) = \cO[T(b, c)]$ and $\Lambda_b = \plim_c \Lambda(b, c)$. If $b \geq 1$, then $\Lambda_b$ is a complete Noetherian local ring (it is the completed group ring of an abelian pro-$p$ group). If $b = 1$, then we simply write $\Lambda = \Lambda_1$.

We now define the ordinary completed cohomology groups which are our objects of ultimate interest. By Lemma \ref{lem_ordinary_hecke_operators}, the natural maps
\[ H_i(X_{U(c', c')}, \cO/\varpi^{c'}\cO) \to H_i(X_{U(c, c)}, \cO/\varpi^c\cO)\text{ }(c' \geq c) \]
commute with the action of the Hecke operator 
\[ \mathbf{U}_p = \prod_{v \in S_p} \prod_{i = 1}^n \mathbf{U}_v^i = [ U(c, c) \alpha_p U(c, c)], \]
 where
\[ \alpha_p = \prod_{v \in S_p} \prod_{i=1}^n \alpha_{v, i} = (\diag(\varpi_v^{n}, \varpi_v^{n-1}, \dots, \varpi_v) )_{v \in S_p} \in \prod_{v \in S_p} G(F_v). \]
We therefore have induced maps (where ordinary parts are taken with respect to the operator $\mathbf{U}_p$):
\[ H_{i}(X_{U(c', c')}, \cO/\varpi^{c'}\cO)_\text{ord} \to H_{i}(X_{U(c, c)}, \cO/\varpi^c\cO)_\text{ord}. \]
We define
\[ H^i_\text{ord}({U}) = \plim_c H_{d-i}(X_{U(c, c)}, \cO/\varpi^c \cO)_\text{ord}. \]
The $\Lambda$-modules $H^i_\text{ord}({U})$ receive a canonical action of the algebra $\cH(G^{p, \infty}, U^p)^\text{op}$, induced by the action at finite levels. Our goal in this section is to prove the following version of Hida's control theorem (\cite[Theorem 6.2]{Hid98}): 
\begin{proposition}\label{prop_complex_control_theorem}
There is a minimal complex $F^\bullet_\infty$ of $\Lambda$-modules, together with morphisms for every $c \geq 1$:
\begin{gather*} g_c : F^\bullet_{\infty} \otimes_\Lambda \Lambda(1, c)/(\varpi^c) \to C_{\bbA, d - \bullet}(U(c, c), \cO/\varpi^c \cO), \\
g_c' :  C_{\bbA, d-\bullet}(U(c, c), \cO/\varpi^c \cO) \to  F^\bullet_{\infty} \otimes_\Lambda \Lambda(1, c)/(\varpi^c),
\end{gather*}
all satisfying the following conditions:
\begin{enumerate}
\item We have $g'_c g_c = 1$, and $g_c g'_c$ is an idempotent in $\End_{\mathbf{D}(\Lambda(1, c)/(\varpi^c))}(C_{\bbA, d - \bullet}(U(c, c), \cO/\varpi^c \cO))$.
\item For each $c \geq 1$, the diagram
\[ \xymatrix{F^\bullet_\infty\otimes_\Lambda \Lambda(1, c+1)/(\varpi^{c+1}) \ar[r] \ar[d] &  C_{\bbA, d - \bullet}(U(c+1, c+1), \cO/\varpi^{c+1} \cO) \ar[d] \ar[r] & F^\bullet_\infty\otimes_\Lambda \Lambda(1, c+1)/(\varpi^{c+1})\ar[d] \\
F^\bullet_\infty\otimes_\Lambda \Lambda(1, c)/(\varpi^{c}) \ar[r]  & C_{\bbA, d - \bullet}(U(c, c), \cO/\varpi^c \cO) \ar[r] & F^\bullet_\infty\otimes_\Lambda \Lambda(1, c)/(\varpi^{c}).} \]
is commutative, up to chain homotopy.
\item For each $c \geq 1$, the induced map of $g_c g'_c$ on homology is the natural projection
\[ H_{d-\ast}(X_{U(c, c)}, \cO/\varpi^c \cO) \to H_{d-\ast}(X_{U(c, c)}, \cO/\varpi^c \cO)_\text{ord}. \]
\end{enumerate}
The minimal complex $F^\bullet_\infty$ is determined uniquely up to unique isomorphism in $\mathbf{D}(\Lambda)$ by these properties. Finally, there is a unique homomorphism $f : \cH(G^{\infty, p}, U^p)^\text{op} \to \End_{\mathbf{D}(\Lambda)}(F^\bullet_\infty)$ with the property that for all $t \in \cH(G^{\infty, p}, U^p)^\text{op}$, we have $(f(t) \otimes_\Lambda \Lambda(1, c)/(\varpi^c)) = g'_c t g_c$.
\end{proposition}
When we wish to emphasize the level, we will use the alternative notation $F^\bullet_{\infty, U}$. Later, we will also introduce a variant $F^\bullet_{\infty, U, \ffrm}$, which computes cohomology after localization at a maximal ideal $\ffrm$ of a suitable Hecke algebra, and a variant where we allow a finite $p$-group to act at an auxiliary place (see Lemma \ref{lem_properties_of_individual_taylor_wiles_prime} below). We see (Lemma \ref{lem_limit_of_good_complexes}) that there are natural Hecke-equivariant isomorphisms
\[ H^\ast(F^\bullet_\infty) \cong \plim_c H_{d-\ast}(X_{U(c, c)}, \cO/\varpi^c \cO)_\text{ord} \cong \plim_c H_{d-\ast}(X_{U(c, c)}, \cO)_\text{ord}. \]
\begin{corollary}\label{cor_control_theorem_finite_levels}
For each $c \geq 1$, there is a spectral sequence:
\[ E_2^{p, q} = \mathrm{Tor}^\Lambda_{-p}(H^q_\text{ord}({U}), \Lambda(1, c)/(\varpi^c)) \Rightarrow H_{d - (p + q)}(X_{U(c, c)}, \cO/\varpi^c \cO)_\text{ord}. \]
\end{corollary}
\begin{proof}[Proof of Corollary \ref{cor_control_theorem_finite_levels}]
This follows on combining Proposition \ref{prop_complex_control_theorem} and Proposition \ref{prop_change_of_coefficients_spectral_sequence}.
\end{proof}
\begin{corollary}\label{cor_control_theorem_spectral_sequence}
For each $c \geq 1$, there is a spectral sequence:
\[ E_2^{p, q} = \mathrm{Tor}^\Lambda_{-p}(H^q_\text{ord}({U}), \Lambda(1, c)) \Rightarrow H_{d - (p + q)}(X_{U(c, c)}, \cO)_\text{ord}. \]
\end{corollary}
\begin{proof}[Proof of Corollary \ref{cor_control_theorem_spectral_sequence}]
We want to show that there is a quasi-isomorphism \
\[ F^\bullet_\infty \otimes_\Lambda \Lambda(1, c) \cong C_{\bbA, \bullet}(U(c, c), \cO)_\text{ord}, \]
where the ordinary part is taken in the sense of the discussion following Lemma \ref{lem_ordinary_part_in_derived_category}.
It follows from the construction of $F^\bullet_\infty$ that for every $d \geq c$ there is a quasi-isomorphism
\[ F^\bullet_\infty \otimes_\Lambda \Lambda(1, c)/(\varpi^d) \cong C_{\bbA, \bullet}(U(c, c), \cO/\varpi^d \cO)_\text{ord}, \]
and that the natural diagrams 
\[ \xymatrix{ F^\bullet_\infty \otimes_\Lambda \Lambda(1, c)/(\varpi^{d+1}) \ar[r] \ar[d] & C_{\bbA, \bullet}(U(c, c), \cO/\varpi^{d+1} \cO)_\text{ord} \ar[d] \\
 F^\bullet_\infty \otimes_\Lambda \Lambda(1, c)/(\varpi^{d}) \ar[r]  & C_{\bbA, \bullet}(U(c, c), \cO/\varpi^{d} \cO)_\text{ord} } \]
 commute in $\mathbf{D}(\cO)$. It follows that we can pass to the limit, as in the proof of Lemma \ref{lem_limit_of_good_complexes}, to obtain the desired result.
\end{proof}
We now begin the proof of Proposition \ref{prop_complex_control_theorem}, starting with the following useful lemma:
\begin{lemma}\label{lem_existence_of_perfect_complex}
Let $U$ be a good subgroup, and let $V \subset U$ be a good normal subgroup such that the quotient $\Delta = U/V$ is abelian of $p$-power order and $Z(F) \cap U \subset V$. (Then $\Delta$ acts freely on $X_V$.) Then there is a canonical isomorphism 
\[ C_{\bbA, \bullet}(V, \cO/\varpi^c \cO) \otimes_{\cO[\Delta]} \cO \cong C_{\bbA, \bullet}(U, \cO/\varpi^c \cO).
\]
Moreover, $C_{\bbA, \bullet}(V, \cO/\varpi^c \cO)$ is a perfect complex of free $\cO/\varpi^c\cO[\Delta]$-modules.
\end{lemma}
\begin{proof}[Proof of Lemma \ref{lem_existence_of_perfect_complex}]
The existence of the given isomorphism is an easy consequence of the definitions. For the second part, we must show that $C_\bbA^\bullet(X_V, \varpi^{-c} \cO/\cO)^\vee$ is quasi-isomorphic to good complex of $\cO/\varpi^c\cO[\Delta]$-modules. (It is easy to see that it is a complex of free $\cO/\varpi^c\cO[\Delta]$-modules.) We fix a triangulation of $X_U$ and lift it to $X_V$. (Since $X_V$ is not compact, this triangulation will have infinitely many simplices.) Let $C_\bullet$ be the corresponding complex of free $\cO/\varpi^c\cO[\Delta]$-modules which computes $H_\ast(X_V, \cO/\varpi^c\cO)$, concentrated in degrees $[0, d]$. By Lemma \ref{lem_quasi_isomorphism_of_perfect_complexes}, $C_\bullet$ is quasi-isomorphic to a good complex. (Its cohomology is finitely generated because of the existence of the Borel--Serre compactification.) On the other hand, $C_\bullet$ is also quasi-isomorphic to $C_{\bbA, \bullet}(X_V, \cO/\varpi^c \cO)$. This completes the proof.
\end{proof}

\begin{lemma}\label{lem_hidas_main_lemma}
For all $c \geq b \geq 1$, $d \geq 1$, the natural map
\begin{equation}
H_\ast( X_{U(b, c)}, \cO/\varpi^d \cO)_\text{ord} \to H_\ast( X_{U(b, b)}, \cO/\varpi^d \cO)_\text{ord} 
\end{equation}
is an isomorphism.
\end{lemma}
\begin{proof}
Let $c > b \geq 1$, and let $i$ be an integer. It follows from Lemma \ref{lem_ordinary_hecke_operators} that there is a commutative diagram
\begin{equation*}
\xymatrix{ H_\ast(X_{U(b, c+1)},\cO/\varpi^d\cO) \ar[r]^{\mathbf{U}_p} \ar[d]^{j} & H_\ast(X_{U(b, c+1)}, \cO/\varpi^d\cO)\ar[d]^{j} \\
H_\ast(X_{U(b, c)}, \cO/\varpi^d\cO) \ar[ur]^{r} \ar[r]^{\mathbf{U}_p} & H_\ast(X_{U(b, c)}, \cO/\varpi^d\cO). }
\end{equation*}
Here we write $j : H_\ast(X_{U(b, c+1)}, \cO/\varpi^d\cO) \to H_\ast(X_{U(b, c)}, \cO/\varpi^d\cO)$ for the natural pushforward and $r = [U(b, c) \alpha_p U(b, c+1)]_\ast$. It follows easily from the existence of this diagram that $j$ induces the desired isomorphism on ordinary parts. 
\end{proof}
We can now complete the proof of Proposition \ref{prop_complex_control_theorem}. We apply Proposition \ref{prop_limit_and_gluing_ordinary_parts} to the system of complexes $M^\bullet_c = C_{\bbA, d-\bullet}(U(c, c), \cO/\varpi^c \cO)$, together with their Hecke operators $t_c = \mathbf{U}_p$. To apply this proposition, we must check that the natural maps
\[  H^\ast(C_{\bbA, d - \bullet}(U(c+1, c+1), \cO/\varpi^{c+1} \cO) \otimes_{\Lambda} \Lambda(1, c)/(\varpi^c))_\text{ord} \to H^\ast(C_{\bbA, d - \bullet}(U(c, c), \cO/\varpi^{c} \cO))_\text{ord} \]
are isomorphisms; equivalently, that the natural maps
\[  H_\ast(X_{U(c, c+1)},  \cO/\varpi^c \cO)_\text{ord} \to H_\ast(X_{U(c, c)}, \cO/\varpi^c \cO)_\text{ord} \]
are isomorphisms. However, this is exactly the content of Lemma \ref{lem_hidas_main_lemma}.

\subsection{Independence of weight}\label{sec_independence_of_weight}

Let us now write $\GL_n$ for the general linear group over $\bbZ$, and $\mathrm{T}_n \subset \mathrm{B}_n \subset \GL_n$ for its standard diagonal maximal torus and upper-triangular Borel subgroup. We now introduce coefficient systems on the spaces $X_U$ corresponding to algebraic representations of $\GL_n$. Let $\bbZ^n_+ \subset \bbZ^n$ denote the set of tuples of integers $\lambda_1 \geq \dots \geq \lambda_n$ with $\lambda_1 + \dots + \lambda_n = 0$. If $\lambda \in \bbZ^n$, then we write $w_0 \lambda$ for the tuple $(w_0 \lambda)_i = \lambda_{n+1-i}$. Given $\lambda \in \bbZ^n$, we can define the algebraic representation $V_\lambda = \Ind_{\mathrm{B}_n}^{\GL_n}(w_0 \lambda)_{/\bbZ}$ of $\GL_n$ of highest weight $\lambda$ as in \cite[\S 1]{Ger09}. There is a $\mathrm{T}_n$-equivariant weight decomposition
\[ V_\lambda = \oplus_{\mu \in X^\ast(\mathrm{T}_n)} V_{\lambda, \mu}. \]
The highest and lowest weight spaces $V_{\lambda, \lambda}$ and $V_{\lambda, w_0 \lambda}$ of $V_\lambda$ are 1-dimensional. The condition $\lambda_1 + \dots + \lambda_n = 0$ implies that the center $\mathrm{Z}_n \subset \mathrm{GL}_n$ acts trivially on $V_{\lambda}$.

Now suppose given an element $\bl = (\lambda_\tau)_{\tau : F \hookrightarrow E} \in (\bbZ^n_+)^{\Hom(F, E)}$. For each $\tau \in \Hom(F, E)$, $V_{\lambda_\tau}(\cO)$ is a finite free $\cO$-module which receives a continuous action of the group $\GL_n(\cO)$, hence a continuous action of the group $\GL_n(\cO_{F_{v(\tau)}})$, where $v(\tau)$ is the $p$-adic place of $F$ induced by the embedding $\tau$. We define $M_\bl = \otimes_{\tau : F \hookrightarrow E} V_{\lambda_\tau}(\cO)$, the tensor product being over $\cO$. Then $M_\bl$ is a free $\cO$-module which receives a continuous action of the group $\prod_{v \in S_p} \GL_n(\cO_{F_v})$. If $\bl = 0$, then $M_\bl$ is just $\cO$, equipped with the trivial action of the group $\prod_{v \in S_p} \GL_n(\cO_{F_v})$. Moreover, $M_\bl \otimes_\cO E$ is equipped with a continuous action of the group $\prod_{v \in S_p} \GL_n(F_v)$ extending the action of its subgroup $\prod_{v \in S_p} \GL_n(\cO_{F_v})$.

Now suppose that $U \in \cJ$ and $U_v = I_v(1, 1)$ for each $v \in S_p$. Let $\bl \in (\bbZ^n_+)^{\Hom(F, E)}$. If $c \geq b \geq 1$ are integers, then we define $U(b, c)$ as in \S \ref{sec_ordinary_part_of_homology}. For each $v \in S_p$, the Hecke operators $\mathbf{U}_v^i$ act on the complex $C_{\bbA, \bullet}(U, M_\bl \otimes_\cO E)$. In fact, a suitably scaled version of these operators leaves invariant the natural integral lattice:
\begin{lemma}\label{lem_integrality_of_u_p_operator}
The operator $\bl(\alpha_{v, i}) \mathbf{U}_v^i$ on $C_{\bbA, \bullet}(U, M_\bl \otimes_\cO E)$ preserves the submodule $C_{\bbA, \bullet}(U, M_\bl)$.
\end{lemma}
\begin{proof}
We first note that the map $C_{\bbA, \bullet}(U, M_\bl) \to C_{\bbA, \bullet}(U, M_\bl \otimes_\cO E)$ is indeed injective (because $C_{\bbA, \bullet}(U, M_\bl)$ is a separated, torsion-free $\cO$-module, being isomorphic to the space of singular chains in $X_U$ with $M_\bl$-coefficients, cf. the proof of Proposition \ref{prop_adelic_complex_computes_cohomology}). Choose a decomposition $U(c, c) = \sqcup_j g_j (U(c, c) \cap \alpha_{v, i} U(c, c) \alpha_{v, i}^{-1})$. Then for any $(h \times \sigma) \otimes m \in C_{\bbA, \bullet}(U, M_\bl \otimes_\cO E) = C_{\bbA, \bullet} \otimes_{\bbZ[G_F \times U]} M_\bl \otimes_\cO E$, we have
\[ \mathbf{U}_{v}^i (h \times \sigma) \otimes m = \sum_j (h g_j \alpha_{v, i} \times \sigma) \otimes \alpha_{v, i}^{-1} g_j^{-1} m.  \]
To prove the lemma, it is therefore enough to show that the element $\bl(\alpha_{v, i}) \alpha_{v, i}^{-1}$ acting on $M_\bl \otimes_\cO E$ leaves the lattice $M_\bl$ invariant. It even suffices to show that for each embedding $\tau : F_v \hookrightarrow E$, the element $\lambda_\tau(\alpha_{v, i}) \alpha_{v, i}^{-1}$ acting on $V_{\lambda_\tau}(F_v)$ leaves the lattice $V_{\lambda_\tau}(\cO_{F_v})$ invariant. For every weight $\mu$ appearing in $V_{\lambda_\tau}$ and every $t = \diag(t_1, \dots, t_n) \in \mathrm{T}_n(F_v)$, we can write
\[ (\lambda_\tau - \mu)(t) = \prod_{i=1}^{n-1} (t_j/t_{j+1})^{a_j}, \]
where the $a_j$ are non-negative integers (because $\lambda_\tau$ is the highest weight of $V_{\lambda_\tau})$. If $t = \alpha_{v, i}$ then all the elements $t_j/t_{j+1}$ lie in $\cO_{F_v}$, and this implies the desired result.
\end{proof}
We write $\mathbf{U}_{v, \bl}^i = \bl(\alpha_{v, i}) \mathbf{U}_v^i$ for the induced endomorphism of the complex $C_{\bbA, \bullet}(U, M_\bl)$. For each $b \geq 1$, there is a canonical isomorphism 
\begin{equation}\label{eqn_reduction_modulo_r} C_{\bbA, \bullet}(U, M_\bl) \otimes_\cO \cO/\lambda^b \cong C_{\bbA, \bullet}(U, M_\bl \otimes_\cO \cO/\lambda^b), 
\end{equation}
and we use this to define an action of $\mathbf{U}_{v, \bl}^i$ on the right-hand side of (\ref{eqn_reduction_modulo_r}) by reduction modulo $\varpi^b$. We define 
\[ \mathbf{U}_{p, \bl} = \prod_{v \in S_p} \prod_{i=1}^n \mathbf{U}_{v, \bl}^i = \left( \prod_{v \in S_p} \prod_{i=1}^n \bl(\alpha_{v, i}) \right) \cdot \mathbf{U}_p. \]
The operator $\mathbf{U}_{p, \bl}$ therefore acts on the groups $H_\ast(X_U(c, c), M_\bl)$ and $H_\ast(X_U, M_\bl \otimes_\cO \cO/\varpi^b)$ for all $c, b \geq 1$, and we define the ordinary parts of these groups as in Lemma \ref{lem_ordinary_part_for_modules} with respect to this operator. Just as in the previous section, one can prove the following result.
\begin{proposition}\label{prop_complex_control_theorem_variable_weight}
There is a minimal complex $F^\bullet_{\infty, \bl}$ of $\Lambda$-modules, together with morphisms for every $c \geq 1$:
\begin{gather*} g_c : F^\bullet_{\infty, \bl} \otimes_\Lambda \Lambda(1, c)/(\varpi^c) \to C_{\bbA, d-\bullet}(U(c, c), M_\bl \otimes_\cO \cO/\varpi^c), \\
g_c' :  C_{\bbA, d-\bullet}(U(c, c), M_\bl \otimes_\cO \cO/\varpi^c) \to  F^\bullet_{\infty, \bl} \otimes_\Lambda \Lambda(1, c)/(\varpi^c),
\end{gather*}
all satisfying the following conditions:
\begin{enumerate}
\item We have $g'_c g_c = 1$, and $g_c g'_c$ is an idempotent in $\End_{\mathbf{D}(\Lambda(1, c)/(\varpi^c))}(C_{\bbA, d-\bullet}(U(c, c), M_\bl \otimes_\cO \cO/\varpi^c))$.
\item For each $c \geq 1$, the diagram
\[ \xymatrix{F^\bullet_{\infty, \bl} \otimes_\Lambda \Lambda(1, c+1)/(\varpi^{c+1}) \ar[r] \ar[d] & C_{\bbA, d-\bullet}(U(c+1, c+1), M_\bl \otimes_\cO \cO/\varpi^{c+1}) \ar[d] \ar[r] & F^\bullet_{\infty, \bl} \otimes_\Lambda \Lambda(1, c+1)/(\varpi^{c+1})\ar[d] \\
F^\bullet_{\infty, \bl} \otimes_\Lambda \Lambda(1, c)/(\varpi^{c}) \ar[r]  &  C_{\bbA, d-\bullet}(U(c, c), M_\bl \otimes_\cO \cO/\varpi^c)  \ar[r] & F^\bullet_{\infty, \bl} \otimes_\Lambda \Lambda(1, c)/(\varpi^{c}).} \]
is commutative, up to chain homotopy.
\item For each $c \geq 1$, the induced map of $g_c g'_c$ on cohomology is the natural projection
\[ H_\ast(X_{U(c, c)}, M_\bl \otimes_\cO \cO/\varpi^c) \to H_\ast(X_{U(c, c)}, M_\bl \otimes_\cO \cO/\varpi^c)_\text{ord}. \]
\end{enumerate}
The minimal complex $F^\bullet_{\infty, \bl}$ is determined uniquely up to isomorphism in $\mathbf{D}(\Lambda)$ by these properties. Finally, there is a unique homomorphism $f : \cH(G^{\infty, p}, U^p)^\text{op} \to \End_{\mathbf{D}(\Lambda)}(F^\bullet_{\infty, \bl})$ with the property that for all $t \in \cH(G^{\infty, p}, U^p)^\text{op}$, we have $(f(t) \otimes_\Lambda \Lambda(1, c)/(\varpi^c)) = g'_c t g_c$.
\end{proposition}
 The following result is the analogue of \cite[Proposition 2.6.1]{Ger09} in our context.
\begin{proposition}\label{prop_independence_of_weight}
Let $\bl \in (\bbZ^n_+)^{\Hom(F, E)}$, and let $c \geq  r$ be integers. Then there is a morphism of complexes 
\begin{equation}\label{eqn_independence_of_weight} \alpha_{\bl, \ast} : C_{\bbA, \bullet}(U(c, c), \cO/\varpi^r) \to C_{\bbA, \bullet}(U(c, c), M_\bl \otimes_\cO \cO/\varpi^r),
\end{equation}
equivariant for the Hecke operators at finite places $v \not\in S_p$, and satisfying the equations 
\[ U^i_{v, \bl} \alpha_{\bl, \ast} = \alpha_{\bl, \ast} U^i_v \text{ and }\alpha_{\bl, \ast} \langle u \rangle = \bl(u) \langle u \rangle \alpha_{\bl, \ast} \]
for all $u \in \mathrm{T}_n(\cO_{F_v})$, $v \in S_p$. After passing to homology, the map (\ref{eqn_independence_of_weight}) induces an isomorphism
\[ \alpha_{\bl, \ast} : H_\ast(X_{U(c, c)}, \cO/\varpi^r)_\text{ord} \to H_\ast(X_{U(c, c)}, M_\bl \otimes_\cO \cO/\varpi^r)_\text{ord}. \]
\end{proposition}
\begin{proof}
We recall that by definition, we have 
\[ C_{\bbA, \bullet}(U(c, c), M_\bl \otimes_\cO \cO/\varpi^r) = C_{\bbA, \bullet} \otimes_{\bbZ[G(F) \times U(c, c)]} M_\bl \otimes_\cO \cO/\varpi^r. \]
If $\lambda \in \bbZ^n_+$, let $\alpha_\lambda : \bbZ(\lambda) \to V_\lambda$ be inclusion of the highest weight space, and let $\beta_\lambda : V_\lambda \to \bbZ(\lambda)$ be projection to the highest weight space. If $\bl \in (\bbZ_+^n)$, then we write $\alpha_\bl : \cO(\bl) \to M_\bl$, $\beta_\bl : M_\bl \to \cO(\bl)$ for the maps obtained by taking tensor products and then extending scalars to $\cO$. We observe that if $c \geq r$, then the map $\alpha_\bl$ is $U(c, c)$-equivariant.

 The map 
 \[ \alpha_{\bl, \ast} :C_{\bbA, \bullet}(U(c, c), \cO/\varpi^r) \to C_{\bbA, \bullet}(U(c, c), M_\bl \otimes_\cO \cO/\varpi^r)  \]
  is then given as $\text{id} \otimes \alpha_{\bl}$. The equivariance with respect to Hecke operators away from $p$ is clear from the definitions. To show that the map $\alpha_{\bl, \ast}$ induces an isomorphism on ordinary parts of cohomology, we construct a map in the other direction. To this end, we choose a decomposition $U(c, c) \alpha_p^r U(c, c) = \sqcup_i g_i U(c, c)$, and write
\[ \varphi : C_{\bbA, \bullet}(U(c, c), M_\bl \otimes_\cO \cO/\varpi^r) \to C_{\bbA, \bullet}(U(c, c), \cO/\varpi^r) \]
for the map given by
\[ \varphi(h \times \sigma \otimes m) = \sum_i \beta_\bl( \bl(\alpha_p^r) g_i^{-1} m ) \cdot h g_i \times \sigma. \]
This expression makes sense since, as observed above, the endomorphism $\bl(\alpha_p^r) g_i^{-1}$ preserves the natural integral lattice $M_\bl \subset M_\bl \otimes_\cO E$, so can be reduced modulo $\varpi^r$. In fact, its action modulo $\varpi^r$ is just projection to the highest weight space $M_\bl$. Formally, we have $\varphi = \beta_{\bl, \ast} \circ \mathbf{U}_{p, \bl}^r$. To show that $\varphi$ is well-defined, we check that 
\[ \varphi(hu \times \sigma \otimes m) = \varphi(h \times \sigma \otimes um) \]
for each $u \in U(c, c)$. This follows from the observation that 
\[ \sum_i\beta_\bl( \bl(\alpha_p^r) g_i^{-1} u m ) = \sum_j \beta_\bl(\bl(\alpha_p^r) g_j^{-1} m), \]
as we can write $u^{-1} g_i = g_{j(i)} u_{j(i)}$ for some $u_j \in U(c, c)$, and $U(c, c)$ acts trivally on the image of $\bl(\alpha_p^r) g_j^{-1}$ modulo $\varpi^r$.

We have therefore constructed the arrows in a diagram
\[ \xymatrix{ C_{\bbA, \bullet}(U, \cO/\varpi^r\cO) \ar[r]^-{\alpha_{\bl, \ast}} \ar[d]^{\mathbf{U}_{p}^r} & C_{\bbA, \bullet}(U, M_\bl \otimes_\cO \cO/\varpi^r\cO) \ar[d]^{\mathbf{U}^r_{p, \bl}} \ar[ld]^{\varphi} \\
 C_{\bbA, \bullet}(U, \cO/\varpi^r\cO) \ar[r]^-{\alpha_{\bl, \ast}}  & C_{\bbA, \bullet}(U, M_\bl \otimes_\cO \cO/\varpi^r\cO).} \]
It is now easy to check that this diagram commutes; the rest of the proposition follows from the existence of this diagram.
\end{proof}
We can now deduce a derived version of Hida's theorem on independence of weight (\cite[Theorem 6.1]{Hid98}):
\begin{corollary}\label{cor_complexes_independent_of_weight}
Let $\cO(\bl)$ denote $\cO$ considered as a $\Lambda$-module via the homomorphism 
\[ \bl : \prod_{v \in S_p} \cO_{F_v}^\times(p)^n \to \cO^\times. \]
 Then there is an isomorphism $F^\bullet_{\infty, \bl} \cong F^\bullet_{\infty} \otimes_\cO \cO(\bl^{-1})$ of minimal complexes of $\Lambda$-modules. It is equivariant for the action of $\cH(G^{\infty, p}, U^p)^\text{op}$. In particular, there is a canonical isomorphism of $\cO$-modules
\[ H^\ast_\text{ord}(U) \cong \plim_c H_{d-\ast}(X_U, M_\bl)_\text{ord}. \]
\end{corollary}
\begin{proof}
It follows immediately from Proposition \ref{prop_independence_of_weight} and the defining properties of the complexes $F^\bullet_\infty$ and $F^\bullet_{\infty, \bl}$ that there is a quasi-isomorphism $f : F^\bullet_{\infty, \bl} \to F^\bullet_{\infty} \otimes_\cO \cO(\bl)$. Since these are minimal complexes, it is an isomorphism.
\end{proof}
\begin{corollary}\label{cor_ordinary_complex_computes_fixed_weight_cohomology}
Let $\wp_\bl \subset \Lambda$ be the kernel of the homomorphism $\Lambda \to \cO = \End_\cO(\cO(\bl^{-1}))$. Then there is a canonical isomorphism
\[ H^\ast(F^\bullet_{\infty} \otimes_\Lambda \Lambda/\wp_\bl) \cong H_{d-\ast}(X_{U(1, 1)}, M_\bl)_\text{ord}, \]
and hence a spectral sequence
\[ \mathrm{Tor}^\Lambda_{-p}(H^q_\text{ord}(U), \Lambda/\wp_\bl) \Rightarrow H_{d - (p+q)}(X_{U(1, 1)}, M_\bl)_\text{ord}. \]
In particular, there is an isomorphism 
\[H^\ast(F^\bullet_{\infty} \otimes_\Lambda \Lambda/\wp_\bl)[1/p] \cong \Hom_E(H^{d-\ast}(X_{U(1, 1)}, M_{-w_0 \bl})_\text{ord} \otimes_\cO E, E). \]
\end{corollary}
\begin{proof}
This follows immediately from Corollary \ref{cor_complexes_independent_of_weight} and the analogue of Corollary \ref{cor_control_theorem_spectral_sequence} for the complex $F^\bullet_{\infty, \bl}$. (Note that if $\lambda \in \bbZ^n_+$, then $- w_0 \lambda \in \bbZ^n_+$ is the highest weight of the dual representation.)
\end{proof}

\subsection{Hecke algebras}\label{sec_hecke_algebras}

We now introduce the big ordinary Hecke algebras. We fix a choice of good subgroup $U$ such that for each $v \in S_p$, $U_v = I_v(1, 1)$. Let $S$ be a finite set of finite places of $F$, containing $S_p$, such that for all $v \not\in S$, we have $U_v = \GL_n(\cO_{F_v})$. We define $\bbT^{S, \text{univ}}$ to be the polynomial algebra over $\Lambda$ in the infinitely many indeterminates $T_v^i$, $v \not\in S$ and $i = 1, \dots, n$ and the indeterminates $\mathbf{U}_v^i$, $v \in S_p$ and $i = 1, \dots, n$. Following the discussion in \S \ref{sec_hecke_operators}, this algebra acts on the groups $H_\ast(X_{U(c, c)}, A)$ for any integer $c \geq 1$ and $\cO$-module $A$, each element $T_v^i \in \bbT^{S, \text{univ}}$ acting on $H_\ast(X_{U(c, c)}, A)$ by the Hecke operator of the same name. If $C^\bullet$ is a complex of $\Lambda$-modules equipped with a homomorphism $\bbT^{S , \text{univ}} \to \End_{\mathbf{D}(\Lambda)}(C^\bullet)$, then we write $\bbT^S(C^\bullet)$ for the image of this homomorphism. 

An important special case is when the complex $C^\bullet$ is just a $\bbT^{S, \text{univ}}$-module $M$ placed in degree 0; in this case, we have $\bbT^S(C^\bullet) = \bbT^S(M) = \im(\bbT^{S, \text{univ}} \to \End_\Lambda(M))$. Using this notation, we see that for any complex $C^\bullet$ of $\Lambda$-modules equipped with a homomorphism $\bbT^{S , \text{univ}} \to \End_{\mathbf{D}(\Lambda)}(C^\bullet)$, there is a surjective homomorphism
\[ \bbT^S(C^\bullet) \to \bbT^S(H^\ast(C^\bullet)), \]
which need not be injective.
\begin{lemma}
Let $M$ be a $\bbT^{S, \text{univ}}$-module, which is finite as a $\Lambda$-module. 
\begin{enumerate}
\item $\bbT^S(M)$ is a finite $\Lambda$-algebra, with finitely many maximal ideals $\ffrm$. For each maximal ideal $\ffrm \subset \bbT^S(M)$, $\bbT^S(M)/\ffrm$ is a finite extension of $k$.
\item Let $\ffrm' \subset \bbT^{S, \text{univ}}$ be the kernel of a homomorphism $\bbT^{S, \text{univ}} \to k$. Then either $M_{\ffrm'} = 0$, or $\ffrm'$ is the pre-image of a maximal ideal of $\bbT^S(M)$.
\end{enumerate}
\end{lemma}
\begin{proof}
For the first part, $\bbT^S(M)$ is a submodule of $\End_\Lambda(M)$, which is a finite $\Lambda$-module. Since $\Lambda$ is a complete local ring, there is a decomposition $\bbT^S(M) = \prod_\ffrm \bbT^S(M)_\ffrm$ over the finitely many maximal ideals of $\bbT^S(M)$. If $\ffrm \subset \bbT^S(M)$ is one of these ideals, then (since $\Lambda \to \bbT^S(M)$ is finite) its pullback to $\Lambda$ equals $\ffrm_\Lambda$, so $\bbT^S(M)/\ffrm$ is a finite $\Lambda/\ffrm_\Lambda = k$-module, hence a finite extension of $k$. 

The second part is easy, using that $M$ is a finite $\bbT^{S, \text{univ}}$-module; see \cite[Corollary 2.7]{Eis95}.
\end{proof}
The most important algebra for us is the big ordinary Hecke algebra $\bbT^{S}_\text{ord}(U) = \bbT^S(F^\bullet_\infty)$.
The kernel $J$ of the homomorphism $\bbT^{S}_\text{ord}(U) \to \bbT^S(H^\ast_\text{ord}(U))$ is nilpotent, by Lemma  \ref{lem_nilpotents_in_derived_category}, and since $F^\bullet_\infty$ is concentrated in degrees $[0, d]$. 
\begin{lemma}\label{lem_maximal_ideal_in_support_of_cohomology}
Every maximal ideal of $\bbT^{S}_\text{ord}(U)$ appears in the support of $H^\ast_\text{ord}(U)$ and $H_\ast(X_U, k)_\text{ord}$.
\end{lemma}
\begin{proof}
Let $\ffrm$ be a maximal ideal of $\bbT^{S}_\text{ord}(U)$. Then $\ffrm$ contains the nilradical of $\bbT^S_\text{ord}(U)$, so appears in the support of $H^\ast_\text{ord}(U)$, by the above remarks. There is a $\bbT^S_\text{ord}(U)$-equivariant spectral sequence
\[ E_2^{p, q} = \mathrm{Tor}^\Lambda_{-p}(H^q_\text{ord}(U), k) \Rightarrow H_{d-(p + q)}(X_U, k)_\text{ord}. \]
Let $m$ be maximal such that $H^m_\text{ord}(U)_\ffrm \neq 0$. Since $\bbT^S_\text{ord}(U)$ is a finite $\Lambda$-algebra, we have $H^m_\text{ord}(U)_\ffrm \otimes_\Lambda k \neq 0$, and $E_2^{0, m} = E_\infty^{0, m} = H^m_\text{ord}(U)_\ffrm \otimes_\Lambda k$ is a non-zero subquotient of $H_{d-m}(X_U, k)_{\text{ord}, \ffrm}$. This implies that $H_{d-m}(X_U, k)_{\ffrm} \neq 0$, and completes the proof of the lemma.
\end{proof}
In order to go further, we need to assume the existence of Galois representations associated to the ordinary completed cohomology groups constructed so far.
\begin{conjecture}\label{conj_existence_of_galois_combined} Let $\ffrm \subset \bbT^{S}_\text{ord}(U)$ be a maximal ideal with residue field $k$.
\begin{enumerate}
\item There exists a continuous semi-simple representation $\overline{\rho}_\ffrm : G_{F, S} \to \GL_n(\bbT^{S}_\text{ord}(U)/\ffrm)$ satisfying the following condition: for any finite place $v \not\in S$ of $F$, $\overline{\rho}_\ffrm(\Frob_v)$ has characteristic polynomial
\[ X^n - T_v^1 X^{n-1} + \dots + (-1)^{i} q_v^{i(i-1)/2} T_v^i X^{n-i} + \dots + (-1)^n q_v^{n(n-1)/2} T_v^n \in (\bbT^{S}_\text{ord}(U)/\ffrm)[X]. \]
If $\overline{\rho}_\ffrm$ is absolutely reducible, we say that the maximal ideal $\ffrm$ is Eisenstein; otherwise, we say that $\ffrm$ is non-Eisenstein. (By the Chebotarev density theorem, the representation $\overline{\rho}_\ffrm$ is uniquely determined up to isomorphism, if it exists.)
\item Suppose that $\ffrm$ is non-Eisenstein. Then there exists a lifting of $\overline{\rho}_\ffrm$ to a continuous homomorphism $\rho_\ffrm : G_{F, S} \to \GL_n(\bbT^{S}_\text{ord}(U)_\ffrm)$ satisfying the following condition: for any finite place $v \not\in S$ of $F$, $\rho_\ffrm(\Frob_v)$ has characteristic polynomial
\[ X^n - T_v^1 X^{n-1} + \dots + (-1)^{i} q_v^{i(i-1)/2} T_v^i X^{n-i} + \dots + (-1)^n q_v^{n(n-1)/2} T_v^n  \in \bbT^{S}_\text{ord}(U)[X]. \]
\end{enumerate}
\end{conjecture}
\begin{remark}
\begin{enumerate}
\item If $F$ is an imaginary CM or totally real field, then the first part of the conjecture is known, thanks to work of Scholze \cite{Sch13}. The second part of the conjecture is also known, thanks to the work of Scholze, with the proviso that one obtains at present only a lifting of $\overline{\rho}_\ffrm$ valued in $\bbT^{S}_\text{ord}(U)/I$, for some nilpotent ideal $I \subset \bbT^{S}_\text{ord}(U)$ of bounded exponent. (Thus one expects to be able to show that $I = 0$.)
\item It is clear that one can state the above conjecture with reference only to a given field $F$, open compact subgroup $U \subset \GL_n(\bbA_F^\infty)$, and maximal ideal $\ffrm \subset \bbT^{S}_\text{ord}(U)$. In the rest of this paper, when we ask the reader to (for example) ``assume that Conjecture \ref{conj_existence_of_galois_combined} holds for $F$'', we refer to the above statements for all possible valid choices of $U$ and $\ffrm$ (i.e.\ satisfying the assumptions of this section so far).
\end{enumerate}
\end{remark}
We now assume Conjecture \ref{conj_existence_of_galois_combined} for the field $F$ for the rest of \S \ref{sec_ordinary_completed_cohomology}. 
\begin{lemma}\label{lem_localization_independent_of_finite_set}
Let $T$ be a finite set of finite places containing $S$, let $\ffrm \subset \bbT^S_\text{ord}(U)$ be a maximal ideal, and let $\frn = \ffrm \cap \bbT^T_\text{ord}(U)$. Then the canonical inclusions
\[ H_\ast(X_U, k)_\frn \subset H_\ast(X_U, k)_\ffrm \]
and
\[ H^\ast_\text{ord}(U)_\frn \subset H^\ast_\text{ord}(U)_\ffrm \]
are isomorphisms.
\end{lemma}
\begin{proof}
Because $\bbT^S_\text{ord}(U)$ is a finite $\Lambda$-algebra, we have a decomposition $\bbT^S_\text{ord}(U) = \prod_{\ffrm'} \bbT^S_\text{ord}(U)_{\ffrm'}$, the product being over the finitely many maximal ideals of $\bbT^S_\text{ord}(U)$. To prove the lemma, it's enough to show that $\bbT^S_\text{ord}(U)$ has a unique maximal ideal lying above the ideal $\frn$ of $\bbT^T_\text{ord}(U)$. These ideals are in bijection with the maximal ideals of $\bbT^S_\text{ord}(U)$ which are in the support of  the finite-dimensional $k$-vector space $H_\ast(X_U, k)_\frn$. If $\ffrm'$ is any such ideal, then we have $H_\ast(X_U, k)_\frn[\ffrm'] \neq 0$, and there exists a continuous representation $\overline{\rho}_{\ffrm'} : G_{F, S} \to \GL_n(\bbT^S_\text{ord}(U)/\ffrm')$ such that for all $v \not\in S$, we have
\[ \tr \overline{\rho}_{\ffrm'}(\Frob_v) =  T_v^1 \text{ mod } \ffrm'. \]
On the other hand, we have
\[ \tr \overline{\rho}_\frn(\Frob_v) =  T_v^1 \text{ mod }\frn = \tr \overline{\rho}_{\ffrm'}(\Frob_v) \]
for all $v \not\in T$. By the Chebotarev density theorem, we obtain the relation $\overline{\rho}_\ffrm \cong \overline{\rho}_\frn \cong \overline{\rho}_{\ffrm'},$ which then implies $\ffrm = \ffrm'$.
\end{proof}
If $\ffrm \subset \bbT^S_\text{ord}(U)$ is a maximal ideal, then (cf. \S \ref{sec_ordinary_parts}) it determines an idempotent $e_\ffrm \in \End_{\mathbf{D}(\Lambda)}(F^\bullet_\infty)$, and we write $F^\bullet_{\ffrm}$ for a minimal complex representing the direct factor $e_\ffrm F^\bullet_\infty$ in $\mathbf{D}(\Lambda)$, which is well-defined up to quasi-isomorphism. Then there is a canonical identification
\[ \bbT^S_\text{ord}(U)_\ffrm = e_\ffrm \bbT^S_\text{ord}(U) \cong \bbT^S_\text{ord}(F^\bullet_\ffrm). \]
\begin{corollary}\label{cor_localization_independent_of_finite_set}
Let $T$ be a finite set of finite places containing $S$, let $\ffrm \subset \bbT^S_\text{ord}(U)$ be a maximal ideal, and let $\frn = \ffrm \cap \bbT^T_\text{ord}(U)$. Then $e_\ffrm = e_\frn$ in $\End_{\mathbf{D}(\Lambda)}(F^\bullet_\infty)$, and we have $F^\bullet_\ffrm = F^\bullet_\frn$ in $\mathbf{D}(\Lambda)$.
\end{corollary}
\begin{proof}
It is enough to show that $e_\ffrm = e_\frn$. However, Lemma \ref{lem_localization_independent_of_finite_set} shows that the inclusion $\bbT^T_\text{ord}(U) \subset \bbT^S_\text{ord}(U)$ induces a bijection on maximal ideals, and this implies the result.
\end{proof}
\begin{lemma}\label{lem_determinant_of_galois}
Let $\ffrm \subset \bbT^{S}_\text{ord}(U)$ be a non-Eisenstein maximal ideal with residue field $k$. The character 
\[ \det \rho_\ffrm : G_{F, S} \to \bbT^{S}_\text{ord}(U)_\ffrm^\times \]
equals the character
\[ \epsilon^{n(1-n)/2} \eta_\ffrm : G_{F, S} \to \bbT^{S}_\text{ord}(U)_\ffrm^\times, \]
where $\eta_\ffrm : G_{F, S}^\text{ab} \to \bbT^{S}_\text{ord}(U)_\ffrm^\times$ is a finite order character uniquely characterized by the formula $\eta_\ffrm(\Frob_v) = T_v^n$ for all $v \not\in S$.
\end{lemma}
\begin{proof} Let $c \geq 1$. The action of $\bbA_F^\times$ on $X_{U(c, c)}$ factors through the quotient $F^\times \backslash \bbA_F^\times / (U(c, c) \cap Z(\bbA_F) \cdot Z_\infty)$, a finite group which is independent of $c$ (since $U(c, c) \cap Z(\bbA_F)$ is independent of $c$). This action is compatible as $c$ varies, so we can glue to obtain a character $\zeta : F^\times \backslash \bbA_F^\times / (U(1, 1) \cap Z(\bbA_F) \cdot Z_\infty) \to \Aut_{\mathbf{D}(\Lambda)}(F^\bullet_\infty)$. 

We claim that that $\zeta$ in fact takes image in $\bbT^S_\text{ord}(U)^\times \subset \Aut_{\mathbf{D}(\Lambda)}(F^\bullet_\infty)$. This is clear: every element of the finite group $F^\times \backslash \bbA_F^\times / (U(1, 1) \cap Z(\bbA_F) \cdot Z_\infty)$ can be represented by an element of the form $\varpi_v \in F_v^\times \subset \bbA_F^\times$, $v \not\in S$, and we have $\zeta(\varpi_v) = T_v^n$.

We define $\eta_\ffrm = \zeta \circ \Art_F^{-1} : G_{F, S}^\text{ab} \to \bbT^S_\text{ord}(U)_\ffrm^\times$. We see immediately that $\det \rho_\ffrm(\Frob_v) = (\epsilon^{n(1-n)/2}\eta_\ffrm)(\Frob_v)$ for all $v \not\in S$. The Chebotarev density theorem then implies the equality $\det \rho_\ffrm = \epsilon^{n(1-n)/2} \eta_\ffrm$.
\end{proof}
We end this section with a result about cohomology after localization at a non-Eisenstein maximal ideal. The proof of this result only uses the first part of Conjecture \ref{conj_existence_of_galois_combined}, which is known unconditionally when $F$ is an imaginary CM or totally real field, by Scholze's results; it thus holds unconditionally when $F$ has this form.
\begin{theorem}\label{thm_only_cuspidal_cohomology_survives_localization}
Let $\ffrm \subset \bbT^S_\text{ord}(U)$ be a non-Eisenstein maximal ideal with residue field $k$, and let $\bl \in (\bbZ_+^n)^{\Hom(F, E)}$ be a regular dominant weight \(i.e. for each $\tau \in \Hom(F, E)$, $\lambda_\tau$ lies on no root hyperplane of $\GL_n$\). Let $\wp_\bl \subset \Lambda$ be the corresponding ideal.  Then the groups 
\begin{equation}\label{eqn_computation_of_specialized_cohomology} H^i(F^\bullet_\ffrm \otimes_\Lambda \Lambda/\wp_\bl)[1/p] \cong \Hom_\cO(H^{d-i}(X_{U(1, 1)}, M_{-w_0 \bl})_{\ffrm}, E). 
\end{equation}
can be non-zero only if $i \in [q_0, q_0 + l_0]$.
\end{theorem}
\begin{proof}
We just sketch the proof; see \cite{Tho15} for a detailed proof.
The existence of the isomorphism (\ref{eqn_computation_of_specialized_cohomology}) follows from the definition of $F^\bullet_\ffrm$ and Corollary \ref{cor_ordinary_complex_computes_fixed_weight_cohomology}. It therefore suffices to show that the groups
\[ H^i(X_{U(1, 1)}, M_{-w_0 \bl})_{\ffrm} \otimes_\cO E \]
can be non-zero only for $i \in [q_0, q_0 + l_0]$. Let $\overline{X}_{U(1, 1)}$ denote the Borel--Serre compactification of $X_{U(1, 1)}$, and let $\partial X_{U(1, 1)} = \overline{X}_{U(1, 1)} - X_{U(1, 1)}$. Then the local system $\cL_{-M_{w_0 \bl}}$ extends naturally to a local system on $\overline{X}_{U(1, 1)}$, and we have a long exact sequence relating cohomology with compactly supported cohomology and boundary cohomology:
\[ \dots \to H^i_c(X_{U(1, 1)}, M_{-w_0 \bl})_\ffrm \otimes_\cO E \to H^i(X_{U(1, 1)}, M_{-w_0 \bl})_\ffrm \otimes_\cO E \to H^i(\partial \overline{X}_{U(1, 1)}, M_{-w_0 \bl})_\ffrm \otimes_\cO E \to \dots \]
We will show that the groups $H^i(\partial \overline{X}_{U(1, 1)}, M_{-w_0 \bl})_{\ffrm}$ all vanish. This will imply the theorem: indeed, it then follows that the maps $H^i_c(X_{U(1, 1)}, M_{-w_0 \bl})_{\ffrm} \otimes_\cO E \to H^i(X_{U(1, 1)}, M_{-w_0 \bl})_{\ffrm}$ are all isomorphisms, and the regularity of the highest weight $\bl$ implies that these groups can be calculated in terms of cuspidal cohomology (see \cite[\S 5.3]{Li04}). The statement of the theorem then follows from \cite[Proposition 5.2]{Li04}. 

It remains to show the vanishing of the groups $H^i(\partial \overline{X}_{U(1, 1)}, M_{-w_0 \bl})_\ffrm$. In fact, it's enough to show that the groups $H^i(\partial \overline{X}_{U(1, 1)}, M_{-w_0 \bl} \otimes_\cO k)_{\ffrm}$ are zero, and by Proposition \ref{prop_independence_of_weight}, enough to show that the groups $H^i(\partial \overline{X}_{U(1, 1)}, k)_{\ffrm}$ are zero. However, the space $\partial \overline{X}_{U(1, 1)}$ can be stratified by (nilmanifold covers of) the symmetric spaces associated to the Levi subgroups $\GL_{n_1} \times \dots \times \GL_{n_k}$ of $G$. By assumption, the compactly supported cohomologies of these spaces admit $n$-dimensional group determinants, which are sums of determinants of dimensions $n_1, n_2, \dots, n_k$. In particular, these cohomologies all vanish after localization at the maximal ideal $\ffrm$, which corresponds to an (irreducible) $n$-dimensional determinant. This completes the proof.
\end{proof}
Similarly, the following corollary holds unconditionally (i.e.\ without assuming Conjecture \ref{conj_existence_of_galois_combined}).
\begin{corollary}\label{cor_cohomology_is_torsion_lambda_module}
Let $F$ be an imaginary CM field. Suppose that $\ffrm \subset \bbT^S_\text{ord}(U)$ be a non-Eisenstein maximal ideal with residue field $k$, and let $\bl \in (\bbZ_+^n)^{\Hom(F, E)}$ be a regular dominant weight. Let $c \in \Aut(F)$ denote complex conjugation, and suppose that there exists $\tau \in \Hom(F, E)$ such that $\lambda_{\tau c} \neq -w_0 \lambda_\tau$. Then the groups 
\[  H^i(F^\bullet_\ffrm \otimes_\Lambda \Lambda/\wp_\bl)[1/p] \cong \Hom_\cO(H^{d-i}(X_{U(1, 1)}, M_{-w_0 \bl})_{\ffrm}, E) \]
are all zero. In particular, $H^\ast_\text{ord}(U)_{\ffrm}[1/p]$ is a torsion $\Lambda[1/p]$-module.
\end{corollary} 
\begin{proof}
The proof of Theorem \ref{thm_only_cuspidal_cohomology_survives_localization} shows that the groups $H^i(X_{U(1, 1)}, M_{-w_0 \bl})_{\ffrm} \otimes_\cO E$ can be computed in terms of the $(\Lie G_\infty/Z_\infty \otimes_\bbR \bbC, K_\infty)$-cohomology of cuspidal automorphic representations of $\GL_n(\bbA_F)$. A well-known vanishing theorem shows that these groups will necessarily be 0 unless $\bl$, viewed as an element of $X^\ast(T_\bbC)$, is fixed by the Cartan involution of $G_\bbC$ associated to the maximal compact subgroup $K_\infty \subset G_\infty$ (see \cite[Ch. II, Proposition 6.12]{Bor00}). This is equivalent to asking that for each $\tau \in \Hom(F, E)$, we have $\lambda_{\tau c} = - w_0 \lambda_\tau$.

To show that $H^\ast_\text{ord}(U)_\ffrm[1/p]$ is a torsion $\Lambda[1/p]$-module, we choose any $\bl \in (\bbZ_+^n)^{\Hom(F, E)}$ which is not fixed by the Cartan involution. Then there is an ``independence of weight'' spectral sequence
\[ \mathrm{Tor}_{-p}^\Lambda(H^q_\text{ord}(U)_\ffrm, \Lambda/\wp_\bl)[1/p] \Rightarrow \Hom_\cO(H^{d-(p + q)}(X_{U(1, 1)}, M_{-w_0 \bl})_{\ffrm}, E). \]
 Suppose for the sake of contradiction that there exists $m$ such that $H^m_\text{ord}(U)_\ffrm[1/p] \neq 0$, and let $m$ be maximal with this property. Then the term $H^m_\text{ord}(U)_\ffrm \otimes_\Lambda \Lambda/\wp_\bl[1/p]$ of the above spectral sequence is stable, which contradicts the fact that abutment is trivial, by the first part of the corollary. This completes the proof.
\end{proof}
\subsection{Auxiliary primes}\label{sec_auxiliary_primes}

We now discuss Taylor--Wiles primes. We adopt a slightly different method to the one used in \cite{Clo08} and \cite{Tho12}, that necessitates the assumption of `enormous image'. We again fix a choice of good subgroup $U$ such that for each $v \in S_p$, $U_v = I_v(1, 1)$. Let $S$ be a finite set of finite places of $F$, containing $S_p$, such that for all $v \not\in S$, we have $U_v = \GL_n(\cO_{F_v})$. Let $\ffrm \subset \bbT^{S}_\text{ord}(U)$ be a non-Eisenstein maximal ideal. We assume that $\ffrm$ has residue field $k$, and that $k$ contains the eigenvalues of all elements of the image of $\overline{\rho}_\ffrm$.

Let $w \not\in S$ be a finite place of $F$ such that $q_w \equiv 1 \text{ mod }p$, and $\overline{\rho}_\ffrm(\Frob_w)$ has $n$ distinct eigenvalues. Let $\cK = \GL_n(\cO_{F_w})$ and $\cB = I_w(0, 1)$, the standard Iwahori subgroup of $G = \GL_n(F_w)$. Let $\cB_1$ denote the smallest subgroup of $\cB$ containing $I_w(1, 1)$ such that $\cB/\cB_1$ has $p$-power order. Then there is an isomorphism $\Delta_w = \cB/\cB_1 \cong (k(w)^\times(p))^{n-1}$, given by the same formula as in (\ref{eqn_remove_the_center}). There is an analogue of Proposition \ref{prop_complex_control_theorem} at level $U^w \cB_1$ which is proved in exactly the same way, using Lemma \ref{lem_existence_of_perfect_complex}; this gives rise to a minimal complex $F^\bullet_{\infty, U^w \cB_1}$ of $\Lambda[\Delta_w]$-modules and a homomorphism $\cH(G^{\infty, p, w}, U^{p, w})^\text{op}[\mathbf{U}_w^1, \dots, \mathbf{U}_w^n] \to \End_{\mathbf{D}(\Lambda[\Delta_w])}(F^\bullet_{\infty, U^w \cB_1})$ with the following properties:
\begin{itemize}
\item There is an isomorphism $F^\bullet_{\infty, U^w \cB_1} \otimes_{\Lambda[\Delta_w]} \Lambda \cong F^\bullet_{\infty, U^w \cB}$ of minimal complex of $\Lambda$-modules, compatible with the two actions of the algebra $\cH(G^{\infty, p, w}, U^{p, w})^\text{op}[\mathbf{U}_w^1, \dots, \mathbf{U}_w^n] $ by endomorphisms in $\mathbf{D}(\Lambda)$.
\item For any $i \geq 0$, there is an isomorphism 
\[ H^i(F^\bullet_{\infty, U^w \cB_1}) \cong \plim_c H_{d - i}(X_{U^w \cB_1(c, c)}, \cO / \varpi^c)_\text{ord} \]
of $\Lambda[\Delta_w]$-modules, again compatible with the action of $\cH(G^{\infty, p, w}, U^{p, w})^\text{op}[\mathbf{U}_w^1, \dots, \mathbf{U}_w^n]$.
\end{itemize}
This complex is characterized up to unique isomorphism in $\mathbf{D}(\Lambda)$ by the existence of additional data as in the statement of Proposition \ref{prop_complex_control_theorem}, which we do not write down explicitly.
\begin{lemma}\label{lem_properties_of_individual_taylor_wiles_prime}
Write $U = U^w \cK$. Then:
\begin{enumerate}
\item For any $\cO$-module $A$, the map $H_\ast(X_{U^w \cB}, A) \to H_\ast(X_{U^w \cK}, A)$ is canonically split as a morphism of $\cH(G^{\infty, w}, U^w)^\text{op}$-modules. Similarly, the morphism $F^\bullet_{\infty, U^w \cB} \to F^\bullet_{\infty, U^w \cK}$ is canonically split in $\mathbf{D}(\Lambda)$. 
\item Let $\bbT^{S \cup \{w \},w}_\text{ord}(U^w\cB)$ be the \(commutative\) $\Lambda$-subalgebra of $\End_{\mathbf{D}(\Lambda)}(F^\bullet_{\infty, U^w \cB})$ generated by the Hecke operators $T_v^i$, $i \not\in S \cup \{ w \}$, $\mathbf{U}_v^i$, $v \in S_p$, and $\mathbf{U}_w^i$, $i = 1, \dots, n$. There is a natural inclusion 
\begin{equation}\label{eqn_inclusion_of_hecke_algebras} \bbT^{S \cup \{w \}}_\text{ord}(U^w \cB) \subset \bbT^{S \cup \{w \},w}_\text{ord}(U^w \cB)
\end{equation}
and a natural surjection
\begin{equation}\label{eqn_surjection_of_hecke_algebras} \bbT^{S \cup \{w \}}_\text{ord}(U^w \cB) \to \bbT^{S \cup \{ w \}}_\text{ord}(U^w \cK). 
\end{equation}
We write $\frn$ for the pullback of $\ffrm$ to $\bbT^{S \cup \{ w \}}_\text{ord}(U^w \cK)$, and $\frn'$ for its pre-image in $\bbT^{S \cup \{w \}}_\text{ord}(U^w \cB)$.
\item For each ordering $\alpha_1, \dots, \alpha_n$ of the eigenvalues of $\overline{\rho}_\ffrm(\Frob_w)$, there is a maximal ideal $\ffrm_\alpha$ of $\bbT^{S \cup \{w \},w}_\text{ord}(U^w \cB)$ above $\frn'$, given by the formula $\ffrm_\alpha = (\frn', \mathbf{U}_w^1 - \alpha_1, \dots, \mathbf{U}_w^n - \alpha_1 \cdots \alpha_n)$.
\item Fix a choice of ordering $\alpha_1, \dots, \alpha_n$. Then there is an isomorphism $F^\bullet_{\ffrm_\alpha} \to F^\bullet_\ffrm$ in $\mathbf{D}(\Lambda)$, equivariant for the action of $\bbT^{S \cup \{ w \}, \text{univ}}$. In particular, there are canonical isomorphisms of $\Lambda$-algebras
\[ \bbT^{S \cup \{ w \}}_\text{ord}(U)_{\frn} \cong \bbT^{S \cup \{ w \}}(F^\bullet_\ffrm) \cong \bbT^{S \cup \{ w \}}(F^\bullet_{\ffrm_\alpha}). \]
\item Let $\bbT^{S \cup \{w \},w}_\text{ord}(U^w\cB_1)$ be the \(commutative\) $\Lambda[\Delta_w]$-subalgebra of $\End_{\mathbf{D}(\Lambda[\Delta_w])}(F^\bullet_{\infty, U^w \cB_1})$ generated by the Hecke operators $T_v^i$, $v \not\in S \cup \{ v \}$, $\mathbf{U}_v^i$, $v \in S_p$, and $\mathbf{U}_w^i$, $i = 1, \dots, n$. Then there is an isomorphism
\[ F^\bullet_{\infty,U^w \cB_1} \otimes_{\Lambda[\Delta_w]} \Lambda \cong F^\bullet_{\infty,U^w \cB}, \]
and a corresponding surjection
\[ \bbT^{S \cup \{w \},w}_\text{ord}(U^w\cB_1) \to \bbT^{S \cup \{w \},w}_\text{ord}(U^w\cB). \]
\item Let $\ffrm_{\alpha, 1}$ denote the pullback of $\ffrm_\alpha$ to a maximal ideal of $\bbT^{S \cup \{w \},w}_\text{ord}(U^w\cB_1)$. Then there is a $\bbT^{S \cup \{ w \}, \text{univ}}$-equivariant isomorphism
\[ F^\bullet_{\ffrm_{\alpha,1}} \otimes_{\Lambda[\Delta_w]} \Lambda \cong F^\bullet_{\ffrm_\alpha}. \]
In particular, if $\bbT^{S \cup \{ w \}}(F^\bullet_{\ffrm_{\alpha,1}})$ denotes the $\Lambda[\Delta_w]$-subalgebra of $\End_{\mathbf{D}(\Lambda[\Delta_w])}(F^\bullet_{\ffrm_{\alpha, 1}})$ generated by the image of $\bbT^{S \cup \{ w \}, \text{univ}}$, then there is a canonical surjection of $\Lambda[\Delta_w]$-algebras
\[ \bbT^{S \cup \{ w \}}(F^\bullet_{\ffrm_{\alpha,1}}) \to \bbT^{S \cup \{w \}}(F^\bullet_{\ffrm_\alpha}) \cong \bbT^{S \cup \{ w \}}(F^\bullet_\ffrm). \]
\end{enumerate}
\end{lemma}
\begin{proof}
For the first part, this map on homology groups is induced by the inclusion of complexes $C_{\bbA, \bullet}(U^w \cB, A) \to C_{\bbA, \bullet}(U^w \cK, A)$, split by the trace map, since the index $[\cK:\cB]$ is prime to $p$; see Lemma \ref{lem_splitting_of_inclusion_map_in_taylor_wiles_set_up}. The splitting of the map $F^\bullet_{\infty, U^w \cB} \to F^\bullet_{\infty, U^w \cK}$ is glued from the corresponding maps $C_{\bbA, \bullet}(U^w \cK(c, c), \cO/\varpi^c) \to C_{\bbA, \bullet}(U^w \cB(c, c), \cO/\varpi^c)$. The second part of the lemma follows immediately from the first.

For the third part, we find it most convenient to argue using the cohomology groups $H^\ast(X_U, k)$ (dual to the groups $H_\ast(X_U, k)$). We observe that
\begin{equation}\label{eqn_equality_of_cohomology_groups} [\cK]H^\ast(X_{U^w \cB}, k)_{\text{ord}}[\frn'] = H^\ast(X_{U^w \cK}, k)_\text{ord}[\frn] = H^\ast(X_{U^w \cK}, k)_\text{ord}[\ffrm], 
\end{equation}
by the first part of the lemma and Lemma \ref{lem_localization_independent_of_finite_set}. In the notation of \S \ref{sec_q_adic_hecke_algebras}, we have $\mathbf{U}_w^i = t_1 \cdots t_i$, at least up to powers of  $q^{1/2}$, which we ignore. It follows from (\ref{eqn_equality_of_cohomology_groups}) and the discussion preceding Lemma \ref{lem_unramified_hecke_eigenvalues_determine_eigenspaces} that $e_i(t_1, \dots, t_n)$ acts on $[\cK]H^\ast(X_{U^w \cB}, k)_{\text{ord}}[\frn']$ by $e_i(\alpha_1, \dots, \alpha_n)$. The third part of the lemma now follows from Lemma \ref{lem_unramified_hecke_eigenvalues_determine_eigenspaces}.

The map of the fourth part is the composite
\[ F^\bullet_{\ffrm_\alpha} \to F^\bullet_{\frn'} \to F^\bullet_\frn \cong F^\bullet_\ffrm. \]
To show it is an isomorphism, it is enough to show that the induced map
\[ H_\ast(X_{U^w \cB}, k)_{\ffrm_\alpha} = H_\ast(F^\bullet_{\ffrm_\alpha} \otimes_\Lambda k) \to H_\ast(F^\bullet_{\frn} \otimes_\Lambda k) = H_\ast(X_{U^w \cK}, k)_{\frn} \]
is an isomorphism. This is true, by Lemma \ref{lem_hecke_module_with_regular_infinitesimal_character}. The fifth and sixth parts then follow easily.
\end{proof}
We also give a version of Lemma \ref{lem_properties_of_individual_taylor_wiles_prime} in the presence of a Taylor--Wiles datum, i.e.\ a finite set $Q$ of finite places of $F$ satisfying the following conditions:
\begin{itemize}
\item $Q \cap S = \emptyset$.
\item For each $v \in Q$, $q_v \equiv 1 \text{ mod }p$ and $\overline{\rho}_\ffrm(\Frob_v)$ has $n$ distinct eigenvalues $\gamma_{v, 1}, \dots, \gamma_{v, n} \in k$.
\end{itemize}
We define open compact subgroups subgroups $U_0(Q) = \prod_v U_0(Q)_v$ and $U_1(Q) = \prod_v U_1(Q)_v$, where:
\begin{itemize}
\item If $v \not\in Q$, then $U_0(Q)_v = U_1(Q)_v = U_v$.
\item If $v \in Q$, then $U_0(Q)_v = I_v(0, 1)$ and $U_1(Q)_v$ is the pre-image in $I_v(0, 1)$ of the maximal $p$-power quotient of the group $I_v(0, 1)/I_v(1, 1)$.
\end{itemize}
Let $\Delta_Q = \prod_{v \in Q} U_0(Q)_v/U_1(Q)_v$, so that $\Delta_Q \cong \prod_{v \in Q} k(v)^\times(p)^{n-1}$. Generalizing Proposition \ref{prop_complex_control_theorem} once more, there is a minimal complex $F^\bullet_{\infty, U_1(Q)}$ of $\Lambda[\Delta_Q]$-modules and a homomorphism $\cH(G^{\infty, p, Q}, U^{p, Q})^\text{op}[\{\mathbf{U}_v^i\}_{v \in Q}^{i = 1, \dots, n}] \to \End_{\mathbf{D}(\Lambda[\Delta_Q])}(F^\bullet_{\infty, U_1(Q)})$ with the following properties:
\begin{itemize}
\item There is an isomorphism $F^\bullet_{\infty, U_1(Q)} \otimes_{\Lambda[\Delta_Q]} \Lambda \cong F^\bullet_{\infty, U_0(Q)}$ of minimal complexes of $\Lambda$-modules, compatible with the two actions of the algebra $\cH(G^{\infty, p, Q}, U^{p, Q})^\text{op}[\{\mathbf{U}_v^i\}_{v \in Q}^{i = 1, \dots, n}]$ by endomorphisms in $\mathbf{D}(\Lambda)$.
\item For any $i \geq 0$, there is an isomorphism 
\[ H^i(F^\bullet_{\infty, U_1(Q)}) \cong \plim_c H_{d - i}(X_{U_1(Q)(c, c)}, \cO / \varpi^c)_\text{ord} \]
of $\Lambda[\Delta_Q]$-modules, again compatible with the action of $\cH(G^{\infty, p, Q}, U^{p, Q})^\text{op}[\{\mathbf{U}_v^i\}_{v \in Q}^{i = 1, \dots, n}]$.
\end{itemize}
Again, this complex is characterized by the existence of additional data as in the statement of Proposition \ref{prop_complex_control_theorem}. In this context, we have the following result, which generalizes the previous lemma.
\begin{proposition}\label{prop_hecke_algebras_at_auxiliary_taylor_wiles_level} With notation as above, let $\bbT^{S \cup Q, Q}_\text{ord}(U_0(Q))$ denote the $\Lambda$-subalgebra of $\End_{\mathbf{D}(\Lambda)}(F^\bullet_{\infty, U_0(Q)})$ generated by the unramified Hecke operators $T_v^i$, $i = 1, \dots, n$, $v \not\in S \cup Q$, together with the operators $\mathbf{U}_v^i$, $i = 1, \dots, n$, $v \in S_p \cup Q$. Let $\bbT^{S \cup Q, Q}_\text{ord}(U_1(Q))$ denote the $\Lambda[\Delta_Q]$-subalgebra of $\End_{\mathbf{D}(\Lambda[\Delta_Q])}(F^\bullet_{\infty, U_1(Q)})$ generated by the unramified Hecke operators $T_v^i$, $i = 1, \dots, n$, $v \not\in S \cup Q$, together with the operators $\mathbf{U}_v^i$, $i = 1, \dots, n$, $v \in S_p \cup Q$. Then:
\begin{enumerate}
\item There are natural inclusions
\[ \bbT^{S \cup Q}_{\text{ord}}(U) \subset \bbT^S_\text{ord}(U) \]
and
\[ \bbT^{S \cup Q}_\text{ord}(U_0(Q)) \subset \bbT^{S \cup Q, Q}_\text{ord}(U_0(Q)) \]
and natural surjections
\[ \bbT^{S \cup Q}_\text{ord}(U_0(Q)) \to \bbT^{S \cup Q}_\text{ord}(U) \]
and
\[ \bbT^{S \cup Q}_\text{ord}(U_1(Q)) \to \bbT^{S \cup Q}_\text{ord}(U_0(Q)) \]
and
\[ \bbT^{S \cup Q,Q}_\text{ord}(U_1(Q)) \to \bbT^{S \cup Q,Q}_\text{ord}(U_0(Q)). \]
\item Let $\frn = \ffrm \cap \bbT^{S \cup Q}_\text{ord}(U)$, and let $\frn' \subset \bbT^{S \cup Q}_\text{ord}(U_0(Q))$ denote the pullback of $\frn$, and let $\ffrm_{Q, 0} \subset  \bbT^{S \cup Q, Q}_\text{ord}(U_0(Q))$ denote the ideal generated by $\frn'$ and the elements $\mathbf{U}_v^i - \prod_{j=1}^i \gamma_{v, i}$, $i = 1, \dots, n$ and $v \in Q$. Then $\ffrm_{Q, 0}$ is a maximal ideal. We write $\ffrm_{Q, 1} \subset  \bbT^{S \cup Q, Q}_\text{ord}(U_1(Q))$ for its pullback.
\item There are $\bbT^{S \cup Q, \text{univ}}$-equivariant isomorphisms
\[ F^\bullet_\ffrm \cong F^\bullet_\frn \]
and
\[ F^\bullet_\frn \cong F^\bullet_{\ffrm_{Q, 0}} \]
and
\[ F^\bullet_{\ffrm_{Q, 1}} \otimes_{\Lambda[\Delta_Q]} \Lambda \cong F^\bullet_{\ffrm_{Q, 0}}. \]
\(For clarity, we remark that the complexes in the first isomorphism are direct factors of $F^\bullet_{\infty, U}$; that $F^\bullet_{\ffrm_{Q, 0}}$ is a direct factor of $F^\bullet_{\infty,U_0(Q)}$; and that $F^\bullet_{\ffrm_{Q, 1}}$ is a direct factor of $F^\bullet_{\infty,U_1(Q)}$.\) Consequently, if $\bbT^{S \cup Q}(F^\bullet_{\ffrm_{Q, 1}})$ denotes the $\Lambda[\Delta_Q]$-subalgebra of $\End_{\mathbf{D}(\Lambda[\Delta_Q])}(F^\bullet_{\ffrm_{Q, 1}})$ generated by the image of $\bbT^{S \cup Q, \text{univ}}$, then there are maps of Hecke algebras
\[ \bbT^{S \cup Q}(F^\bullet_{\ffrm_{Q, 1}}) \twoheadrightarrow \bbT^{S \cup Q}(F^\bullet_\ffrm) \subset \bbT^S(F^\bullet_\ffrm) \cong \bbT^S_\text{ord}(U)_\ffrm. \]
\end{enumerate}
\end{proposition}
\begin{proof}
If $\# Q = 1$, then this is just a reformulation of Lemma \ref{lem_properties_of_individual_taylor_wiles_prime}. In general, the proposition follows by repeating this argument `one prime at a time'.
\end{proof}
We then have the following addendum to Conjecture \ref{conj_existence_of_galois_combined}:
\setcounter{theorem}{17}
\begin{conjecture}[bis]
\begin{enumerate}
\item[3.] Let assumptions be as in Proposition \ref{prop_hecke_algebras_at_auxiliary_taylor_wiles_level}. Then there exists a lifting of $\overline{\rho}_\ffrm$ to a continuous homomorphism $\rho_{\ffrm_{Q}} : G_{F, S} \to \GL_n(\bbT^{S \cup Q}(F^\bullet_{\ffrm_{Q, 1}}))$  satisfying the following condition: for any finite place $v \not\in S \cup Q$ of $F$, $\rho_{\ffrm_Q}(\Frob_v)$ has characteristic polynomial
\[ X^n - T_v^1 X^{n-1} + \dots + (-1)^{i} q_v^{i(i-1)/2} T_v^i X^{n-i} + \dots + (-1)^n q_v^{n(n-1)/2} T_v^n  \in \bbT^{S \cup Q}(F^\bullet_{\ffrm_{Q, 1}})[X]. \]
\end{enumerate}
\end{conjecture}
\setcounter{theorem}{26}
\subsection{The Taylor--Wiles argument}\label{sec_taylor_wiles_argument}

We now specialize the discussion to our case of ultimate interest. We therefore fix a place $a \not\in S_p$ of $F$ which is absolutely unramified and not split in $F(\zeta_p)$. We define an open compact subgroup $U = \prod_v U_v$ of $G^\infty$ as follows:
\begin{itemize}
\item If $v \in S_p$, then $U_v = I_v(1, 1)$.
\item If $v = a$, then $U_v = \ker(\GL_n(\cO_{F_a}) \to \GL_n(k(a)))$ (i.e.\ the principal congruence subgroup).
\item If $v \not\in S_p \cup \{ a \}$, then $U_v = \GL_n(\cO_{F_v})$.
\end{itemize}
Let $S = S_p \cup \{ a \}$. If the residue characteristic of $a$ is sufficiently large, then $U$ is a good subgroup (as the groups $\Gamma_{U, g}$ are neat), and we assume this. We fix a non-Eisenstein maximal ideal $\ffrm \subset \bbT^S_\text{ord}(U)$, and assume that $\overline{\rho}_\ffrm$ satisfies the following conditions:
\begin{itemize}
\item The representation $\overline{\rho}_\ffrm|_{G_{F(\zeta_p)}}$ has enormous image, in the sense of Definition \ref{def_huge_image}.
\item The representation $\overline{\rho}_\ffrm$ is unramified outside $S_p$. 
\item For each $v \in S_p$, $\overline{\rho}_\ffrm|_{G_{F_v}}$ is trivial and $[F_v : \bbQ_p] > 1 + n(n-1)/2$. (We include this assumption as it simplifies the deformation theory at $p$, and allows us to prove Proposition \ref{prop_ordinary_lifting_rings}.)
\item The matrix $\overline{\rho}_\ffrm(\Frob_a)$ is scalar. (It follows that any lift of $\overline{\rho}_\ffrm$ is unramified at $a$.)
\item The residue field of $\ffrm$ is $k$, and contains the eigenvalues of all elements in the image of $\overline{\rho}_\ffrm$. (This can always be arranged by simply enlarging the field $E$ of coefficients.)
\end{itemize} 
Suppose given a Taylor--Wiles datum $(Q; (\gamma_{v, 1}, \dots, \gamma_{v, n})_{v \in Q})$. According to Proposition \ref{prop_hecke_algebras_at_auxiliary_taylor_wiles_level}, we can find a good complex $F^\bullet_{\ffrm_{Q, 1}}$ of $\Lambda[\Delta_Q]$-modules, together with an isomorphism
\begin{equation}\label{eqn_change_level_for_tw_complex} F^\bullet_{\ffrm_{Q, 1}} \otimes_{\Lambda[\Delta_Q]} \Lambda \cong F^\bullet_\ffrm,
\end{equation}
and a homomorphism $\bbT^{S \cup Q, \text{univ}} \to \End_{\mathbf{D}(\Lambda[\Delta_Q])}(F^\bullet_{\ffrm_{Q, 1}})$ with respect to which the isomorphism (\ref{eqn_change_level_for_tw_complex}) is equivariant. We define $\bbT_{Q, 1}$ to be the $\Lambda[\Delta_Q]$-subalgebra of $\End_{\mathbf{D}(\Lambda[\Delta_Q])}(F^\bullet_{\ffrm_{Q, 1}})$ generated by the image of $\bbT^{S \cup Q, \text{univ}}$,  $\bbT_{Q, 0} = \bbT^{S \cup Q}(F^\bullet_{\ffrm})$, and $\bbT_0 = \bbT^S(F^\bullet_\ffrm)$. There is a diagram of $\Lambda[\Delta_Q]$-algebras:
\[ \xymatrix@1{ \bbT_{Q, 1} \ar@{->>}[r] & \bbT_{Q, 0} \ar@{^{ (}->}[r] & \bbT_0, } \]
as well a lifting $\rho_{\ffrm, Q} : G_{F, S \cup Q} \to \GL_n(\bbT_{Q, 1})$ of $\overline{\rho}_\ffrm$, by Conjecture  \ref{conj_existence_of_galois_combined}. (We allow the case $Q = \emptyset$. Then $\bbT_{Q, 1} = \bbT_{Q, 0} = \bbT_0$ and $\rho_{\ffrm, Q} = \rho_\ffrm$.) 

Let $\mu : G_{F, S} \to \cO^\times$ be the unique continuous character such that $\mu \epsilon^{n(n-1)/2}$ equals the Teichm\"uller lift of $\det \overline{\rho}_\ffrm \epsilon^{n(n-1)/2}$. Then the character $\det \rho_{\ffrm, Q} \otimes \mu^{-1}$  has finite $p$-power order, bounded independently of $Q$ (as the group $F^\times \backslash \bbA_F^\times / (U_1(Q) \cap Z(\bbA_F) \cdot Z_\infty)$ has order independent of $Q$; cf. Lemma \ref{lem_determinant_of_galois}), and we write $r_{\ffrm, Q}$ for the unique twist of $\rho_{\ffrm, Q}$ of determinant $\mu$ (which exists, since $p > n$; the twisting character is everywhere unramified, so is trivial if the class group of $F$ is $p$-torsion free).

We can now state a supplement to Conjecture \ref{conj_existence_of_galois_combined}:
\begin{conjecture}\label{conj_local_global_compatibility_at_taylor_wiles_primes}
Define a global deformation problem
\[ \cS_Q = (\overline{\rho}_\ffrm, \mu, \{ \Lambda_v \}_{v \in S} \cup \{ \cO[\Delta_v] \}_{v \in Q}, \{ \cD_v^\triangle \}_{v \in S_p} \cup \{ \cD_v^\text{TW} \}_{v \in Q}, ) \]
where $\Lambda_v = \cO \llbracket \cO_{F_v}^\times(p)^{n-1} \rrbracket$ if $v \in S_p$. \(There is a canonical isomorphism $\Lambda \cong \widehat{\otimes}_{v \in S_p} \Lambda_v$, because of local class field theory; in the notation of \S \ref{sec_galois_deformation_theory}, we therefore have $\Lambda_S = \Lambda$ and $\Lambda_{S \cup Q} = \Lambda[\Delta_Q]$. \) Then the lifting $r_{\ffrm,Q}$ is of type $\cS_Q$.
\end{conjecture}
\begin{remark}
\begin{enumerate}
\item Conjecture \ref{conj_local_global_compatibility_at_taylor_wiles_primes} is a form of local-global compatibility for the representations $\rho_{\ffrm, Q}$ at the places of $S_p \cup Q$. It seems easier to approach for the places of $Q$ than the places of $S_p$. In particular, in the case where $F$ is a CM field, so the representations $\rho_{\ffrm, Q}$ are almost known to exist, the part of Conjecture \ref{conj_local_global_compatibility_at_taylor_wiles_primes} to do with local-global compatibility at $Q$ should follow from the techniques of forthcoming work of Varma \cite{Var14}.
\item Conjecture \ref{conj_local_global_compatibility_at_taylor_wiles_primes} makes sense for a given choice of $F$, $U$, $\ffrm$, and $Q$, satisfying the assumptions of this section. In what follows, when we ask the reader to (for example) ``assume Conjecture \ref{conj_local_global_compatibility_at_taylor_wiles_primes} for the field $F$'', we mean that the reader should assume that the conjecture holds for all possible choices of $U$, $\ffrm$, and $Q$.
\end{enumerate}
\end{remark}
We now \textbf{assume Conjecture \ref{conj_local_global_compatibility_at_taylor_wiles_primes} for the field $F$ for the rest of \S \ref{sec_ordinary_completed_cohomology}}. (We remind the reader that we have already assumed Conjecture \ref{conj_existence_of_galois_combined}.) We obtain a $\Lambda[\Delta_Q]$-algebra homomorphism $R_{\cS_Q} \to \bbT_{Q, 1}$ classifying the lifting $r_\ffrm$. It need not be surjective if the class group of $F$ contains $p$-torsion, but it is in any case finite (since the class group is finite). If $Q = \emptyset$, then we write $\cS_Q = \cS$, and get a homomorphism $R_\cS \to \bbT_0$.

For any such $Q$, we thus obtain a diagram 
\[ \xymatrix@1{ R_{\cS_Q} \ar[r]\ar[d]   & \bbT_{Q, 1} \ar[d] \ar[r] & \End_{\mathbf{D}(\Lambda[\Delta_Q])}(F^\bullet_{\ffrm_{Q, 1}}) \ar[d] \\
R_\cS \ar[r] & \bbT_0  \ar[r] & \End_{\mathbf{D}(\Lambda)}(F^\bullet_{\ffrm}).} \]
\begin{theorem}\label{thm_r_equals_t}
Let assumptions be as above, and suppose further that $\overline{\rho}_\ffrm$ is totally odd. In particular, we \textbf{assume Conjecture \ref{conj_existence_of_galois_combined} and Conjecture \ref{conj_local_global_compatibility_at_taylor_wiles_primes}}. 
\begin{enumerate}
\item Assume the following hypothesis:
\begin{itemize}
\item[\(CG\):] The groups $H_i(X_{U}, k)_{\ffrm}$ are zero for $i \not\in [q_0, q_0 + l_0]$.
\end{itemize}
Assume moreover that for each $v \in S_p$, $F_v$ contains no non-trivial $p^\text{th}$ roots of unity. Then $R_\cS$ acts nearly faithfully on $H^\ast(F^\bullet_\ffrm)$.
\item Let $\bl \in (\bbZ^n_+)^{\Hom(F, E)}$ be a regular dominant weight, and let $\wp_\bl$ be as in Corollary \ref{cor_ordinary_complex_computes_fixed_weight_cohomology}. Suppose that $H^\ast(X_{U}, M_{-w_0 \bl})_{\ffrm} \otimes_\cO E \neq 0$. Then $R_{\cS,(\wp_\bl)}$ acts nearly faithfully on $H^\ast(F^\bullet_{\ffrm, (\wp_\bl)})$.
\end{enumerate}
\end{theorem}
We remind the reader that a ring $R$ is said to act nearly faithfully on a module $M$ is the ideal $\Ann_R(M)$ is nilpotent. By definition, we have $R_{\cS, (\wp_\bl)} = (\Lambda - \wp_\bl)^{-1} R_\cS$.
\begin{corollary}\label{cor_dimension_of_r_bounded_by_t}
Let $\bl \in (\bbZ^n_+)^{\Hom(F, E)}$ be a regular dominant weight, and suppose that $H^\ast(X_{U}, M_{-w_0 \bl})_{\ffrm} \otimes_\cO E \neq 0$.
\begin{enumerate}
\item We have $\dim R_{\cS,(\wp_\bl)} = \dim_{\Lambda_{(\wp_\bl)}} H^\ast(F^\bullet_{\ffrm, (\wp_\bl)})$.
\item Let $x \in \Spec R_\cS[1/p]$ be a closed point lying above the point $\wp_\bl \in \Spec \Lambda[1/p]$.
Then the Galois representation $\rho_x$ corresponding to $x$ is automorphic, and $\dim R_{\cS, (x)} \leq \dim_{\Lambda[1/p]} H^\ast(F^\bullet_{\ffrm})[1/p]$. 
\end{enumerate}
\end{corollary}
\begin{proof}
The first part follows from Theorem \ref{thm_r_equals_t} and Lemma \ref{lem_dimension_theory}. For the second point, the point $x$ gives a closed point of $\Spec R_{\cS,(\wp_\bl)}$ which is contained in the support of $H^\ast(F^\bullet_{\ffrm, (\wp_\bl)})$. It is easy to see that this implies $x$ is in the image of $\Spec \bbT_0$, hence corresponds to an automorphic Galois representation. Since $R_{\cS, (x)}$ is a localization of $R_{\cS,(\wp_\bl)}$, the final assertion follows from the first part of the corollary.
\end{proof}
\begin{proof}[Proof of Theorem \ref{thm_r_equals_t}]
We patch, and then prove both parts of the theorem at the same time. Let $T = S_p$, and let $q = h^1_{\cS, T}(\ad^0 \overline{\rho}_\ffrm(1))$ (see \S \ref{sec_galois_cohomology}). By Lemma \ref{lem_existence_of_taylor_wiles_primes}, and the assumption that $\overline{\rho}_\ffrm$ has enormous image, we can find for each $N \geq 1$ a Taylor--Wiles datum $(Q_N, (\gamma_{v, 1}, \dots, \gamma_{v, n})_{v \in Q_N})$ satisfying the following conditions.
\begin{itemize}
\item For each $v \in Q_N$, we have $q_v \equiv 1 \text{ mod }p^N$.
\item The ring $R^T_{\cS_{Q_N}}$ is a quotient $A_\cS^T$-algebra of $R_\infty = A_\cS^T \llbracket X_1, \dots, X_g \rrbracket$, where 
\[ g = (n-1)q - n(n-1)[F : \bbQ]/2 - l_0 - 1 + \# T \]
(and so $\dim R_\infty = \dim S_\infty - l_0$). We fix a choice of surjection $R_\infty \to R^T_{\cS_{Q_N}}$.
\end{itemize}
We fix a choice of $v_0 \in T$, and define $\cT = \Lambda \llbracket \{ Y^{i, j}_v \}_{1 \leq i, j \leq n, v \in T} \rrbracket / (Y^{1, 1}_{v_0})$. Then the choice of $\rho_{\ffrm, Q}$ as a lifting within its strict equivalence class determines an isomorphism $R_{\cS_{Q_N}}^T \cong R_{\cS_{Q_N}} \widehat{\otimes}_\Lambda \cT$. We define $C^\bullet_0 = F^\bullet_{\ffrm} $. Then $C^\bullet_0$ is a minimal complex of $\Lambda$-modules, concentrated in the range $[0, d]$. 

We define $\Delta_\infty = \bbZ_p^{(n-1)q}$, $S_N = \cT[\Delta_{Q_N}]$, and $S_\infty = \cT \llbracket \Delta_\infty \rrbracket$. We fix for each $N \geq 1$ a surjection $\Delta_\infty \to \Delta_{Q_N}$; this determines a corresponding surjection $S_\infty \to S_N$. We write $\fra \subset S_\infty$ for the kernel of the natural map $S_\infty \to \Lambda$.

If $N \geq 1$, then we define $C^\bullet_N = F^\bullet_{\ffrm_{Q_N, 1}} \widehat{\otimes}_\Lambda \cT$. Then $C^\bullet_N$ is a minimal complex of $S_N$-modules, and there is an isomorphism $C^\bullet_N \otimes_{S_N} \Lambda \cong C^\bullet_0$. We obtain commutative diagrams
\[ \xymatrix{R^T_{\cS_{Q_N}} \ar[r] \ar[d] & \bbT_{Q_N, 1} \widehat{\otimes}_\Lambda \cT \ar[d] \ar@{^{ (}->}[r] & \End_{\mathbf{D}(S_N)}(C^\bullet_N) \ar[d]^{- \otimes_{S_N} \Lambda} \\
R_\cS \ar[r] & \bbT_0 \ar@{^{ (}->}[r] & \End_{\mathbf{D}(\Lambda)}(C^\bullet_0).} \]
We now apply Proposition \ref{prop_main_patching_argument} to obtain the following data:
\begin{itemize}
\item A minimal complex $C^\bullet_\infty$ of $S_\infty$-modules, concentrated in degrees $[0, d]$, together with an isomorphism $C^\bullet_\infty \otimes_{S_\infty} \Lambda \cong C^\bullet_0$.
\item A homomorphism $S_\infty \to R_\infty$ of $\Lambda$-algebras and a map $R_\infty \to \End_{\mathbf{D}(S_\infty)}(C^\bullet_\infty)$ of $S_\infty$-algebras.
\item A surjection $R_\infty \to R_\cS$ of $S_\infty$-algebras making the following a commutative diagram of $S_\infty$-algebras:
\begin{equation} \begin{aligned}\label{eqn_taylor_wiles_diagram} \xymatrix{ R_\infty \ar[d] \ar[r] & \End_{\mathbf{D}(S_\infty)}(C^\bullet_\infty) \ar[d]^{- \otimes_{S_\infty} \Lambda}\\
 R_\cS  \ar[r] & \End_{\mathbf{D}(\Lambda)}(C^\bullet_0).}
\end{aligned}
\end{equation}
\end{itemize}
The groups $H^\ast(C^\bullet_\infty)$ are finite $R_\infty$-modules (since they are finite $S_\infty$-modules, by construction). Applying Lemma \ref{lem_dimension_theory}, we thus have
\[ \dim_{S_\infty} H^\ast(C^\bullet_\infty) = \dim_{R_\infty} H^\ast(C^\bullet_\infty) \leq \dim R_\infty = \dim S_\infty - l_0. \]
This implies the first part of the theorem. Indeed, if (CG) holds then the complex $C^\bullet_0 = F^\bullet_\ffrm$ is concentrated in degrees $[q_0, q_0 + l_0]$ (by the definition of a minimal complex), and since $C^\bullet_\infty \otimes_{S_\infty} \Lambda \cong C^\bullet_0$, the same applies to $C^\bullet_\infty$. By Lemma \ref{lem_dimension_criterion_for_exactness}, we must have $\dim_{S_\infty} H^\ast(C^\bullet_\infty) = \dim S_\infty - l_0$, $H^i(C^\bullet_\infty) \neq 0$ if and only if $i = q_0 + l_0$, and $\depth_{S_\infty} H^{q_0 + l_0}(C^\bullet_\infty) = \dim S_\infty - l_0$. In particular, we get 
\[ \depth_{R_\infty} H^{q_0 + l_0}(C^\bullet_\infty) \geq \depth_{S_\infty} H^{q_0 + l_0}(C^\bullet_\infty) \geq \dim R_\infty, \]
implying that equality holds, and $\Supp_{R_\infty} H^{q_0 + l_0}(C^\bullet_\infty)$ is a union of irreducible components of $\Spec R_\infty$.

Since $F_v$ $(v \in S_p)$ is assumed to contain no non-trivial $p^\text{th}$ roots of unity, $R_\infty$ is a domain (its irreducible components are in bijection with those of $\Spec \Lambda$; see Proposition \ref{prop_ordinary_lifting_rings}). We deduce that $R_\infty$ acts nearly faithfully on $H^{q_0 + l_0}(C^\bullet_\infty)$, and hence $R_\infty/(\fra)$ acts nearly faithfully on $H^{q_0 + l_0}(C^\bullet_\infty)/(\fra) \cong H^{q_0 + l_0}(F^\bullet_\ffrm)$. However, this action factors through the homomorphism $R_\infty \to R_\cS$. This completes the proof of the first part of the theorem.

We now come to the second part of the theorem, which is proved in a similar way. Let $P_\bl \subset S_\infty$ denote the pullback of $\wp_\bl$ along the augmentation homomorphism $S_\infty \to \Lambda$. Localizing the diagram (\ref{eqn_taylor_wiles_diagram}) of $S_\infty$-algebras at the ideal $P_\bl$, we obtain a diagram
\[ \xymatrix{S_{\infty, (P_\bl)} \ar[r] \ar[d] & R_{\infty, (P_\bl)} \ar[d] \ar[r] & \End_{\mathbf{D}(S_{\infty,(P_\bl)})}(C^\bullet_{\infty, (P_\bl)}) \ar[d]\\
\Lambda_{(\wp_\bl)} \ar[r] & R_{\cS, (\wp_\bl)}  \ar[r] & \End_{\mathbf{D}(\Lambda_{(\wp_\bl)})}(C^\bullet_{0, (\wp_\bl)})} \]
We observe that the ring $R_{\infty, (P_\bl)}$ is a domain, as it is non-zero, and a localization of a domain. (The ring $\Lambda$ is not a domain if there exists $v \in S_p$ such that $F_v$ contains non-trivial $p^\text{th}$-roots of unity, but the ring $\Lambda_{(\wp_\bl)}$ is a domain, in fact a regular local ring.) Moreover, we have $\dim R_{\infty, (P_\bl)} = \dim R_\infty - 1$. Indeed, our assumption that $H^\ast(X_{U}, M_{-w_0 \bl})_{\ffrm} \otimes_\cO E \neq 0$ implies, after choosing arbitrarily a Hecke eigenclass, the existence of a homomorphism $R_\cS \to \overline{\bbQ}_p$ such that the composite $\Lambda \to R_\cS \to \overline{\bbQ}_p$ has kernel $\wp_\bl$. The kernel of the composite $R_\infty \to R_\cS \to \overline{\bbQ}_p$ gives rise to a prime ideal of $R_{\infty, (P_\bl)}$ of height $\dim R_\infty - 1$, implying $\dim R_{\infty, (P_\bl)} \geq \dim R_\infty - 1$. On the other hand, $R_\infty$ is a local ring, so we have $\dim R_{\infty, (P_\bl)} \leq \dim R_\infty - 1$, and equality holds.

 By Corollary \ref{cor_ordinary_complex_computes_fixed_weight_cohomology}, we see that there is an isomorphism
\[ H^\ast(C^\bullet_{\infty, (P_\bl)} \otimes_{S_{\infty, (P_\bl)}} S_{\infty, (P_\bl)}/P_\bl) \cong H^\ast(C^\bullet_{0, (\wp_\bl)} \otimes_{\Lambda_{(\wp_\bl)}} \Lambda_{(\wp_\bl)}/\wp_\bl) \cong \Hom(H^{d-\ast}(X_{U}, M_{-w_0 \bl})_\ffrm \otimes_\cO E, E), \]
and these groups are non-zero only in the range $[q_0, q_0 + l_0]$, by Theorem \ref{thm_only_cuspidal_cohomology_survives_localization}. On the other hand, we have by Lemma \ref{lem_dimension_theory}
\[ \dim_{S_{\infty, (P_\bl)}} H^\ast(C^\bullet_{\infty, (P_\bl)}) \leq \dim R_{\infty, (P_\bl)} \leq \dim R_\infty - 1 = \dim S_{\infty, (P_\bl)} - l_0. \]
It follows from Lemma \ref{lem_dimension_criterion_for_exactness} that these inequalities are equalities, and the group $H^i(C^\bullet_{\infty, (P_\bl)})$ is non-zero if and only if $i = q_0 + l_0$. The proof now proceeds much as in the previous case. The ring $R_{\infty, (P_\bl)}$ acts nearly faithfully on $H^{q_0 + l_0}(C^\bullet_{\infty, (P_\bl)})$, because
\[ \depth_{R_{\infty, (P_\bl)}} H^{q_0 + l_0}(C^\bullet_{\infty, (P_\bl)}) = \dim R_{\infty, (P_\bl)} \]
and $\Spec R_{\infty, (P_\bl)}$ is irreducible. Since $\fra \subset P_\bl \subset S_\infty$, we can divide out to see that $R_{\infty, (P_\bl)}/(\fra)$ acts nearly faithfully on the group
\[ H^{q_0 + l_0}(C^\bullet_{\infty, (P_\bl)})/( \fra) \cong H^{q_0 + l_0}(C^\bullet_{0, (\wp_\bl)}).\]
This action factors through $R_{\infty, (P_\bl)} \to R_{\cS, (\wp_\bl)}$, so it follows that the map 
\[ R_{\cS, (\wp_\bl)} \to \End_{\mathbf{D}(\Lambda_{(\wp_\bl)})}(F^\bullet_{\ffrm, (\wp_\bl)}) \]
has nilpotent kernel. This completes the proof.
\end{proof}

\section{Potential automorphy and Leopoldt's conjecture}\label{sec_potential_automorphy_and_leopoldt}

Let $K$ be a totally real field, and let $p$ be a prime. We call the following statement the abelian Leopoldt's conjecture for $K$ and $p$:
\begin{conjecture}[$\mathrm{AL}(K, p)$]\label{conj_abelian_leopoldt} Let $\Delta$ be the Galois group of the maximal abelian pro-$p$ extension of $K$, unramified outside $p$. Then $\dim_{\bbQ_p} \Delta[1/p] = 1$.
\end{conjecture}
On the other hand, for a general number field $F$, we call the following statement the non-abelian Leopoldt's conjecture for $F$, $p$ and $n$:
\begin{conjecture}[$\mathrm{NAL}(F, p, n)$]\label{conj_non-abelian_leopoldt} Let $U \subset \GL_n(\bbA_F^\infty)$ be a good subgroup, and let $\ffrm \subset \bbT^S_\text{ord}(U)$ be a non-Eisenstein maximal ideal \(cf. \S \ref{sec_hecke_algebras}\). Then $\dim_{\Lambda[1/p]} H^\ast_\text{ord}(U)_\ffrm[1/p] = \dim \Lambda[1/p] - l_0$. 
\end{conjecture}
\begin{remark}
\begin{enumerate}
\item When $F$ is an imaginary CM or totally real field, the first part of Conjecture \ref{conj_existence_of_galois_combined} holds unconditionally, by the work of Scholze; in this case, no additional conjectures are required at least to \emph{state} Conjecture \ref{conj_non-abelian_leopoldt}. However, without the first part of Conjecture \ref{conj_existence_of_galois_combined}, the notion of non-Eisenstein maximal ideal is not defined.
\item The abelian Leopoldt conjecture can also be rephrased in terms of the dimension of a completed cohomology group over an Iwasawa algebra (see \cite[\S 4.3.3]{Hil10}). We have avoided this formulation here in order to emphasize the spaces $X_U$, for which there is no contribution from the center.
\item For completeness, we recall what is known about the Leopoldt conjecture. We thank the anonymous referee for these references. The original conjecture of Leopoldt predicts that the closure of the global units in the $p$-adic local units has the largest possible rank, or equivalently that the $p$-adic regulator is non-zero \cite{Leo62}. For a totally real field, this is equivalent to Conjecture \ref{conj_abelian_leopoldt}, by class field theory.

Leopoldt's conjecture is known for abelian extensions of $\bbQ$, and of imaginary quadratic fields, by work of Brumer \cite{Bru67}. On the other hand, Waldschmidt has proved a bound on the so-called Leopoldt defect $\delta$ which is valid for an arbitrary totally real field \cite{Wal81}. Both of these works use the methods of Diophantine approximation.
\end{enumerate}
\end{remark}
We can now state our main theorem, which shows that the above two conjectures are closely related:
\begin{theorem}\label{thm_nal_implies_al}
Let $K$ be a totally real field, Galois over $\bbQ$, let $n = [K : \bbQ]$, and let $p > n$ be a prime such that $K \cap \bbQ(\zeta_p) = \bbQ$. Assume conjecture $\mathrm{NAL}(F, p, n)$ for all imaginary CM extensions $F$ of $\bbQ$. Then conjecture $\mathrm{AL}(K, p)$ holds.
\end{theorem}
The rest of this section is devoted to the proof of Theorem \ref{thm_nal_implies_al}. We will deduce it from Theorem \ref{thm_r_equals_t} using the techniques of potential automorphy. 
\begin{proof}[Proof of Theorem \ref{thm_nal_implies_al}]
Let $K$ and $p$ be as in the statement of the theorem, and let $F/\bbQ$ be an imaginary CM extension, linearly disjoint from $K(\zeta_p)$, and such that every place of $F$ dividing $p$ splits in the extension $F\cdot K/F$. Let $A = F \cdot K$, let $F^+$ denote the maximal totally real subfield of $F$, and let $c \in G_{F^+}$ be a choice of complex conjugation. By Lemma \ref{lem_existence_of_odd_huge_representations}, we can find a coefficient field $E$ and a continuous character $\overline{\chi} : G_A \to k^\times$ such that $\overline{\chi} \overline{\chi}^c = \epsilon^{1-n}$, $\overline{\rho} = \Ind_{G_A}^{G_F} \overline{\chi}$ has enormous image and the pair $(\overline{\rho}, \epsilon^{1-n} \delta_{F/F^+}^n)$ is polarized. 

By \cite[Lemma A.2.5]{Bar14}, we can find an algebraic character $\chi : G_A \to \cO^\times$ lifting $\overline{\chi}$ (and in turn $\rho = \Ind_{G_A}^{G_F} \chi$ lifting $\overline{\rho}$) such that $\chi \chi^c = \epsilon^{1-n}$  and for each embedding $\tau : F \hookrightarrow \overline{\bbQ}_p$, $\mathrm{HT}_\tau(\rho)$ contains $n$ distinct elements, no pair of which differ by 1. Moreover, we can assume that $\rho$ is ordinary and $\det \rho \cdot \epsilon^{n(n-1)/2}$ is of finite order. In our situation, we have for each place $v$ of $F$ dividing $p$:
\[ \rho|_{G_{F_v}} = \oplus_{w|v} \chi_w, \]
and the assumption that $\rho$ is ordinary just means that we can number the places of $A$ above $v$ as $w_1, \dots, w_n$ so that for each embedding $\tau : F_v \hookrightarrow \overline{\bbQ}_p$, the sequence $\mathrm{HT}_\tau(\chi_{w_i})$ is increasing. The assumption that $\det \rho \cdot \epsilon^{n(n-1)/2}$ is of finite order means that for each embedding $\tau : F_v \hookrightarrow \overline{\bbQ}_p$, $\sum_{i=1}^n \mathrm{HT}_\tau(\chi_{w_i}) = n(n-1)/2$. 

By \cite[Theorem 4.5.1]{Bar14}, we can find an imaginary CM extension $B/F$, linearly disjoint over $F$ from $F(\overline{\rho}, \zeta_p$), a regular algebraic, conjugate self-dual cuspidal automorphic representation $\pi$ of $\GL_n(\bbA_{B})$, and an isomorphism $\iota : \overline{\bbQ}_p \cong \bbC$ such that $\rho|_{G_{B}} \cong r_\iota(\pi)$. Moreover, we can assume that $\rho|_{G_B}$ is crystalline at each place dividing $p$, unramified outside $p$, and for each place $w|p$ of $B$, $\overline{\rho}|_{G_{B_w}}$ is trivial and $[B_w : \bbQ_p] > n(n-1)/2 + 1$. Then $\pi$ is $\iota$-ordinary and unramified at the places dividing $p$, while $\overline{\rho}|_{G_B} \cong \Ind_{G_{BK}}^{G_B} \overline{\chi}|_{G_{BK}}$ still has enormous image (and in particular, is still absolutely irreducible, even after restriction to $G_{B(\zeta_p)}$).

Let $\delta$ denote the Leopoldt defect of $K$, and let $M/K$ be the maximal abelian pro-$p$ extension, unramified outside $p$, with $\Gal(M/K)$ $p$-torsion free. Let $N/BK$ be the maximal abelian pro-$p$ extension, unramified outside $p$, with $\Gal(N/BK)$ $p$-torsion free and such that conjugation by $c$ acts as multiplication by $-1$ on $\Gal(N/BK)$. Let $L = M \cdot N$. Then $\Gal(L/BK) \cong \bbZ_p^{1 + \delta + [ B^+ K : \bbQ]} = \bbZ_p^{1 + \delta + n[B^+ : \bbQ]}$, and we want to show that $\delta = 0$.

Let $\mu : G_B \to \cO^\times$ denote unique continuous character such that $\mu \epsilon^{n(n-1)/2}$ equals the Teichm\"uller lift of $\det \overline{\rho}|_{G_B} \epsilon^{n(n-1)/2}$, and consider the global deformation problem
\[ \cS = (\overline{\rho}|_{G_B}, \mu \epsilon^{n(1-n)/2}, S_p, \{ \Lambda_v \}_{v \in S_p}, \{ \cD_v^\text{ord} \}_{v \in S_p}). \]
Then the global deformation ring $R_\cS$ exists, and the representation $\rho|_{G_{B}}$ defines a point $R_\cS \to \cO$. Let $R = \cO \llbracket \Gal(L/BK) \rrbracket$, and let $\Psi_R : G_{BK} \to R^\times$ be the universal deformation of the trivial character. We consider the representation $\rho_R = \Ind_{G_{BK}}^{G_B} \chi \otimes \Psi_R$. Let $\overline{R}$ denote the largest quotient of $R$ over which $\det \rho_{\overline{R}} = \mu$. We now make the following observations:
\begin{itemize}
\item We have $\dim \overline{R}[1/p] \geq \delta + (n-1)[B^+ : \bbQ]$. In fact, every irreducible component of $\overline{R}[1/p]$ has dimension at least $\delta + (n-1)[B^+ : \bbQ]$.
\item The natural map $R_\cS \to \overline{R}$ is surjective. 
\end{itemize}
For the first point, we note that $\det \rho_R$ arises, up to a finite order character, from the compositum of the anticyclotomic $\bbZ_p$-extension of $B$, which has rank $[B^+ : \bbQ]$, with the cyclotomic $\bbZ_p$-extension of $B$, which has rank 1. This implies that $\dim \overline{R} \geq \dim R - [B^+ : \bbQ] - 1$, hence the ring $\overline{R}[1/p]$ has dimension at least $\delta + (n-1)[B^+ :\bbQ]$, provided it is non-zero. However, we have $\overline{R}[1/p] \neq 0$ because of the existence of $\rho|_{G_B}$. For the second point, we note that in the ring $\overline{R}/(\ffrm_{R_\cS})$, the deformation of $\overline{\rho}|_{G_B}$ is trivial, hence the composite character 
\[ G_{BK} \to R^\times \to \overline{R}^\times \to (\overline{R}/(\ffrm_{R_\cS}))^\times \]
must be trivial. By universality, the quotient map $R \to \overline{R}/(\ffrm_{R_\cS})$ factors through $R \to k$. It is surjective, so we deduce that the map $k \to \overline{R}/(\ffrm_{R_\cS})$ is surjective, and hence that $R_\cS \to \overline{R}$ is surjective, as required.

Fix a subgroup $U$ satisfying the conditions at the beginning of \S \ref{sec_taylor_wiles_argument} and a weight $\bl$ such that $\pi$ contributes to $H^\ast(X_{U(1, 1)}, M_{-w_0 \bl}) \otimes_\cO \overline{\bbQ}_p$. We write $\ffrm \subset \bbT_\text{ord}(U)$ for the corresponding maximal ideal. (The subgroup $U$ is uniquely determined, except that we must choose an auxiliary place $a$ of $B$ such that $\overline{\rho}|_{G_{F_a}}$ is unramified and scalar and $a$ is absolutely unramified and not split in $B(\zeta_p)$. Such a place exists because the image of $\ad \overline{\rho}$ cuts out a field disjoint from $B(\zeta_p)$, by construction. The weight $\bl$ is uniquely determined by $\pi$; $\pi$ contributes to cohomology because of assumption that the character $\det \rho \cdot \epsilon^{n(n-1)/2}$ has finite order.) We finally deduce the inequality
\begin{equation}\label{eqn_leopoldt_defect_lower-bounds_R} \dim_{\Lambda[1/p]} H^\ast_\text{ord}(U)_\ffrm[1/p] \geq \dim R_{\cS, (\wp_\bl)} \geq \delta + (n-1)[B^+ : \bbQ]. 
\end{equation}
Indeed, let $x$ denote the closed point of $R_{\cS}[1/p]$ corresponding to $\rho|_{G_B}$. We have shown that
\[ \dim R_{\cS, (\wp_\bl)} \geq \dim R_{\cS, (x)} \geq \dim \overline{R}_{(x)} \geq \delta + (n-1)[B^+ : \bbQ]. \]
We now apply Theorem \ref{thm_r_equals_t} (or rather its corollary), working over the base field $B$, to deduce that $\dim R_{\cS, (\wp_\bl)} \leq \dim_{\Lambda[1/p]} H^\ast_\text{ord}(U)_\ffrm[1/p]$. (This is the point of the argument which uses Conjecture \ref{conj_existence_of_galois_combined} and Conjecture \ref{conj_local_global_compatibility_at_taylor_wiles_primes}.) On the other hand, Conjecture \ref{conj_non-abelian_leopoldt} implies that 
\begin{equation}\label{eqn_non-abelian_leopoldt_defect_upper-bounds_T} \dim_{\Lambda[1/p]} H^\ast_\text{ord}(U)_\ffrm[1/p] = \dim \Lambda[1/p] - l_0 = (n-1)[B^+ : \bbQ]. 
\end{equation}
Comparing (\ref{eqn_leopoldt_defect_lower-bounds_R}) and (\ref{eqn_non-abelian_leopoldt_defect_upper-bounds_T}), we see the only possibility is $\delta = 0$; in other words, Conjecture \ref{conj_abelian_leopoldt} holds for the pair $(K, p)$. This completes the proof.
\end{proof}


\begin{thebibliography}{BLGGT}

\bibitem[BH93]{Bru93}
Winfried Bruns and J{\"u}rgen Herzog.
\newblock {\em Cohen-{M}acaulay rings}, volume~39 of {\em Cambridge Studies in
  Advanced Mathematics}.
\newblock Cambridge University Press, Cambridge, 1993.

\bibitem[BLGGT]{Bar14}
T.~Barnet-Lamb, T.~Gee, D.~Geraghty, and R.~Taylor.
\newblock Potential automorphy and change of weight.
\newblock To appear in Annals of Math.


\bibitem[Bru67]{Bru67}
A.~Brumer.
\newblock On the units of algebraic number fields.
\newblock {\em Mathematika}, 14:121--124, 1967.  


\bibitem[BW00]{Bor00}
A.~Borel and N.~Wallach.
\newblock {\em Continuous cohomology, discrete subgroups, and representations
  of reductive groups}, volume~67 of {\em Mathematical Surveys and Monographs}.
\newblock American Mathematical Society, Providence, RI, second edition, 2000.

\bibitem[CG]{Cal13}
F.~Calegari and D.~Geraghty.
\newblock Modularity lifting beyond the {T}aylor-{W}iles method.
\newblock Preprint.

\bibitem[CHT08]{Clo08}
Laurent Clozel, Michael Harris, and Richard Taylor.
\newblock Automorphy for some {$l$}-adic lifts of automorphic mod {$l$}
  {G}alois representations.
\newblock {\em Publ. Math. Inst. Hautes \'Etudes Sci.}, (108):1--181, 2008.
\newblock With Appendix A, summarizing unpublished work of Russ Mann, and
  Appendix B by Marie-France Vign{\'e}ras.

\bibitem[CLH]{Car14}
A.~Caraiani and Bao~V. Le~Hung.
\newblock On the image of complex conjugation in certain {G}alois
  representations.
\newblock Preprint.

\bibitem[CT14]{Clo13}
Laurent Clozel and Jack~A. Thorne.
\newblock Level-raising and symmetric power functoriality, {I}.
\newblock {\em Compos. Math.}, 150(5):729--748, 2014.

\bibitem[Eis95]{Eis95}
David Eisenbud.
\newblock {\em Commutative algebra}, volume 150 of {\em Graduate Texts in
  Mathematics}.
\newblock Springer-Verlag, New York, 1995.
\newblock With a view toward algebraic geometry.

\bibitem[Ger]{Ger09}
D.~Geraghty.
\newblock Modularity lifting theorems of ordinary {G}alois representations.
\newblock Preprint.

\bibitem[Gou01]{Gou01}
Fernando~Q. Gouv{\^e}a.
\newblock Deformations of {G}alois representations.
\newblock In {\em Arithmetic algebraic geometry ({P}ark {C}ity, {UT}, 1999)},
  volume~9 of {\em IAS/Park City Math. Ser.}, pages 233--406. Amer. Math. Soc.,
  Providence, RI, 2001.
\newblock Appendix 1 by Mark Dickinson, Appendix 2 by Tom Weston and Appendix 3
  by Matthew Emerton.

\bibitem[Hid95]{Hid95}
Haruzo Hida.
\newblock Control theorems of {$p$}-nearly ordinary cohomology groups for
  {${\rm SL}(n)$}.
\newblock {\em Bull. Soc. Math. France}, 123(3):425--475, 1995.

\bibitem[Hid98]{Hid98}
Haruzo Hida.
\newblock Automorphic induction and {L}eopoldt type conjectures for {${\rm
  GL}(n)$}.
\newblock {\em Asian J. Math.}, 2(4):667--710, 1998.
\newblock Mikio Sato: a great Japanese mathematician of the twentieth century.

\bibitem[Hil10]{Hil10}
Richard Hill.
\newblock On {E}merton's {$p$}-adic {B}anach spaces.
\newblock {\em Int. Math. Res. Not. IMRN}, (18):3588--3632, 2010.

\bibitem[HKP10]{Hai10}
Thomas~J. Haines, Robert~E. Kottwitz, and Amritanshu Prasad.
\newblock Iwahori-{H}ecke algebras.
\newblock {\em J. Ramanujan Math. Soc.}, 25(2):113--145, 2010.

\bibitem[Leo62]{Leo62}
Heinrich-Wolfgang Leopoldt. 
\newblock Zur Arithmetik in abelschen Zahlk\"orpern.
\newblock {\em J. Reine Angew. Math.}, 209:54--71,  1962.

\bibitem[LS04]{Li04}
Jian-Shu Li and Joachim Schwermer.
\newblock On the {E}isenstein cohomology of arithmetic groups.
\newblock {\em Duke Math. J.}, 123(1):141--169, 2004.

\bibitem[Maz89]{Maz89}
B.~Mazur.
\newblock Deforming {G}alois representations.
\newblock In {\em Galois groups over {${\bf Q}$} ({B}erkeley, {CA}, 1987)},
  volume~16 of {\em Math. Sci. Res. Inst. Publ.}, pages 385--437. Springer, New
  York, 1989.

\bibitem[Mil06]{Mil06}
J.~S. Milne.
\newblock {\em Arithmetic duality theorems}.
\newblock BookSurge, LLC, Charleston, SC, second edition, 2006.

\bibitem[Mum08]{Mum08}
David Mumford.
\newblock {\em Abelian varieties}, volume~5 of {\em Tata Institute of
  Fundamental Research Studies in Mathematics}.
\newblock Published for the Tata Institute of Fundamental Research, Bombay,
  2008.
\newblock With appendices by C. P. Ramanujam and Yuri Manin, Corrected reprint
  of the second (1974) edition.

\bibitem[Sch]{Sch13}
P.~Scholze.
\newblock On torsion in the cohomology of locally symmetric varieties.
\newblock Preprint.

\bibitem[NT]{Tho15}
J.~Newton. and J.~Thorne.
\newblock Torsion Galois representations over CM fields and Hecke algebras in the derived category. 
\newblock Preprint.

\bibitem[Tho]{Tho14}
J.~Thorne.
\newblock Automorphy lifting for residually reducible $l$-adic {G}alois
  representations.
\newblock To appear in Journal of the AMS.

\bibitem[Tho12]{Tho12}
Jack Thorne.
\newblock On the automorphy of {$l$}-adic {G}alois representations with small
  residual image.
\newblock {\em J. Inst. Math. Jussieu}, 11(4):855--920, 2012.
\newblock With an appendix by Robert Guralnick, Florian Herzig, Richard Taylor
  and Thorne.

\bibitem[Var]{Var14}
I.~Varma.
\newblock Local-global compatibility for {G}alois representations associated to
  regular algebraic cuspidal automorphic representations when $l \neq p$.
\newblock Preprint.

  
\bibitem[Wal81]{Wal81}
M.~Waldschmidt.
\newblock A lower bound for the p-adic rank of the units of an algebraic number field.
\newblock In {\em Topics in classical number theory, Vol. I, II} (Budapest, 1981), volume~34 of {\em Colloq. Math. Soc. J\'a nos Bolyai}, pages 1617--1650. North-Holland, Amsterdam, 1984. 

\bibitem[Wei94]{Wei94}
Charles~A. Weibel.
\newblock {\em An introduction to homological algebra}, volume~38 of {\em
  Cambridge Studies in Advanced Mathematics}.
\newblock Cambridge University Press, Cambridge, 1994.


\end{thebibliography}

\end{document}